\documentclass[10pt]{amsart}

\usepackage{amssymb}
\usepackage{mathdots}

\usepackage{tikz}
\usepackage{tikz-cd}
\usepackage{mathtools}

\usepackage{leftidx}

\usepackage{epsfig}  		
\usepackage{epic,eepic}       

\usepackage[notcite]{showkeys}

\newtheorem{thm}{Theorem}[section]

\newtheorem{lemma}[thm]{Lemma}
\newtheorem{proposition}[thm]{Proposition}
\newtheorem{corollary}[thm]{Corollary}

\theoremstyle{definition}
\newtheorem{definition}[thm]{Definition}

\theoremstyle{remark}

\numberwithin{equation}{section}

\DeclareMathOperator{\Hom}{Hom}

\DeclareMathOperator{\Ind}{Ind}

\DeclareMathOperator{\Gal}{Gal}

\newcommand{\rightarrowdbl}{\rightarrow\mathrel{\mkern-14mu}\rightarrow}

\begin{document}

\title[]{LOCAL LANGLANDS CORRESPONDENCE FOR THE TWISTED EXTERIOR AND SYMMETRIC SQUARE $\epsilon$-FACTORS OF $\textrm{GL}_n$}


\author{Dongming She}
\date{Sept. 2019}
\address{Department of Mathematics, Purdue University, 
West Lafayette, IN 47907}
\email{shed@purdue.edu}




\begin{abstract}
Let $F$ be a non-Archimedean local field. Let $\mathcal{A}_n(F)$ be the set of equivalence classes of irreducible admissible representations of $\textrm{GL}_n(F)$, and $\mathcal{G}_n(F)$ be the set of equivalence classes of n-dimensional Frobenius semisimple Weil-Deligne representations of $W'_F$. The local Langlands correspondence(LLC) establishes the reciprocity maps $\textrm{Rec}_{n,F}: \mathcal{A}_n(F)\longrightarrow \mathcal{G}_n(F)$ , satisfying some nice properties. An important invariant under this correspondence is the L- and $\epsilon$-factors. This is also expected to be true under parallel compositions with a complex analytic representations of $\textrm{GL}_n(\mathbb{C})$. J.W. Cogdell, F. Shahidi, and T.-L. Tsai proved the equality of the symmetric and exterior square L- and $\epsilon$-factors [7] in 2017. But the twisted symmetric and exterior square L- and $\epsilon$-factors are new and very different from the untwisted case. In this paper we will define the twisted symmetric square L- and $\gamma$-factors using $\textrm{GSpin}_{2n+1}$, and establish the equality of the corresponding L- and $\epsilon$-factors. We will first reduce the problem to the analytic stability of their $\gamma$-factors for supercuspidal representations, then prove the supercuspidal stability by establishing general asymptotic expansions of partial Bessel function following the ideas in [7].
\end{abstract}

\maketitle




\section{INTRODUCTION}
The local Langlands Correspondence(LLC) for $\textrm{GL}_n$ has been proved by G. Laumon, M. Rapoport, and U. Stuhler for function fields (1993, [14]), by  G. Henniart (2000, [12]) and also by M. Harris and R. Taylor (2001, [10]), and later by P. Scholze (2010, [15]) using a different approach for p-adic fields.
Let $\rho$ be an n-dimensional Frobenius semisimple representation of the local Weil-Deligne group $W'_F$, and $\pi=\pi(\rho)$ be its corresponding irreducible admissible representation of $\textrm{GL}_n(F)$,  then one expects the equality of their L- and $\epsilon$-factors:
$$\epsilon(s,\rho,\psi)=\epsilon(s,\pi(\rho),\psi),$$
$$L(s,\rho)=L(s,\pi(\rho)),$$ where the local arithmetic $\epsilon$-factor $\epsilon(s,\rho, \psi)$ is defined by P. Deligne in [9], in which he showed that the global $\epsilon$-factors admit a factorization into a product of local ones. Here $L(s,\rho)$ is the local Artin L-factor and $\psi$ is a non-trivial additive character of $F$. The local analytic $\epsilon(s,\pi(\rho),\psi)$ and $L(s,\pi(\rho))$ are defined by Langlands-Shahidi method first for generic representations, then for tempered representations and finally using Langlands classification for all irreducible admissible representations of $\textrm{GL}_n(F)$. If $r$ is a continuous representation of $\textrm{GL}_n(\mathbb{C})$, then one can define the local Artin L- and $\epsilon$-factors $L(s,r\circ\rho,\psi)$ and $\epsilon(s,r\circ \rho, \psi)$. Therefore a natural question is to see if the following equalities hold:
$$L(s,r\circ\rho)=L(s,\pi,r),$$ 
$$\epsilon(s,r\circ\rho,\psi)=\epsilon(s,\pi,r,\psi),$$ as long as the factors on the analytic side are defined. We have a finite list of such factors defined by Langlands-Shahidi method, first for tempered representations, then use Langlands classification and multiplicativity to generalize the definitions to all irreducible admissible representations ([16], [17]). One has the following relationship for analytic $\epsilon$-, $\gamma$-, and L-factors:
$$\epsilon(s,\pi,r,\psi)=\frac{\gamma(s,\pi,r,\psi)L(s,\pi,r)}{L(1-s,\tilde{\pi},r)}.$$
On the arithmetic side, one can naturally define 
$$\gamma(s,r\circ\rho,\psi)=\frac{\epsilon(s,r\circ\rho,\psi)L(1-s, r\circ \rho^{\vee})}{L(s, r\circ\rho)}.$$

So the equalities of $\epsilon$- and L-factors are equivalent to the equalities of $\gamma$- and L-factors.
One method to prove equalities like this was first introduced by J.W. Cogdell, F. Shahidi, and T.-L. Tsai [7] in 2017, for the case where $r=\wedge^2$ and $\textrm{Sym}^2$. The proof uses a globalization method and certain reductions, and relies on two main results called the arithmetic stability and analytic stability of $\gamma$-factors respectively. The former was introduced and proved by P. Deligne in [9], the later for the case $r=\wedge^2$ (and by symmetry also $r=\textrm{Sym}^2$) was proved in [7]. The authors used the group $H=\textrm{GSp}_{2n}$ and its maximal self-associate Levi subgroup $M_H\simeq \textrm{GL}_n\times \textrm{GL}_1$ to construct the analytic factors for $r=\wedge^2$, using the fact that the adjoint representation $r$ of ${^LM}_H$ on $^L\mathfrak{n}_H=\textrm{Lie}(^LN_H)$ decomposes as $r=r_1\oplus r_2$, where $r_1$ is isomorphic to the standard representation of $\textrm{GL}_n(F)$ and $r_2=\wedge^2$. As a consequence the problem was reduced to establishing the stability of Shahidi local coefficients, which can be written as the Mellin transform of certain partial Bessel functions [19] under some conditions. The partial Bessel functions defined on the relevant part of the big Bruhat cells have nice asymptotic behaviors. Their asymptotic expansions can be written as a sum of two parts. The first part depends only on the central character of $\pi(\rho)$, and the second part is a uniformly smooth function on certain torus, which becomes zero after a highly ramified twist. 

In this paper we will define the twisted symmetric and exterior square $\gamma$- and L-factors of $\text{GL}_n(F)$, and prove the following result: 
\begin{thm}Let $F$ be a non-archimedean local field, $\rho$ be an n-dimensional $\Phi$-semisimple Weil-Delinge representation of $W_F'$, $\pi=\pi(\rho)$ be the corresponding irreducible admissible representation of $G=\text{GL}_n(F)$ attached to $\rho$ under the local Langlands correspondence.
Let $\text{Sym}^2$ and $\wedge^2$ denote the symmetric and exterior square representations of $^LG=\textrm{GL}_n(\mathbb{C})$, fix a character $\eta: F^{\times}\rightarrow \mathbb{C}^{\times}.$ Let $\epsilon(s, \pi, \textrm{Sym}^2\otimes\eta,\psi)$ and $\epsilon(s, \pi, \wedge^2\otimes\eta,\psi)$ be the twisted symmetric and exterior square local analytic $\epsilon$-factors, and $\epsilon(s, \textrm{Sym}^2\rho\otimes\eta,\psi)$, $\epsilon(s, \wedge^2\otimes\eta,\psi)$ their corresponding local arithmetic $\epsilon$-factors. 
Then
$$\epsilon(s, \textrm{Sym}^2\rho\otimes\eta,\psi)=\epsilon(s, \pi, \textrm{Sym}^2\otimes\eta,\psi);$$
$$\epsilon(s, \wedge^2\rho\otimes\eta,\psi)=\epsilon(s, \pi, \wedge^2\otimes\eta,\psi);$$

and
$$L(s, \textrm{Sym}^2\rho\otimes\eta)=L(s, \pi, \textrm{Sym}^2\otimes\eta);$$
$$L(s, \wedge^2\rho\otimes\eta)=L(s, \pi, \wedge^2\otimes\eta)).$$
\end{thm}
We will show the equalies of their $\gamma$- and L-factors. 

First, the $\gamma$-factors $\gamma(s, \pi,\textrm{Sym}^2\otimes\eta,\psi)$ and $\gamma(s, \pi,\wedge^2\otimes\eta,\psi)$, once constructed, will have to satisfy the symmetry
$$\gamma(s, (\pi\times\pi)\times \eta,\psi)=\gamma(s,\pi,\wedge^2\otimes \eta,\psi)\gamma(s,\pi,Sym^2\otimes\eta,\psi),$$
$$\gamma(s,(\rho\otimes\rho)\otimes\eta,\psi)=\gamma(s,\wedge^2\rho\otimes\eta,\psi)\gamma(s,Sym^2\rho\otimes\eta,\psi).$$ As the LLC preserves $L$- and $\epsilon$-factors of pairs, and is compatible with twisting by characters, it suffices to prove Theorem 1.1 only for the twisted symmetric square $\gamma$-factors. We will use Langlands-Shahidi method for odd $\textrm{GSpin}$ groups to produce the twisted symmetric square $\gamma$-factors. The reason is that when $n$ is odd, the maximal parabolic subgroups in $\textrm{GSpin}_{2n}$ that produce the twisted exterior square $\gamma$-factors, are not self-associate, although their unipotent radicals have relatively simpler structures. Hence Theorem 6.2 of [19], which we will use to write the local coefficient as the Mellin transform of partial Bessel functions, can not be applied in this situation. 
\section{TWISTED SYMMETRIC SQUARE L- AND $\gamma$-FACTORS}

We will construct the twisted symmetric square $\gamma$- and L-factors of $\textrm{GL}_n$ using the group $H=\textrm{GSpin}_{2n+1}$. It is a reductive group of type $\textrm{B}_n$ with derived group $\textrm{Spin}_{2n+1}$, which is the simply connected double cover of $\textrm{SO}_{2n+1}$. By Proposition 2.1 of [2], the root datum of $H$ can be given as:
$$X=\mathbb{Z}e_0\oplus\mathbb{Z}e_1\oplus\cdots\oplus \mathbb{Z}e_n,$$
$$X^{\vee}=\mathbb{Z}e^*_0\oplus \mathbb{Z}e_1^*\oplus\cdots\oplus\mathbb{Z}e_n^*,$$
$$\Delta=\{\alpha_1=e_1-e_2,\alpha_2=e_2-e_3,\cdots,\alpha_{n-1}=e_{n-1}-e_n,\alpha_n=e_n\}$$
$$\Delta^{\vee}=\{\alpha_1^{\vee}=e_1^*-e_2^*,\alpha_2^{\vee}=e^*_2-e_3^*,\cdots,\alpha^{\vee}_{n-1}=e^*_{n-1}-e_n^*,\alpha_n^{\vee}=2e_n^*-e_0^*\}.$$

Take the self-associate parabolic subgroup $P_H$ of $H$ with Levi decomposition $P_H=M_HN_H$,  where $M_H=M_{\theta}$, $\theta=\Delta-\{\alpha_n\}$. Then $M_H\simeq \textrm{GL}_n\times \textrm{GL}_1$ (Theorem 2.7, [1]). Let $\psi$ be a non-trivial additive character of $F$, and
$(\pi, V)$ be an irreducible $\psi$-generic representation of $\textrm{GL}_n(F)$. Let $\eta:F^{\times}\rightarrow \mathbb{C}^{\times}$ be a character of $F^{\times}$. We lift $\pi$ to a $\psi$-generic representation $\sigma$ of $M_H(F)$, being trivial on the $\textrm{GL}_1$-component. Define a generic representation
 $\sigma_{\eta}:M_H(F)\simeq \textrm{GL}_n(F)\times \textrm{GL}_1(F)\longrightarrow \textrm{GL}(V)$ by
$\sigma_{\eta}(m(g,a))v= \eta^{-1}(a)\pi(g)v.$ 

Denote the L-group of $H$ by $^LH$, similarly we can define $^LM_H$ and $^LN_H$. We have $^LH \simeq GSp_{2n}(\mathbb{C})=\{h\in GL_{2n}(\mathbb{C}): {^th} J h=\phi (h) J \ \ \textrm{for}\ \ \textrm{some} \ \ \phi(h)\in F^\times \}$, where 
$$J=\begin{bmatrix}
\ \ & J' \\
-{^tJ}' & \ \ \\
\end{bmatrix}, J'=\begin{bmatrix}
\ \ & \ \ & \ \ & 1 \\
\ \ & \ \ & -1 \ \ \\
\ \ & \iddots & \ \ & \ \ \\
(-1)^{n-1} & \ \ & \ \ & \ \ \\
\end{bmatrix},$$
and $\phi: H\rightarrow \mathbb{C}^\times$ is the similitude character of $H$. Therefore we have 
$$\leftidx{^L}M_H=\{m=m(g,a_0)=\begin{bmatrix}
g & \ \ \\
\ \ & a_0 J'{^tg^{-1}J'^{-1}}\\
\end{bmatrix}: g\in \textrm{GL}_n(\mathbb{C}), a_0\in \mathbb{C}^{\times})\}$$$$\simeq \textrm{GL}_n(\mathbb{C})\times \textrm{GL}_1(\mathbb{C}).$$

Let $^L\mathfrak{n}_H=\textrm{Lie}(\leftidx{^L}N_H)$. The adjoint action $r: \leftidx{^L}M_H\longrightarrow \textrm{GL}(^L\mathfrak{n}_H)$ is irreducible (Appendix A, ($\textrm{B}_{n,ii}$), [17]).
Then by Langlands-Shahidi method (Theorem 3.1 in [16] or Theorem 8.3.2 in [17]), the local $\gamma$-factor $\gamma(s, \sigma_{\eta}, r,\psi)$ is well-defined. $\sigma_\eta$ is unramified if both $\pi$ and $\eta$ are. Fix a uniformizer $\varpi$ of $F$, then the semisimple conjugacy class $c(\pi)$ attached to $\pi$ is given by $c(\pi)=\textrm{diag}\{\chi_1(\varpi),\cdots,\chi_n(\varpi)\}$, where $\chi_1,\cdots,\chi_n$ are $n$ unramified characters of $F^\times$. Therefore the semisimple conjugacy class attached to $\sigma$ is given by $$c(\sigma)=\textrm{diag}\{\chi_1(\varpi),\cdots,\chi_n(\varpi),\chi_n(\varpi)^{-1},\cdots,\chi_1(\varpi)^{-1}\}.$$ On the other hand, $c(\eta)=\textrm{diag}\{1,\cdots,1,\eta(\varpi)^{-1},\cdots,\eta(\varpi)^{-1}\}$, so
$$c(\sigma_\eta)=c(\sigma)c(\eta)=\textrm{diag}\{\chi_1(\varpi),\cdots,\chi_n(\varpi),\eta(\varpi)^{-1}\chi_n(\varpi)^{-1},\cdots,\eta(\varpi)^{-1}\chi_1(\varpi)^{-1}\}.$$
It follows that $$L(s, \sigma_{\eta}, r)=\det(1-r(c(\sigma_\eta)q_F^{-s}))^{-1}=\prod_{1\leq i\leq j\leq n}(1-(\chi_i\chi_j\eta)(\varpi)q_F^{-s})^{-1}$$
which is what we usually referred as the unramified twisted symmetric square local L-factor for $\textrm{GL}_n$ (section 1, [20]). 

We can use Langlands-Shahidi method to first define the twisted symmetric square L-factor for $\pi$ being tempered, and use Langlands classification and multiplicativity to define for any irreducible admissible representation $\pi$ of $\textrm{GL}_n(F)$ that
$L(s, \pi, \textrm{Sym}^2\otimes\eta)=L(s, \sigma_{\eta},r)$ and $\gamma(s, \pi, \textrm{Sym}^2\otimes \eta,\psi)=\gamma(s, \sigma_{\eta}, r,\psi).$ This is how the general definitions of all Langlands-Shahidi $\gamma$- and $L$-factors are given ([16], [17]). 
\section{STABLE EQUALITY}
Suppose $\rho$ is mapped to $\pi=\pi(\rho)$ under the local Langlands correspondence. The character $\eta: F^\times\longrightarrow \mathbb{C}^\times$can be viewed as a character of the local Weil group $W_F$ by $ W_F\twoheadrightarrow W_F^{ab}\simeq F^{\times}\rightarrow\mathbb{C}^{\times}$ through the local Artin map $\textrm{Art}_F^{-1}: W_F^{ab}\simeq F^{\times}$. We still denote it by $\eta$. On the other hand, $\rho$ and $\eta$ define a homomorphism
$$\rho_{\eta}:W_F\longrightarrow {^LM}_H\simeq \textrm{GL}_n(\mathbb{C})\times \textrm{GL}_1(\mathbb{C})$$ by $\rho_{\eta}(w)=(\rho(w), \eta^{-1}(w)).$ It is easy to see that
$r\circ \rho_{\eta}\simeq \textrm{Sym}^2\rho\otimes \eta.$ 

Now Let $\chi:F^\times\rightarrow \mathbb{C}^\times$ be a continuous character of $F^\times$, viewed as a character of $\textrm{GL}_n(F)$ through the determinant. Similar to $\eta$ we can also view $\chi$ as a character of $W_F$. $\rho$ and $\chi$ determine a homomorphism $$\rho\otimes\chi:W_F\longrightarrow  \textrm{GL}_n(\mathbb{C})$$ by $w\mapsto \chi(w)\rho(w)$. Consequently we also have
$$(\rho\otimes \chi)_{\eta}: W_F\longrightarrow {^LM}_H\simeq \textrm{GL}_n(\mathbb{C})\times \textrm{GL}_1(\mathbb{C})$$ defined by
$(\rho\otimes \chi)_{\eta}(w)=((\rho\otimes \chi)(w),\eta^{-1}(w))=(\chi(w)\rho(w), \eta^{-1}(w))$. We can see that
$r\circ (\rho\otimes \chi)_{\eta}\simeq \textrm{Sym}^2(\rho\otimes \chi)\otimes \eta$. Therefore on the arithmetic side we have $L(s, \textrm{Sym}^2(\rho\otimes \chi)\otimes\eta)=L(s, r\circ (\rho\otimes \chi)_{\eta})$ and
$\gamma(s,\textrm{Sym}^2(\rho\otimes\chi)\otimes\eta,\psi)=\gamma(s,r\circ (\rho\otimes \chi)_{\eta}, \psi)$.
We aim to prove the following proposition in this section.
\begin{proposition}(\textbf{Stable Equality})
Let $F$ be a p-adic field of characteristic zero, $\eta$ a fixed character of $F^{\times}$, and $\rho$ be an $n$-dimensional continuous irreducible representation of $W_F$. Then for every sufficiently highly ramified character $\chi$ of $F^{\times}$, we have
$$\gamma(s,\textrm{Sym}^2(\rho\otimes \chi)\otimes \eta,\psi)=\gamma(s,\pi\otimes\chi, \textrm{Sym}^2\otimes \eta, \psi),$$ where $\pi=\pi(\rho)\in Irr(\textrm{GL}_n(F))$ is the irreducible admissible representation attached to $\rho$ under the local Langlands correspondence. 
\end{proposition}

We will prove Proposition 3.1 by induction on $n$. It is important to point out that the induction hypothesis will be used in the proof of Proposition 3.2 using a global-to-local argument. We will first establish the proposition for a fixed irreducible representation $\rho_0$ of $W_F$(Proposition 3.2), then use both the arithmetic and analytic stability of $\gamma$-factors (Proposition 3.3 \& 3.4) on the two sides to deform the equality for the fixed representation to obtain the result of Proposition 3.1 for all $n$-dimensional representations $\rho$. We begin with the first step:

\begin{proposition}(\textbf{Stable Equality at a base point})
Let $F$ be a p-adic field, fix a character $\eta$ of $F^{\times}$. Given a character $\omega_0$ of $F^{\times}$, there exists an irreducible n-dimensional representation $\rho_0$ of $W_F$ with $\det\rho_0$ corresponding to $\omega_0$ by local class field theory, such that for all characters $\chi$ of $F^{\times}$, we have
$$\gamma(s,\textrm{Sym}^2(\rho_0\otimes \chi)\otimes \eta,\psi)=\gamma(s,\pi(\rho_0)\otimes\chi, \textrm{Sym}^2\otimes \eta, \psi),$$
\end{proposition}
\begin{proof}This is essentially the same as the proof of Proposition 3.2 in [7]. Using the globalization method provided by Lemma 3.1 in [7], we see that there exists a number field $\mathbb{F}$ and an irreducible continuous n-dimensional representation $\Sigma$ of the global Weil group $W_{\mathbb{F}}$, such that if $\Sigma_v=\Sigma\vert_{W_{\mathbb{F}_v}}$, then there is a place $v_0$ of $\mathbb{F}$ such that $\mathbb{F}_{v_0}=F$, $\det \Sigma_{v_0}$ corresponds to $\omega_0$ by local class field theory. Moreover, $\Sigma_{v_0}$ is irreducible, $\Sigma_v$ is reducible for all $v<\infty$ with $v\neq v_0$, and $\Pi=\pi(\Sigma):=\otimes_v \pi(\Sigma_v)$ is a cuspidal automorphic representation of $\textrm{GL}_n(\mathbb{A}_{\mathbb{F}})$. Therefore all the local components $\Pi_v$ are generic. Let $\Psi=\otimes_v \Psi_v$ be a nontrivial additive character of $\mathbb{F}\backslash \mathbb{A}_{\mathbb{F}}$ so that $\Psi_{v_0}=\psi$, the nontrivial additive character which defines the generic character of $U_n(F)$.
We also take $\tau:\mathbb{F}^{\times}\backslash \mathbb{A}^{\times}_{\mathbb{F}}\rightarrow \mathbb{C}^{\times}$ to be a Hecke character with $\tau_{v_0}=\eta$. Outside a finite set of places $S$ containing $v_0$ and the infinite places, $\Pi_v$, $\tau_v$ and $\Psi_v$ are all unramified. 

Take $\xi:\mathbb{F}^{\times}\backslash \mathbb{A}^{\times}_{\mathbb{F}}\rightarrow \mathbb{C}^{\times}$ a Hecke character such that $\xi_{v_0}=\chi$, it is easy to see that globally we have $\pi(\Sigma\otimes\xi)_{\tau}=(\Pi\otimes\xi)_{\tau}$. Similar to the local case the global $L$-functions are given by
$L(s,\textrm{Sym}^2(\Sigma\otimes\xi)\otimes \tau)=L(s, r\circ (\Sigma\otimes \xi)_{\tau})$ and
$L(s,\Pi\otimes \xi, \textrm{Sym}^2\otimes \tau)=L(s,(\Pi\otimes\xi)_{\tau}, r).$

Now we apply the global functional equations for the Artin L-functions in general as given in [9], and the twisted symmetric square L-function for the automorphic side through Langlands-Shahidi method as in [17], and do some simple calculation on the unramified places, we will be able to match the the product of L-factors at those places. We obtain the equality of the product of local $\gamma$-factors at those "bad" places. Since by [18] we know that the arithmetic and the analytic factors defined by the Langlands-Shahidi method always agree at all Archimedean places [18], we are left with the product of $\gamma$-factors of a finite set of places at which the local components $\Sigma_v$ are all reducible, and a fixed place $v_0$. Let $\Sigma_v=\Sigma_{v,1}\oplus\cdots\oplus \Sigma_{v,r_v}$ be the decomposition of $\Sigma_v$ into irreducibles. We will prove the equality
$\gamma(s, \textrm{Sym}^2((\Sigma_{v,1}\oplus\cdots\oplus \Sigma_{v,r_v})\otimes \xi_v)\otimes \tau_v, \Psi_v)=\gamma(s, \textrm{Ind}(\Pi_{v,1}\otimes\cdots\otimes \Pi_{v,r_v})\otimes\xi_v, \textrm{Sym}^2\otimes\tau_v, \Psi_v)$, by induction on $r_v$. 

Since $\Sigma_v$ is reducible, $r_v\ge 2$. When $r_v=2$ we have
$$\gamma(s, \textrm{Sym}^2((\Sigma_{v,1}\oplus\Sigma_{v,2})\otimes \xi_v)\otimes \tau_v, \Psi_v)$$$$=\gamma(s, \textrm{Sym}^2(\Sigma_{v,1}\otimes\xi_v)\otimes\tau_v, \Psi_v)\gamma(s, \textrm{Sym}^2(\Sigma_{v,2}\otimes \xi_v)\otimes\tau_v,\Psi_v)$$$$\cdot\gamma(s,((\Sigma_{v,1}\otimes \xi_v)\otimes(\Sigma_{v,2}\otimes \xi_v))\otimes\tau_v,\Psi_v)$$
$$=\gamma(s, \Pi_{v,1}\otimes \xi_v, \textrm{Sym}^2\otimes \tau_v,\Psi_v)\gamma(s, \Pi_{v,2}\otimes \xi_v, \textrm{Sym}^2\otimes \tau_v,\Psi_v)$$$$\cdot\gamma(s,((\Pi_{v,1}\otimes \xi_v)\times(\Pi_{v,2}\otimes \xi_v))\otimes\tau_v,\Psi_v)$$
$$=\gamma(s, \textrm{Ind}(\Pi_{v,1}\otimes\Pi_{v,2})\otimes\xi_v, \textrm{Sym}^2\otimes \tau_v,\Psi_v).$$ Here the first equality is the additivity of the arithmetic $\gamma$-factors, the second equality follows from our induction hypothesis of Proposition 3.1 on the dimension $n$ of $\rho$, and the fact the LLC preserves the local $\gamma$-factors in pairs. The last equality is a consequence of the multiplicativity of the analytic $\gamma$-factors. Indeed, recall that the adjoint action $r:{^LM_H}\simeq \textrm{GL}_n(\mathbb{C})\times \textrm{GL}_1(\mathbb{C})\longrightarrow \textrm{GL}({^L\mathfrak{n}}_H)$ is irreducible. ${^LM_H}=\{m=m(g,a_0)=\begin{bmatrix}
g & \ \ \\
\ \ & a_0J'{^tg^{-1}J'^{-1}}\\
\end{bmatrix} g\in \textrm{GL}_n(\mathbb{C}), a_0\in \textrm{GL}_1(\mathbb{C})\}$ and ${^L\mathfrak{n}_H}=\{\begin{bmatrix}
0 & X \\
0 & 0 \\
\end{bmatrix}: J'{^tX}J'=X\}.$
Let $Y=XJ'^{-1}$ then $J'{^tX}J'=X\Leftrightarrow {^tY}=Y.$ Denote $\mathfrak{n}(Y)=\begin{bmatrix}
0 & X \\
0 & 0\\
\end{bmatrix}=\begin{bmatrix}
0 & YJ'^{-1}\\
0 & 0 \\
\end{bmatrix}$. Then an easy calculation shows that 
$r(m(g,a_0))\mathfrak{n}(Y)=\mathfrak{n}(a_0gY{^tg}J').$ Let $\theta_1\subset \theta\subset \Delta$ be the subset of simple roots which gives the Levi subgroup $M_{\theta_1}\simeq \textrm{GL}_{n_1}\times \textrm{GL}_{n_2}\times \textrm{GL}_1$ with $n=n_1+n_2$, therefore ${^L}M_{\theta_1}\simeq \textrm{GL}_{n_1}(\mathbb{C})\times \textrm{GL}_{n_2}(\mathbb{C})\times \textrm{GL}_1(\mathbb{C})$.  Write 
$Y=\begin{bmatrix}
Y_1 & Y_2  \\
Y_3 &Y_4\\
\end{bmatrix}$, then $^tY=Y$ is equivalent to say that $^tY_1=Y_1$, $Y_3={^tY}_2$ and ${^tY}_4=Y_4$. According to the inductive construction of local $\gamma$-factors through Langlands-Shahidi method, we need to decompose the restriction of the adjoint action $r$ on ${^LM_{\theta_1}}$ on $^L\mathfrak{n}_H$ into a direct sum of irreducible subrepresentations (Theorem 8.3.2 of [17]). In our case each of them contributes to a local $\gamma$-factor. The restriction gives that
$$r(m(\begin{bmatrix}
g_1 & \ \ \\
\ \ & g_2 \\
\end{bmatrix}), a_0)(\mathfrak{n}(Y))=\mathfrak{n}(a_0\begin{bmatrix}
g_1 & \ \ \\
\ \ & g_2 \\
\end{bmatrix}\begin{bmatrix}
Y_1 &  Y_2 \\
{^tY}_2 & Y_4 \\
\end{bmatrix}\begin{bmatrix}
{^tg}_1 & \ \ \\
\ \ & {^tg}_2 \\
\end{bmatrix}J')$$$$=\mathfrak{n}(\begin{bmatrix}
a_0g_1 Y_2{^tg}_2 J_{n_2}' & a_0g_1Y_1{^tg}_1 J_{n_1}' \\
a_0 g_2Y_4 {^tg}_2 J_{n_2}' & a_0 g_2 {^tY}_2 {^tg_1} J_{n_1}' \\
\end{bmatrix},$$ where $J'=\begin{bmatrix}
\ \ & J_{n_1}' \\
J_{n_2}' & \ \ \\
\end{bmatrix}$ with $J_{n_i}'$ the same type of matrix as $J'$ of size $n_i$.

Now let's get back to our setting. For $v\in S$, non-archimedean and $v\neq v_0$, $\Pi_{v,1}$ and $\Pi_{v,2}$ are irreducible admissible representations of $\textrm{GL}_{n_1}(\mathbb{F}_v)$ and $\textrm{GL}_{n_2}(\mathbb{F}_v)$ respectively. $\tau_v$ is a fixed character of $\mathbb{F}_v^{\times}$, and $\xi_v$ is a character of $\mathbb{F}_v^{\times}$.
Notice that here $Y_2$ is a free matrix of size $n_1\times n_2$, so the two diagonal blocks above give an irreducible subrepresentation. It is isomorphic to the tensor product $\Pi_{v,1}$ and $\Pi_{v,2}$, twisted by a character $\tau_v$ which is given by the $a_0$-component in the above expression. Therefore it contributes to the twisted Rankin-Selberg local $\gamma$-factor $\gamma(s, (\Pi_{v,1}\times \Pi_{v,2})\otimes \tau_v, \Psi_v)$. If we take $\Pi_{v,i}\otimes \xi_v$ instead of $\Pi_{v,i}$, we obtain the twisted Rankin-Selberg $\gamma$-factor $\gamma(s, ((\Pi_{v,1}\otimes \xi_v)\times (\Pi_{v,2}\otimes \xi_v))\otimes \tau_v, \Psi_v)$. Moreover, notice that ${^tY}_1=Y_1$ and ${^tY_4}=Y_4$, and the form of each of the rest blocks shows that each of them is isomorphic to the adjoint action of $^LM_i$ on $^L\mathfrak{n_i}$, where $M_i$ is the same type of Siegel Levi inside $\textrm{GSpin}_{2n_i+1}$. Therefore they are both irreducible, and they contribute to the twisted symmetric square local $\gamma$-factors $\gamma(s, \Pi_{v,i}, \textrm{Sym}^2\otimes \tau_v, \Psi_v)$, $i=1,2$. Again take $\Pi_{v,i}\otimes \xi_v$ instead of $\Pi_{v,i}$, we obtain the two $\gamma$-factors $\gamma(s, \Pi_{v,1}\otimes \xi_v, \textrm{Sym}^2\otimes \tau_v, \Psi_v)$ and $\gamma(s, \Pi_{v,2}\otimes \xi_v, \textrm{Sym}^2\otimes \tau_v, \Psi_v)$. Therefore by the multiplicativity of the local analytic $\gamma$-factors, we obtain that
$$\gamma(s, \textrm{Ind}(\Pi_{v,1}\otimes\Pi_{v,2})\otimes\xi_v, \textrm{Sym}^2\otimes \tau_v,\Psi_v)$$
$$=\gamma(s, \Pi_{v,1}\otimes \xi_v, \textrm{Sym}^2\otimes \tau_v,\Psi_v)\gamma(s, \Pi_{v,2}\otimes \xi_v, \textrm{Sym}^2\otimes \tau_v,\Psi_v)$$$$\cdot\gamma(s,((\Pi_{v,1}\otimes \xi_v)\times(\Pi_{v,2}\otimes \xi_v))\otimes\tau_v,\Psi_v).$$ This establishes the last equality. The general case follows from the case $r_v=2$ by induction on $r_v$. Hence from the global functional equations we are left with
$\gamma(s, \textrm{Sym}^2(\rho_0\otimes \chi)\otimes \eta, \psi)=\gamma(s, \pi(\rho_0)\otimes\chi, \textrm{Sym}^2\otimes \eta, \psi)$.
\end{proof}

To prove Proposition 3.1, besides Proposition 3.2, we also need both the arithmetic and analytic stability for $\gamma$-factors. We will explain as follows.

On the arithmetic side, P. Deligne showed the existence and uniqueness of the local $\epsilon$-factors on page 535-547 in [9]. For $V$ a finite dimensional complex representation of the local Weil group, $\chi$ is sufficiently ramified character of $F^{\times}$, the arithmetic $\epsilon$-factor attached to $V\otimes \chi$ depends only on $\det(V)$ and $\dim(V)$. Apply this to the case when $V\simeq \textrm{Sym}^2\rho\otimes \eta$ where $\rho$ is an irreducible n-dimensional representation of $W_F$, and $\eta$ is a character of $F^{\times}$ viewed as a character of $W_F$ as before. Also notice that $L(s, V\otimes \chi)=1$ for $\chi$ sufficiently ramified, we obtain:
\begin{proposition}(\textbf{Arithmetic Stability for the twisted symmetric square $\gamma$-factors})
Let $\rho_1$ and $\rho_2$ be two continuous n-dimensional representations of $W_F$ with $\det(\rho_1)=\det(\rho_2)$, $\eta$ be a fixed character of $F^{\times}$. Then for all sufficiently ramified characters $\chi$ of $F^{\times}$ we have
$$\gamma(s, \textrm{Sym}^2(\rho_1\otimes \chi)\otimes \eta, \psi)=\gamma(s, \textrm{Sym}^2(\rho_2\otimes \chi)\otimes\eta, \psi).$$
\end{proposition}

On the analytic side, $\pi=\pi(\rho)$ is supercusipidal when $\rho$ is irreducible, therefore analogously we should have:
\begin{proposition}(\textbf{Supercuspidal Stability for the twisted symmetric square $\gamma$-factors})
Let $\pi_1$ and $\pi_2$ be two supercusipidal representations of $\textrm{GL}_n(F)$ with $\omega_{\pi_1}=\omega_{\pi_2}$, and $\eta$ is a fixed character of $F^{\times}$. Then for all sufficiently ramified characters $\chi$ of $F^{\times}$, whose degree of ramification depends only on $\pi_1$ and $\pi_2$, identified as characters of $\textrm{GL}_n(F)$ through the determinant, we have
$$\gamma(s, \pi_1\otimes \chi, \textrm{Sym}^2\otimes\eta, \psi)=\gamma(s,\pi_2\otimes\chi, \textrm{Sym}^2\otimes\eta, \psi).$$
\end{proposition}

This is the main result of this paper and will be established in the remainder of the text.

With Proposition 3.2, 3.3, and 3.4, we are ready to prove Proposition 3.1.
\begin{proof}(Proof of Proposition 3.1)
We will do induction on the dimension $n$ with the help of a globalization method provided as on page 2061-2065 in [7]. 

When $n=1$ we obtain that both sides equal to 1, and there is nothing to prove. For $n=2$, one could either follow [8] directly, or instead we show 
$\gamma(s,\wedge^2(\rho\otimes \chi)\otimes \eta, \psi)=\gamma(s, \pi\otimes \chi, \wedge^2\otimes \eta, \psi).$ These $\gamma$-factors are in general defined again through Langlands-Shahidi method by the adjoint action of $^LM$ on ${^L\mathfrak{n}}$ where $M$ is the maximal Levi isomorphic to $\textrm{GL}_n\times \textrm{GSpin}_0\simeq \textrm{GL}_n\times \textrm{GL}_1$ inside $\textrm{GSpin}_{2n}$(Theorem 2.7 [1]). Notice that in this case $\wedge^2\rho\otimes \eta=\det(\rho)\otimes \eta$. On the other hand, it is not hard to see that $\gamma(s,\pi, \wedge^2\otimes \eta, \psi)=\gamma(s, \omega_{\pi}\times \eta,\psi)$, where $\omega_{\pi}$ is the central character of $\pi$, and the right hand side is the $\gamma$-factor attached to the Rankin-Selberg L-function $L(s,\omega_{\pi}\times \eta).$ Since we know that $\det \rho\leftrightarrow \omega_{\pi}$ under the local Langlands correspondence, and tensor product of representations on the arithmetic side corresponds to Rankin-Selberg convolutions on the analytic side, so $\det\rho\otimes \eta\leftrightarrow \omega_{\pi}\times \eta$. Moreover, since LLC is compatible with twisting by characters, we see that the stable equality is true for the twisted exterior square $\gamma$-factors when $n=2$, and for this case we don't even need to assume $\chi$ is highly ramified. Now apply the equalities
$$\gamma(s, (\pi\times\pi)\times \eta,\psi)=\gamma(s,\pi,\wedge^2\otimes \eta,\psi)\gamma(s,\pi,Sym^2\otimes\eta,\psi)$$
$$\gamma(s,(\rho\otimes\rho)\otimes\eta,\psi)=\gamma(s,\wedge^2\rho\otimes\eta,\psi)\gamma(s,Sym^2\rho\otimes\eta,\psi),$$ and by the fact that LLC preserves the $L$- and $\epsilon$-factors of pairs, we see that the proposition is true for the case when $n=2$ and any character $\chi$.

Now $\rho$ is an irreducible n-dimensional representation of $W_F$, let $\pi=\pi(\rho)$ be its corresponding supercuspidal representation of $\textrm{GL}_n(F)$. Take $\omega_0=\omega_{\pi}$ in Proposition 3.2, then there exists an irreducible n-dimensional representation $\rho_0$ of $W_F$ and its corresponding supercuspidal representation $\pi_0=\pi(\rho_0)$ of $\textrm{GL}_n(F)$ such that $\omega_{\pi}=\omega_{\pi_0}$, $\det (\rho)=\det(\rho_0)$ and $\gamma(s, \textrm{Sym}^2(\rho_0\otimes\chi)\otimes\eta,\psi)=\gamma(s, \pi_0\otimes\chi, \textrm{Sym}^2\otimes\eta,\psi)$. Take $\chi$ sufficiently ramified such that Proposition 3.3 holds for the pair $(\rho,\rho_0) $, and Proposition 3.4 holds for the pair $(\pi,\pi_0)$. Then for such $\chi$ we have
$$\gamma(s, \textrm{Sym}^2(\rho\otimes\chi)\otimes\eta,\psi)=\gamma(s, \textrm{Sym}^2(\rho_0\otimes\chi)\otimes\eta,\psi)$$$$=\gamma(s, \pi_0\otimes\chi,\textrm{Sym}^2\otimes\eta,\psi)=\gamma(s,\pi\otimes\chi, \textrm{Sym}^2\otimes\eta,\psi)$$
The degree of ramification now depends on $(\rho,\pi)$ and $(\rho_0,\pi_0)$, so one needs to fix such a base point $(\rho_0,\pi_0)$ for every character $\omega_0$. As in [7], this can be reduced to just fix the character $\omega_0$ since twisting by unramified characters can be absorbed into the complex parameter $s$ of the $\gamma$-factors. This completes the proof of Proposition 3.1.
\end{proof}

Next we extend our result to Weil-Deligne representations.

\begin{corollary}
Let $\rho$ be a continuous n-dimensional $\Phi$-semisimple complex representation of the Weil-Deligne group $W_F'$, and $\eta$ a fixed character of $F^{\times}$. Then for sufficiently ramified characters $\chi$ of $F^{\times}$ we have
$$\gamma(s, \textrm{Sym}^2(\rho\otimes\chi)\otimes\eta,\psi)=\gamma(s, \pi(\rho)\otimes\chi, \textrm{Sym}^2\otimes\eta, \psi).$$
\end{corollary}
\begin{proof}
The corollary follows from the following facts: (1) the compatibility of the construction of $\Phi$-semisimiple representations of $W_F'$ from irreducible representations of $W_F$ and the Bernstein-Zelevinsky construction [3] of irreducible representations of $\textrm{GL}_n(F)$ from supercuspidals; (2) the local $\gamma$-factors attached to $\rho$ only depends on its semisimplification(as representations of $W_F$)(page 201, [4]); (3) LLC is compatible with pairs of local L-factors and the twisted symmetric square L-factors on both the arithmetic and the analytic sides, and under highly ramified twists these become 1 [12]; (4), the additivity of the arithmetic local $\gamma$-factors [7] and the multiplicativity of the analytic local $\gamma$-factors, which was proved by an induction argument as in Proposition 3.2.  
\end{proof}
\begin{corollary}(\textbf{General analytic stability for the twisted symmetric square $\gamma$-factors}) Let $\pi_1$ and $\pi_2$ be two irreducible admissible representations of $GL_n(F)$ with $\omega_{\pi_1}=\omega_{\pi_2}$, $\eta$ is a fixed character of $F^{\times}$. Then for any sufficiently ramified character $\chi$ of $F^{\times}$ we have 
$$\gamma(s, \pi_1\otimes\chi, Sym^2\otimes\eta, \psi)=\gamma(s, \pi_2\otimes\chi, Sym^2\otimes\eta, \psi)$$
\end{corollary}
\begin{proof}
Let $\rho_1$ and $\rho_2$ be two continuous n-dimensional $\Phi$-semisimple representations of the Weil-Deligne group $W_F'$ and $\pi_i=\pi(\rho_i)$ (i=1,2)  be their corresponding irreducible admissible representations of $\textrm{GL}_n(F)$. By corollary 3.5 we have 
$\gamma(s, \textrm{Sym}^2(\rho_i\otimes\chi)\otimes\eta, \psi)=\gamma(s, \pi_i\otimes\chi,\textrm{Sym}^2\otimes\eta,\psi).$ Then we can see that the result would follow if we have the analogue of Proposition 3.3 for Weil-Deligne representations. On the other hand, we know that the arithmetic $\gamma$-factors depend only on the semisimplification, i.e., we have $\gamma(s, \rho,\psi)=\gamma(s,\rho^{ss},\psi)$. Since the semisimplification does not change the determinant $\det{\rho}$ and $\dim(\rho_1)=\dim(\rho_2)=n$, so again since the local arithmetic $\epsilon$-factors depend only on $\det(\rho) $ and $\dim(\rho)$ under suitably highly ramified twist by $\chi$, as we mentioned earlier. So we can take $\chi$ sufficiently ramified such that the arithmetic stability of $\gamma$-factors follows for Weil-Deligne representations. That is, $\gamma(s, \textrm{Sym}^2(\rho_1\otimes\chi)\otimes\eta,\psi)=\gamma(s, \textrm{Sym}^2(\rho_2\otimes\chi)\otimes\eta,\psi).$ Then the result follows immediately from Corollary 3.5.
\end{proof} 

\section{PROOF OF THE MAIN THEOREM}

In this section we will prove our main theorem(Theorem 1.1), by assuming the analytic stability of the twisted symmetric square $\gamma$-factors attached to supercuspidal representations(Proposition 3.4).  

Before we proceed, as in [7], we make a remark on the additive character $\psi$ of $F$. Take $a\in F^{\times}$ and fix a non-trivial additive character $\psi$ of $F$. Let $\psi^a$ denote the character given by $\psi^a(x)=\psi(ax).$ By the study of Henniart [11] and Deligne [9] respectively, it turns out that as a function of $a\in F^{\times}$, both the analytic $\gamma$-factors $\gamma(s, \pi, r, \psi^a)$ and the corresponding arithmetic $\gamma$-factors $\gamma(s, r\circ \rho, \psi^a)$ vary in the same way. Therefore it suffices to prove the result for a fixed $\psi$. 

We will first establish the equality for the $\gamma$-factors, and then use it to obtain the equality for L-factors. We begin with some lemmas:
\begin{lemma}(\textbf{Equality for monomial representations})
Let $E/F$ be a finite Galois extension of degree n contained in a fixed algebraic closure $\overline{F}$ of $F$, and $\eta$ be a fixed character of $F^{\times}$. Denote $G=\Gal(E/F)$. Let $F\subset L\subset E$ be an intermediate extension and $\chi$ be a finite-order character of $H=\Gal(E/L)$. Let $\rho=\textrm{Ind}_H^G(\chi)$, then
$$\gamma(s, \textrm{Sym}^2\rho\otimes\eta,\psi)=\gamma(s,\pi(\rho), \textrm{Sym}^2\otimes\eta,\psi)$$
\end{lemma}
\begin{proof} This is the same globalization method as used in Lemma 3.2 in [7], one may simply replace the $\wedge^2$ there by $\textrm{Sym}^2\otimes\eta$,  change the equalities in the proof accordingly and use Proposition 3.1 and 3.2.
\end{proof}
\begin{lemma}(\textbf{Equality for Galois representations})
Let $\rho$ be an irreducible continuous n-dimensional representation of $W_F$ with $\det(\rho)$ being a character of finite order, and $\eta$ be a fixed character of $F^{\times}$. Then
$$\gamma(s, Sym^2\rho\otimes\eta,\psi)=\gamma(s, \pi(\rho), Sym^2\otimes\eta, \psi).$$
\end{lemma}
\begin{proof} This is also a straightforward analogue of Lemma 3.3 in [7]. A very similar argument shows that the arithmetic and analytic twisted symmetric square local $\gamma$-factors satisfy the same formalism, then we use additivity and multiplicativity of the arithmetic and analytic twisted symmetric square $\gamma$-factors respectively, together with Lemma 4.1 then we are done.
\end{proof}

Now we have all the ingredients for the proof of Theorem 1.1.
\begin{proof}(\textbf{Proof of Theorem 1.1}) First we prove the equality of $\gamma$-factors. By Lemma 4.2, we have the equality of the local twisted symmetric square $\gamma$-factors for irreducible continuous representations of $W_F$ with finite order determinant. After tensoring with an unramified character, we can extend the result to any irreducible continuous n-dimensinal representation of $W_F$. Both LLC  and the formalism of the twisted symmetric square $\gamma$-factors are compatible with twisting by characters. Since LLC also preserves the local $\gamma$-factors for direct sums of representations on the arithmetic side with isobaric sums of the corresponding representations on the analytic side, we can further extend the result in Lemma 4.2 to arbitrary continuous $n$-dimensional representations of $W_F$.

Next, as in the proof of Corollary 3.5, we can extend the result to all continuous $\Phi$-semisimple n-dimensional representations of the Weil-Deligne group $W_F'$.
This completes the proof of the equality of the twisted symmetric square $\gamma$-factors in Theorem 1.1. 

We are left with the equality of L-factors.  We use a similar argument as Henniart's proof in [12] to show that the equality of $\gamma$-factors imply the equality of their corresponding L-factors. One can also see this by using the Langlands-Shahidi method ([16],[17]).

Recall that $\pi$ is an irreducible representation of $\textrm{GL}_n(F)$. 
Suppose $\pi\leftrightarrow \rho$ under LLC, where $\rho=(\rho',V, N)$. In general if $r$ is any analytic representation of $\textrm{GL}_n(\mathbb{C})$ we have that $r\circ\rho=(r\circ\rho', r(V), \frac{d}{dx}\vert_{x=0}(r\circ\rho)(x))$ is also a Weil-Deligne representation, where $r(V)$ is the space given by $r$ and $V$, i.e., $r:{^LG}=GL_n(\mathbb{C})\longrightarrow GL(r(V)).$ Notice that the monodromy operator $N$ satisfies
$\rho(x)v=\exp(xN)v$ for all $v\in V$ and $x\in \mathbb{G}_a$. Recall that  $W_F'\simeq W_F\rtimes \mathbb{G}_a$. So $N=\frac{d}{dx}\vert_{x=0}\rho(x)$, therefore in general the monodromy operator $T$ for $r\circ\rho$ is given by $T=\frac{d}{dx}\vert_{x=0}(r\circ\rho)(x)$.

Following Henniart's terminology in [12], we say a Weil-Delinge representation $\rho$ is tempered if all its indecomposable constituents are of the form $\rho_i'\otimes Sp(m_i)$ where $\rho_i'$ is an irreducible unitary representation of $W_F$ and $Sp(m_i)$  is a special representation of dimension $m_i$, corresponding to a Steinberg representation of $\textrm{GL}_{m_i}(F)$. Equivalently, if we define the Weil-Delinge group to be $W_F\rtimes \textrm{SL}_2(\mathbb{C})$, then the image of $W_F$ is bounded in $\textrm{GL}(V)$. Since we have the exact sequence 
$$0\rightarrow I_F\rightarrow  W_F\rightarrow \mathbb{Z}\rightarrow 0$$ where $I_F$ is the inertial subgroup, which is compact, it is the same as saying that the image of the geometric Frobenius is a unitary operator on $V$.
For this purpose here we use another definition of the Weil-Deligne group given by $W_F\rtimes \textrm{SL}_2(\mathbb{C})$. By Theorem 2.8 of [20], the triple $\rho=(\rho',V,N)$ is equivalent to a representation $\varphi: W_F\rtimes \textrm{SL}_2(\mathbb{C})\rightarrow \textrm{GL}_n(\mathbb{C})$ such that $\varphi$ is trivial on an open subgroup of $I_F$, $\varphi(\Phi)$ is semi-simple and $\varphi\vert_{\textrm{SL}_2(\mathbb{C})}$ is algebraic. By Lemma 2.9 of [23], there exists a unique $\mathfrak{sl}_2$-triple $(e,f,h)$ such that $e=N=\mathfrak{gl}_n^{\rho(I_F)}(\Phi)(q^{-1})$, $f=\mathfrak{gl}_n^{\rho(I_F)}(q)$, and $h=\mathfrak{gl}_n^{\rho(W_F)}=\mathfrak{gl}_n^{\rho(I_F)}(1)$, where $q=\vert\mathcal{O}_F/\mathfrak{m}_F\vert$ is the cardinality of the residue field and $V(q)$ denotes the q-eigenspace of the action of $\rho(\Phi)$ on V. Then the corresponding representation $\varphi: W_F\rtimes \textrm{SL}_2(\mathbb{C})\rightarrow \textrm{GL}_n(\mathbb{C})$ is given by
$\varphi(w)=\exp(\frac{-v(w)}{2}\log q\cdot h)\rho(w).$

First we assume that $\pi$ is tempered and $\eta$ is unitary. Then it follows that the representation $\sigma_\eta$ of $M_H(F)$ is tempered.
We show  $\textrm{Sym}^2\rho\otimes\eta$ is also tempered.  $\rho=(\rho',V,N)$ implies that
$\textrm{Sym}^2\rho\otimes\eta=(\textrm{Sym}^2\rho'\otimes\eta, \textrm{Sym}^2(V), 1\otimes N+N\otimes 1)$, here we identify $\textrm{Sym}^2\rho$ as a subspace of $\rho\otimes\rho$ generated by $e_i\otimes e_j+e_j\otimes e_i$ where $\{e_i\}_{i=1}^n$ is a basis of $V$. Now if $\rho$ is given by $\varphi$ as above, then $\textrm{Sym}^2\rho\otimes\eta$ is given by $\tilde{\varphi}: W_F\rtimes \textrm{SL}_2(\mathbb{C})\rightarrow \textrm{GL}_n(\mathbb{C})$ by
$\tilde{\varphi}(w)=\exp(\frac{-v(w)}{2}\log q\cdot H)\textrm{Sym}^2\rho\otimes\eta(w)=\exp(\frac{-v(w)}{2}\log q\cdot H)(\rho\otimes \rho)\vert_{\textrm{Sym}^2(V)}(w)\cdot \eta(w),$ where $H=1\otimes h+h\otimes 1.$ Notice that if $e=N, f, h$ form an $\mathfrak{sl}_2$-triple, then $E=1\otimes N+N\otimes 1, F=1\otimes f+ f\otimes 1$, and $H=1\otimes h+h\otimes 1$ also form an $\mathfrak{sl}_2$-triple. $\pi$ being tempered implies that $\rho$ is tempered, therefore $U=\varphi(\Phi)=\exp(\frac{1}{2}\log q\cdot h)\rho(\Phi)$ is unitary. Since $\eta$ is unitary, it suffices to show that $\exp(\frac{1}{2}\log q\cdot (1\otimes h+h\otimes 1))(\rho\otimes\rho)\vert_{\textrm{Sym}^2(V)}(\Phi)$ is unitary, thus it suffices to show that $\exp(\frac{1}{2}\log q\cdot (1\otimes h+h\otimes 1))(\rho\otimes\rho)(\Phi)$ is unitary. We have $\exp(\log \sqrt{q}(1\otimes h+h\otimes 1))(\rho\otimes\rho)(\Phi)=\exp(\log \sqrt{q}(1\otimes h))\cdot\exp(\log \sqrt{q}(h\otimes 1))((1 \otimes \rho)(\Phi)\cdot(\rho\otimes 1)(\Phi))=\exp(1\otimes \log\sqrt{q}\cdot h)(1\otimes \rho(\Phi))\cdot \exp(\log \sqrt{q}\cdot h\otimes 1)(\rho\otimes 1)(\Phi)=(1\otimes U)\cdot (U\otimes 1)$ is unitary since $U$ is unitary. Therefore $\textrm{Sym}^2\rho\otimes \eta$ is tempered.

In this case we have that
$L(s, \textrm{Sym}^2\rho\otimes \eta)$ has no poles for $Re(s)>0$, and for the same reason we have that
$L(1-s, \textrm{Sym}^2{\rho}^{\vee}\otimes\eta^{-1})$ has no poles for $Re(s)<1$. By Langlands-Shahidi method we have
$$\gamma(s, \pi, \textrm{Sym}^2\otimes \eta,\psi)=\epsilon(s, \textrm{Sym}^2\rho\otimes \eta, \psi)\frac{L(1-s, \textrm{Sym}^2\rho^{\vee}\otimes\eta^{-1})}{L(s,\textrm{Sym}^2\rho\otimes\eta)}$$
Moreover, $\gamma(s,\pi,\textrm{Sym}^2\otimes \eta,\psi)$ is a rational function of $q^{-s}$. To be precise, $\gamma(s,\pi,\textrm{Sym}^2\otimes\eta,\psi)=F(q^{-s})$ where $F(X)=c X^a\frac{P(X)}{Q(X)}$ with $P(X),Q(X)\in \mathbb{C}[X]$ such that $P(0)=Q(0)=1$, $c\in \mathbb{C}$ and $a\in \mathbb{Z}$. We also know that $\epsilon(s, \textrm{Sym}^2\rho\otimes\eta,\psi)$ is a monomial of $q^{-s}$. The local tempered L-factor is defined as $L(s,\pi, \textrm{Sym}^2\otimes \eta)=P(q^{-s})$. Since $L(s, \textrm{Sym}^2\rho\otimes\eta)$ and $L(1-s,\textrm{Sym}^2\rho^{\vee}\otimes \eta^{-1})$ have no poles in common, similar to Henniart's proof in [12], we can conclude that $L(s,\pi, \textrm{Sym}^2\otimes\eta)=L(s, \textrm{Sym}^2\rho\otimes\eta).$

Now if $\sigma_{\eta}$ is quasi-tempered, then $\pi$ is quasi-tempered and $\eta$ is arbitrary. Let $\tau_0:M(F)\simeq \textrm{GL}_n(F)\times \textrm{GL}_1(F)\rightarrow \mathbb{C}^{\times}$ be an unramified character of $M(F)$ given by $\tau_0=\vert \det(\cdot)\vert^{s_1} \vert\cdot\vert ^{s_2}$, where $s_1, s_2\in \mathbb{C}$. The fundamental weight attached to $\alpha$ is given by $\hat{\alpha}=\langle \rho, \alpha \rangle^{-1}\rho$ where $\rho$ is half of the sum of positive roots in $N_H$. In our case $\alpha=\alpha_n=e_n$ and $\rho=\frac{1}{2}(\sum_{1\leq i<j\leq n}(e_i+e_j)+\sum_{i=1}^n e_i)=\frac{n}{2}\sum_{i=1}^n e_i,$ therefore we have
$$\langle \rho, \alpha\rangle=\frac{2(\rho,\alpha)}{(\alpha,\alpha)}=\frac{2(\frac{n}{2}\sum_{i=1}^n e_i, e_n)}{(e_n,e_n)}=n$$ where $(\cdot,\cdot)$ is a Weyl group invariant non-degenerate bilinear form on $\mathfrak{a}^*=X^*(H)\otimes_{\mathbb{Z}}\mathbb{R}.$ So
$\hat{\alpha}=\langle \rho, \alpha\rangle ^{-1}\rho=n^{-1}(\frac{n}{2}\sum_{i=1}^ne_i)=\frac{1}{2}\sum_{i=1}^ne_i.$

For $s\in \mathbb{C}$, define
$\sigma_{\eta,s}=\sigma_\eta\otimes q^{\langle s\hat{\alpha}, H_M(\cdot)\rangle}\simeq(\sigma_s)_\eta$
where $\sigma_s$ is the lift of the representation $ \pi\otimes \vert\det(\cdot)\vert^{\frac{s}{2}}$ of $\textrm{GL}_n(F)$ to $M_H(F)$. So for $v\in V_{\pi}$,
$\sigma_s(m(g,a))v=\vert \det(g)\vert^{\frac{s}{2}}\pi(g)v$
and
$\sigma_{\eta,s}(m(g,a))v=\eta^{-1}(a)\vert \det(g)\vert^{\frac{s}{2}}\pi(g)v$. Let $\eta_s=\eta\cdot \vert\cdot \vert^s$, then
$\sigma_\eta\otimes \tau_0\simeq \sigma_{\eta_{-s_2},2s_1}.$ Now if $\eta=\eta_0\vert\cdot \vert^{z_0}$ where $\eta_0$ is unitary and $z_0\in \mathbb{C}^*$, take $s_2=z_0$, and take $s_1$ such that $\pi\otimes \vert\det(\cdot)\vert^{s_1}$ is tempered, then by the previous case we have
$$L(s, \sigma_\eta\otimes\tau_0, r)=L(s, (\sigma_{2s_1})_{\eta_0}, r)=L(s, \textrm{Sym}^2(\rho\otimes\vert\vert\cdot\vert\vert^{s_1})\otimes \eta_0)$$
$$=L(s+2s_1, \textrm{Sym}^2\rho\otimes\eta_0)=L(s+2s_1+s_2,\textrm{Sym}^2\rho\otimes \eta).$$ On the other hand, we apply section 2.7 of [12], which states how the local analytic $\gamma$-factor shifts under twists by unramified character of the maximal split quotient of $M_H$, to our case. The maximal split quotient $T_0$ of $M_H\simeq \textrm{GL}_n\times \textrm{GL}_1$ is isomorphic to $\textrm{GL}_1\times \textrm{GL}_1$, since the derived group $M_{H,{der}}$ of $M_H$ is isomorphic to $\textrm{SL}_n$. The adjoint action $r:{^L}M_H\longrightarrow \textrm{GL}({^L}\mathfrak{n}_H)$ is irreducible, so its restriction on the torus $\hat{T}_0$ is given by a character $\chi_r: \hat{T}_0\longrightarrow \mathbb{C}^{\times}.$ In our case, $r$ is given by the symmetric square action twisted by a character given by the $\textrm{GL}_1$ part of $^LM_H$. A direct calculation shows that $\chi_r: \hat{T}_0\longrightarrow GL({^L}{\mathfrak{n}_H})$ is given by $(xI_n, y)\mapsto x^2y$. Taking dual of this map we obtain a one-parameter subgroup $\hat{\chi}_r: F^{\times}\longrightarrow T_0\simeq \textrm{GL}_1\times \textrm{GL}_1$ given by $x\mapsto (x^2,x)$. Notice that $\tau_0\in X_{un}(M)$, and $M_{H,der}\subset \ker(H_{M_H})$, where $H_{M_H}: M_H(F)\longrightarrow \mathfrak{a}_{M_H}=\Hom(X(M_H)_F, \mathbb{Z})\otimes \mathbb{R}$ is the Harish-Chandra map. Therefore $\tau_0$ defines an unramified character on $T_0(F)$, say $\overline{\tau}_0: T_0(F)\longrightarrow \mathbb{C}^{\times}$ such that $\overline{\tau}_0\circ (\det\times id)=\tau_0$. Since $\tau_0=\vert\det(\cdot)\vert^{s_1}\vert\cdot\vert^{s_2}$, we see that $\overline{\tau}_0=\vert\cdot\vert^{s_1}\vert\cdot\vert^{s_2}$. Following [12], this defines an unramified character $\overline{\tau}_0\circ \hat{\chi}_r: F^{\times}\longrightarrow \mathbb{C}^{\times}$ given by $x\mapsto \vert x^2\vert^{s_1}\vert x\vert^{s_2}=\vert x\vert^{2s_1+s_2}$. Therefore by section 2.7 of [12] we obtain
$\gamma(s+2s_1+s_2,\pi, \textrm{Sym}^2\otimes\eta,\psi )=\gamma(s, \sigma_\eta\otimes \tau_0, r,\psi),$ therefore also
$L(s+2s_1+s_2,\pi, \textrm{Sym}^2\otimes\eta)=L(s, \sigma_{\eta}\otimes\tau_0,r)$, by the previous argument on the tempered case.
Compare it with the arithmetic side we obtain
$L(s+2s_1+s_2, \pi, \textrm{Sym}^2\otimes\eta)=L(s+2s_1+s_2, \textrm{Sym}^2\rho\otimes\eta).$ Then by the uniqueness of complex meromorphic functions we see that
$L(s,\pi,\textrm{Sym}^2\otimes\eta)=L(s, \textrm{Sym}^2\rho\otimes\eta).$ This shows the case when $\sigma_{\eta}$ is quasi-tempered.

In general, if $\rho$ is an n-dimensional $\Phi$-semisimple representation of $W_F'$, then $\rho=\oplus_{i=1}^r\rho_i$, where each $\rho_i$ is indecomposable and $\rho_i\simeq \rho_i'\otimes Sp(m_i)$, where each $\rho_i'$ is an irreducible $n_i'$-dimensional representation of $W_F$. Let $\pi'_i=\pi(\rho_i')\leftrightarrow \rho_i'$ under LLC, and let $\Delta_i$ be the segment $\{\pi_i',\pi_i'(1),\cdots,\pi_i'(m_i-1)\}$ where $\pi_i'(j)=\pi_i'\otimes\vert \det(\cdot)\vert^j$. Then the Bernstein-Zelevinsky's classification [3] tells us that $\rho_i\leftrightarrow Q(\Delta_i)$, where $Q(\Delta_i)$ is the unique irreducible subquotient of $\textrm{Ind}_{\textrm{GL}_{n_i}(F)^m}^{\textrm{GL}_{n_im_i}(F)}\pi_i'\otimes \pi_i'(1)\otimes\cdots\otimes\pi_i'(m_i-1)$ and $\pi(\rho)$ is the unique irreducible subquotient of $\textrm{Ind}_{\prod \textrm{GL}_{n_im_i}(F)}^{\textrm{GL}_n(F)}Q(\Delta_1)\otimes Q(\Delta_2)\otimes\cdots\otimes Q(\Delta_r)$. To simplify the notation we use $Q(\Delta_1)\times\cdots\times Q(\Delta_r)$ to denote this induced representation. For each $1\leq i\leq r$ there exists a unique $\beta_i\in \mathbb{R}$ such that $Q(\Delta_i)(-\beta_i)$ is square integrable, thus tempered. We can order the $\Delta_i$'s such that
$\alpha_1:=\beta_1=\beta_2=\cdots=\beta_{m_1}>\alpha_2:=\beta_{m_i+1}=\cdots=\beta_{m_2}>\cdots>\alpha_s:=\beta_{m_{s-1}+1}=\cdots=\beta_r.$
In this order $\Delta_i$ does not precede $\Delta_j$ for $i<j$ and all $\Delta_i$'s corresponding to the same $\alpha_j$ are not linked. For $1\leq j\leq s$, let
$\pi_j=Q(\Delta_{m_{j-1}+1})(-\alpha_j)\times\cdots\times Q(\Delta_{m_j})(-\alpha_j)$ where $m_0=0$ and $m_s=r$. Then all the $\pi_j$'s are irreducible tempered representations, and $\pi=\pi(\rho)$ is the unique irreducible subquotient of $\pi_1(\alpha_1)\times\cdots\times\pi_s(\alpha_s)$. This gives the Langlands classification [13]. We denote the corresponding parabolic subgroup by $P$ and let $\sigma=\pi_1\times\cdots\times\pi_s$, $\nu=\vert\det(\cdot)\vert^{\alpha_1}\otimes\vert\det(\cdot)\vert^{\alpha_2}\otimes\cdots\otimes\vert\det(\cdot)\vert^{\alpha_s}$, and $\pi=\pi(\rho)=J(P,\sigma,\nu)$.

On the other hand, by section 1.4* of [18] we know that $J(P,\sigma,\nu)=\tilde{I}(P,\tilde{\sigma},-\nu)$ where $\tilde{}$ denotes the contragredient, and $I(P,\sigma,\nu)$ denotes the unique irreducible subrepresentation of the parabolic induction $\textrm{Ind}_P^G (\sigma\otimes \nu)$ [5]. By Langlands-Shahidi method we know the multiplicativity of the local analytic $\gamma$-factors attached to generic representations which appear as subrepresentations of parabolic inductions from irreducible generic representations. We also have the multiplicativity of their corresponding local analytic $L$-factors. Using $J(P,\sigma,\nu)=\tilde{I}(P,\tilde{\sigma},-\nu)$ and the local functional equation
$\gamma(s, \pi, \textrm{Sym}^2\otimes\eta,\psi)\gamma(1-s, \tilde{\pi}, \textrm{Sym}^2\otimes\eta^{-1},\overline{\psi})=1$,
we obtain the multiplicativity of $\gamma(s,\pi, \textrm{Sym}^2\otimes\eta,\psi)$ and $L(s,\pi, \textrm{Sym}^2\otimes\eta)$ with respect to their quasi-tempered inducing data. Since we already showed the equality of $L$-factors for quasi-tempered case, we finally obtain that
$L(s, \pi(\rho), \textrm{Sym}^2\otimes\eta)=L(s, \textrm{Sym}^2\rho\otimes\eta).$ By the symmetry between $\wedge^2$ and $\textrm{Sym}^2$ we also obtain that
$L(s, \pi(\rho), \wedge^2\otimes\eta)=L(s, \wedge^2\rho\otimes\eta).$
\end{proof}
So far we have successfully reduced the problem to the supercuspidal stability(Proposition 3.4), which will be established in the rest part of this paper. We will start with some preparations in section 5, in which we will obtain a formula of the local coefficients in our case as the Mellin transform of some partial Bessel functions, and relate the partial Bessel functions with partial Bessel integrals. Then we will study the analysis of partial Bessel integrals in section 6 and obtain their asymptotic expansion formulas, generalizing the results in [7].
\section{PREPARATIONS FOR SUPERCUSPIDAL STABILITY}

We've already seen that the adjoint action $r: \leftidx{^L}M_H\longrightarrow \textrm{GL}(\leftidx{^L}{\mathfrak{n}_H})$
gives the twisted symmetric L- and $\gamma$-factors. Moreover, since $r$ is irreducible we have that the local coefficient $C_{\psi}(s,\pi)=\gamma(s, \pi, \textrm{Sym}^2\otimes \eta,\psi)$(Chapt. 5, [14]).
So it reduces the proof of Proposition 3.4 to the stability of local coefficients. The local coefficients can be written as the Mellin transform of certain partial Bessel functions under some conditions (Theorem 6.2, [19]). In order to study the Mellin transform in our case, we need to understand the following things at first: the structure of $H=\textrm{GSpin}_{2n+1}$, the structure and measure of the orbit space that the partial Bessel function is integrating on, and certain Bruhat decompositions.
\subsection{THE STRUCTURE OF $\textrm{GSpin}_{2n+1}$}
Let $H=\textrm{GSpin}_{2n+1}$. We want to understand its structure and its relationship with $H_D=\textrm{Spin}_{2n+1}$ and $\textrm{SO}_{2n+1}$.
We have an exact sequence
$$1\longrightarrow \mathbb{Z}/2\mathbb{Z}\longrightarrow Spin_{2n+1}\xrightarrow{\varphi} SO_{2n+1}\longrightarrow 1$$ where $\varphi$ is the covering map.
We fix the standard Borel subgroup $B=TU$ of $\textrm{SO}_{2n+1}$, and denote the corresponding Borel subgroup of $H$(resp. $H_D$) by $B_H=T_HU_H$(resp. $B_{H_D}=T_{H_D}U_{H_D}$). We see that $U\simeq U_{H_D}\simeq U_{H}$. 

As in the proof of Proposition 2.4 of [1], we start by fixing a basis $f_1,\cdots, f_n$ of the character lattice $X^*(T)$ of $\textrm{SO}_{2n+1}$. The root datum of $SO_{2n+1}$ can be given as follows:
$$X^*(T)=\mathbb{Z}f_1\oplus\mathbb{Z}f_2\oplus\cdots\oplus\mathbb{Z}f_n$$
$$\Delta=\{\gamma_1=f_1-f_2, \gamma_2=f_2-f_3,\cdots, \gamma_{n-1}=f_{n-1}-f_n, \gamma_n=f_n\}$$
$$X_*(T)=\mathbb{Z}f_1^*\oplus\mathbb{Z}f_2^*\oplus\cdots\oplus\mathbb{Z}f_n^*$$
$$\Delta^{\vee}=\{\gamma_1^{\vee}=f_1^*-f_2^*,\gamma_2^{\vee}=f_2^*-f_3^*,\cdots, \gamma_{n-1}^{\vee}=f_{n-1}^*-f_n^*,\gamma_n^{\vee}=2f_n^*\}.$$

Then the weight lattice $\textrm{P}_{\textrm{SO}_{2n+1}}=\{\lambda\in X^*(T): \langle\lambda,\gamma^{\vee}\rangle\in \mathbb{Z}, \forall \gamma\in \Phi\}$. If $\langle\Sigma c_if_i,\gamma_i^{\vee}\rangle\in \mathbb{Z}$, for $1\leq i \leq n-1$, this implies that $c_i-c_{i+1}\in\mathbb{Z}$, and if $i=n$, this implies that $2c_n\in \mathbb{Z}$.  Therefore $\textrm{P}_{\textrm{SO}_{2n+1}}=\{\Sigma c_if_i: c_i\in \frac{\mathbb{Z}}{2}, c_i-c_j\in \mathbb{Z}\},$ hence equal to the $\mathbb{Z}$-span of $f_1\cdots,f_n,\frac{f_1+f_2\cdots+f_n}{2}$. The group $\textrm{Spin}_{2n+1}$ is the simply connected double cover of $\textrm{SO}_{2n+1}$, hence its character lattice is equal to the root lattice of $\textrm{SO}_{2n+1}$, and its cocharacter lattice is the root lattice of type $C_n$, so we obtain the root datum of $H_D=\textrm{Spin}_{2n+1}$:
$$X^*(T_{H_D})=\mathbb{Z}f_1\oplus\mathbb{Z}f_2\oplus\cdots\oplus\mathbb{Z}f_n+\mathbb{Z}\frac{f_1\cdots+f_n}{2}$$
$$\Delta_{H_D}=\{\beta_1=f_1-f_2,\beta_2=f_2-f_3,\cdots, \beta_{n-1}=f_{n-1}-f_n,\beta_n=f_n\}$$
$$X_*(T_{H_D})=\mathbb{Z}\beta_1^{\vee}\oplus\mathbb{Z}\beta_2^{\vee}\oplus\cdots\oplus\mathbb{Z}\beta_n^{\vee}$$
$$\Delta^{\vee}_{H_D}=\{\beta_1^{\vee}=f_1^*-f_2^*,\beta_2^{\vee}=f_2^*-f_3^*,\cdots,\beta_{n-1}^{\vee}=f_{n-1}^*-f_n^*,\beta_n^{\vee}=2f_n^*\}.$$ We can realize $$H=\textrm{GSpin}_{2n+1}=(\textrm{GL}_1\times \textrm{Spin}_{2n+1})/ \{(1,1), (-1,\beta^{\vee}_n(-1))\}.$$
We add another character $f_0$ so that the character lattice of $\textrm{GL}_1\times \textrm{Spin}_{2n+1}$ is spanned by $f_0,f_1,f_2,\cdots,f_n,\frac{f_1+\cdots f_n}{2}$. Taking the ones that are trivial on $(-1,\beta^{\vee}(-1))$, we see that the character lattice of $\textrm{GSpin}_{2n+1}$ is spanned by $e_0=f_0+\frac{f_1+\cdots f_n}{2}, e_1=f_1,e_2=f_2,\cdots,e_n=f_n.$ Taking the dual basis, we have that the cocharacter lattice of $\textrm{GSpin}_{2n+1}$ is spanned by $e_0^*=f_0^*$, $e_1^*=f_1^*+\frac{f_0^*}{2},e_2^*=f_2^*+\frac{f_0^*}{2},\cdots, e_n^*=f_n^*+\frac{f_0^*}{2}$. Therefore the root datum of $H=\textrm{GSpin}_{2n+1}$ is given by:
$$X^*(T_H)=\mathbb{Z}e_0\oplus\mathbb{Z}e_1\oplus\cdots\oplus\mathbb{Z}e_n$$
$$\Delta_H=\{\alpha_1=e_1-e_2,\alpha_2=e_2-e_3,\cdots,\alpha_{n-1}=e_{n-1}-e_n,\alpha_n=e_n\}$$
$$X_*(T_H)=\mathbb{Z}e_0^*\oplus\mathbb{Z}e_1^*\cdots\oplus\mathbb{Z}e_n^*$$
$$\Delta^{\vee}_H=\{\alpha_1^{\vee}=e_1^*-e_2^*,\alpha_2^{\vee}=e_2^*-e_3^*,\cdots,\alpha_{n-1}^{\vee}=e_{n-1}^*-e_n^*,\alpha^*_n=2e_n^*-e_0^*\}.$$
It is easy to see that the three groups share the same root system, and we can identify $\alpha_i=\beta_i=\gamma_i$ for all $1\leq i\leq n$.

Take the Siegel Levi $M_H=M_{\theta}$ where $\theta=\Delta-\{\alpha_n\}$. We have $M_H\simeq \textrm{GL}_n\times \textrm{GL}_1$. Accordingly we will have that the Siegel Levi subgroup $M$ of $\textrm{SO}_{2n+1}$ is isomorphic to $\textrm{GL}_n$. Let $M_{H_D}$ be the corresponding Levi subgroup of $\textrm{Spin}_{2n+1}$. In the rest of this section we will realize $M_{H_D}$ inside $M_H$. It is crucial for the Bruhat decomposition in section 5.3. 

The covering map $\varphi$ induces a surjective map on the two corresponding Levi subgroups, then we have the following commutative diagram:

$$\begin{tikzcd}
\textrm{GL}_n\times \textrm{GL}_1\simeq & M_H \arrow[r, two heads,"pr"] & M &\simeq \textrm{GL}_n\\
& M_{H_D}\arrow[u, hook, "j"]\arrow[ru, two heads, "\varphi"]
\end{tikzcd}$$
where $j$ is the injection map and $pr$ is the projection of $M_H\simeq\textrm{GL}_n\times \textrm{GL}_1$ onto the $\textrm{GL}_n$-factor. Note that $j$ is induced from the surjective homomorphism of the character groups
$X^*(T_H)\twoheadrightarrow X^*(T_{H_D})$ by mapping $e_i$ to $f_i$ for $1\leq i\leq n-1$ and $e_0\mapsto f_0+\frac{f_1+\cdots+f_n}{2}$. Since $\textrm{Spin}_{2n+1}$ is simply connected, any element in its maxmal torus can be uniquely written as $t=\prod_{i=1}^n\beta^{\vee}(x_i)$. Any element in $T_H$ is of the form $\prod_{i=0}^n e_i^*(t_i)$. Hence if $t=\prod_{i=1}^n\beta_i^{\vee}(x_i)\in T_H$, since $\beta_i^{\vee}=\alpha_i^{\vee}$ for all $1\leq i\leq n$, we have $$t=\prod_{i=1}^n\beta_i^{\vee}(x_i)=\prod_{i=1}^n\alpha_i^{\vee}(x_i)=\prod_{i=1}^{n-1}(e_i^*-e_{i+1}^*)(x_i)\cdot (2e_n^*-e_0^*)(x_n)$$$$=e_1^*(x_1)e_2^*(\frac{x_2}{x_1})\cdots e_{n-1}^*(\frac{x_{n-1}}{x_{n-2}})e_n^*(\frac{x_n^2}{x_{n-1}})e_o^*(x_n^{-1}).$$

Therefore the injection $j: T_{H_D}\hookrightarrow T_H\simeq T_n\times T_1$ is given by
$\prod_{i=1}^n\beta_i^{\vee}(x_i)\mapsto \prod_{i=1}^n e^*_i(t_i)\mapsto e_1^*(x_1)e_2^*(\frac{x_2}{x_1})\cdots e_{n-1}^*(\frac{x_{n-1}}{x_{n-2}})e_n^*(\frac{x_n^2}{x_{n-1}})e_o^*(x_n^{-1})$ for all $x_i\in \mathbb{G}_m$. 
On the other hand, the covering map $\varphi$ induces a surjective map $\varphi:M_{H_D}\twoheadrightarrow M$. Since $\textrm{Spin}_{2n+1}$ and $\textrm{SO}_{2n+1}$ share the same roots, $\varphi$ is given by the surjective map $T_{H_D}\twoheadrightarrow T$, hence by the injection $X^*(T)\hookrightarrow X^*(T_{H_D})$, $f_i\mapsto f_i$, $1\leq i\leq n$, and in return by the surjective map $X_*(T_{H_D})\twoheadrightarrow X_*(T)$, $\beta_i^{\vee}\mapsto \gamma_i^{\vee}$, $1\leq i\leq n$. As a result, $T_{H_D}\twoheadrightarrow T$ can be explicitly written as  $$\prod_{i=1}^n\beta_i^{\vee}(x_i)\mapsto \prod_{i=1}^n\gamma_i^{\vee}(x_i)=\prod_{i=1}^{n-1}(f_i^*-f_{i+1}^*)(x_i)\cdot (2f_n^*)(x_n)$$
$$=f_1^*(x_1)f_2^*(\frac{x_2}{x_1})\cdots f_{n-1}^*(\frac{x_{n-1}}{x_{n-2}})f_n^*(\frac{x_n^2}{x_{n-1}}).$$ The kernel of this map is isomorphic to $\mathbb{Z}/2\mathbb{Z}$ with generator $\beta_n^{\vee}(-1)$. 

The above discussion shows that we have a commutative diagram on the corresponding tori:
$$\begin{tikzcd}
T_n\times T_1 \simeq&T_H \arrow[r, two heads,"pr"] & T_n \\
& T_{H_D} \arrow[u, hook, "j"]\arrow[ru, two heads, "\varphi"]
\end{tikzcd}$$
where $T_n$ and $T_1$ are the maximal tori of $\textrm{GL}_n$ and $\textrm{GL}_1$ respectively. Taking the isomorphisms on the root subgroups and Weyl groups of these groups, and using the Bruhat decomposition, we get the commutative diagram of Levi subgroups we discussed earlier. Moreover, from this we can also realize $M_{H_D}\subset M_H\simeq \textrm{GL}_n\times \textrm{GL}_1$ by 
$$M_{H_D}=\{m(g,a)\in M_H, \det(g)a^2=1\}^\circ,$$ where $\circ$ means taking the connected component.
\subsection{THE SPACE $Z_{M_H}^0U_{M_H}(F)\backslash N_H(F)$, ITS ORBIT REPRESENTATIVES AND MEASURE}

The partial Bessel functions that we are going to define will be integrating over this space. We proceed by first working on the space $U_{M_H}(F)\backslash N_H(F)$, then define $Z_{M_H}^0$  and consider its action after that. 

Let $H=\textrm{GSpin}_{2n+1}$, as an algebraic group defined over $F$. We fix the Borel subgroups $B_H=T_HU_H$, $B=TU$ of $H$ and $\textrm{SO}_{2n+1}$ respectively as in section 5.1. Notice that the Siegel parabolic $P_H=M_HN_H$ of $\textrm{GSpin}_{2n+1}$ share the same unipotent radical $N_H$ with the corresponding parabolic subgroup $P=MN$ of $\textrm{SO}_{2n+1}$. Let $U_{M_H}=U_H\cap M_H$, and $U_M=U\cap N$. We need to study the $U_{M_H}$-action on the $N_H$ by conjugation, both of which lie in the derived group of $H$. We have $U_{M_H}\simeq U_{M}$, and $N_H\simeq M$, and the action of $U_{M_H}$ on $N_H$ in $H=\textrm{GSpin}_{2n+1}$ is compatible with the $U_M$-action on $N$ in $\textrm{SO}_{2n+1}$. Therefore $U_{M_H}\backslash N_H\simeq U\backslash N$. Hence it suffices to study the $U_M$-action on $N$. 

We realize $\textrm{SO}_{2n+1}$ as $$SO_{2n+1}=\{h\in \textrm{GL}_{2n+1}:\leftidx{^t}h\tilde{J}h=\tilde{J} \},$$
where 
$\tilde{J}=
\begin{bmatrix}
\ \ & \ \ & J' \\
\ \ & 1 & \ \ \\
\leftidx{^t}J' & \ \ & \ \ \\
\end{bmatrix}$
and $J'=\begin{bmatrix}
\ \ & \ \ & \ \ & 1 \\
\ \ & \ \ & -1 \ \ \\
\ \ & \reflectbox{$\ddots$} & \ \ & \ \ \\
(-1)^{n-1} & \ \ & \ \ & \ \ \\
\end{bmatrix}.$ 
An easy calculation shows that the $M
=\{m=m(g)=\begin{bmatrix}
g & \ \ & \ \ \\
\ \ & 1 & \ \ \\
\ \ & \ \ & J'\leftidx{^t}g^{-1}J'^{-1} \\
\end{bmatrix}:g\in \textrm{GL}_n\}$. Consequently
$U_{M}=\{ \begin{bmatrix}
u & \ \ & \ \ \\
\ \ & 1 & \ \ \\ 
\ \ & \ \ & J'\leftidx{^t}u^{-1}J'^{-1}\\
\end{bmatrix}: u\in U_n\}$, where $U_n$ is the unipotent radical of the standard Borel subgroup of $\textrm{GL}_n$ consists of upper triangular unipotent matrices.
And the unipotent radical of $P=MN$ is 
$$N=\{n=n(X,\alpha)=\begin{bmatrix}
I & \alpha & X\\
\ \ & 1 & -\leftidx{^t}\alpha J' \\
\ \ & \ \ & I \\
\end{bmatrix}: X\leftidx{^t}J'+J' \leftidx{^t}X+\alpha\leftidx{^t}\alpha=0 \ \ (*)\}$$
 A simple calculation shows that the conjugate action of $U_{M}(F)$ on $N(F)$ is equivalent to 
$$X\mapsto uXJ'\leftidx{^t}uJ'^{-1}, \alpha\mapsto u\alpha.\ \ \cdots\cdots (a)$$
Let $Z=X\leftidx{^t}J' + \frac{\alpha\leftidx{^t}\alpha}{2}$, then $(*)\Leftrightarrow Z+\leftidx{^t}Z=0.$ Now $X=(Z-\frac{\alpha\leftidx{^t}\alpha}{2})\leftidx{^t}J'^{-1}=(Z-\frac{\alpha^t\alpha}{2})J'$. So $n=n(Z,\alpha)\in N_H(F)$ is therefore parameterized by $Z\in Sk_n(F),$ the set of skew-symmetric matrices with $F$-coefficients, and $\alpha\in F^n$. The action (a) translates into 
$$Z\mapsto uZ\leftidx{^t}u, \alpha\mapsto u\alpha. \ \ \cdots\cdots (a'),$$ since if we denote $X'=uXJ'\leftidx{^t}uJ'^{-1}, \alpha'=u\alpha$, then the corresponding $$Z'=X'^tJ'+\frac{\alpha'\leftidx{^t}{\alpha'}}{2}=(uXJ'\leftidx{^t}uJ'^{-1})^tJ'+\frac{u\alpha\leftidx{^t}{\alpha'}\leftidx{^t}u}{2}=u(X^tJ'+\frac{\alpha^t\alpha}{2})^tu=uZ^tu.$$
Now it is equivalent to find the orbit representatives for the action of $U_n(F)$ on $Sk_{n+1}(F)$ because 
$Sk_n(F)\times F^n\longrightarrow Sk_{n+1}(F)$ defined by
$(Z,\alpha)\mapsto \begin{bmatrix}Z & \alpha\\ 
-^t\alpha & 0 \\
\end{bmatrix}$
is a homeomorphism of p-adic manifolds. If we identify $U_n(F)$ with its image in $U_{n+1}(F)$ by the embedding $u \mapsto \begin{bmatrix}
u & \ \ \\
\ \ & 1 \\
\end{bmatrix}$, we also have  
$\begin{bmatrix}
u & \ \ \\
\ \ & 1 \\
\end{bmatrix}
\begin{bmatrix}
Z & \alpha \\
-\leftidx{^t}\alpha & 0
\end{bmatrix}
\begin{bmatrix}
\leftidx{^t}u & \ \ \\
\ \ & 1 \\
\end{bmatrix}
=\begin{bmatrix}
uZ\leftidx{^t}u & u\alpha\\
-\leftidx{^t}(u\alpha) & 0 \\
\end{bmatrix}.$
So it suffices to find orbit representatives of the action of $U_n(F)$ on $Sk_{n+1}(F)$ by $u.\tilde{Z}=\begin{bmatrix}
u & \ \ \\ 
\ \ & 1 \\
\end{bmatrix}\tilde{Z}
\begin{bmatrix}
\leftidx{^t}u & \ \ \\
\ \ & 1 \\
\end{bmatrix}$
 where $u\in U_n(F)$ and $\tilde{Z}\in Sk_{n+1}(F).$
For our concern it suffices to find such orbit representatives for an open dense subset of $N(F)$ under the p-adic topology. We will define this open dense subset inductively. Let's begin with a few lemmas:
\begin{lemma}
Let $\varphi: M\rightarrow N$ be a surjective submersion of manifolds. If we have an open dense subset $V\subset N $, then $U=\varphi^{-1}(V)$ is open dense in M. 
\end{lemma}

\begin{proof}
It suffices to show this locally. Thus without loss of generality, assume $M\simeq F^m$ and $N\simeq F^n$ with $m\ge n$, and $\varphi=pr: F^m\rightarrow F^n$ is the projection map. Then if $V$ is dense in $F^n$, we have
$\varphi^{-1}(V)=pr^{-1}(V)\simeq V\times F^{m-n}$.
So 
$\overline{\varphi^{-1}(V)}\simeq \overline{V\times F^{m-n}}\simeq \overline{V}\times F^{m-n}\simeq F^n\times F^{m-n}\simeq F^m\simeq M$. Since $\overline{\varphi^{-1}(V)}\subset M$, we have $\overline{\varphi^{-1}(V)}=M$.
\end{proof}

\begin{lemma}
Let $\varphi_i: Sk_{i+1}(F)\longrightarrow Sk_i(F)$ be defined by 
$Z=\begin{bmatrix}
Z' & \beta \\
-^t\beta & 0 \\
\end{bmatrix}\mapsto u_i Z' {^t}{u_i}$
where $u_i=\begin{bmatrix}
I_{i-1} & \gamma \\
0 & 1 \\
\end{bmatrix}$, $\beta=\begin{bmatrix}
\beta' \\
b_i \\
\end{bmatrix}$ with $b_i\neq 0$, $I_{i-1}$ denotes the $(i-1)\times (i-1)$ identity matrix and $\gamma=-b_i^{-1}\beta'$. Then
$\varphi_i$ is a surjective submersion of p-adic manifolds.
\end{lemma}
\begin{proof}
Write $Z'=\begin{bmatrix}
Z'' & \alpha' \\
-^t\alpha' & 0 \\
\end{bmatrix}$ with $Z''\in Sk_{i-1}(F)$.
Also notice that 
$u_iZ'^tu_i=\begin{bmatrix}
I_{i-1} & \gamma \\
0 & 1\\
\end{bmatrix}
\begin{bmatrix}
Z'' & \alpha' \\
-^t\alpha' & 0 \\
\end{bmatrix}
\begin{bmatrix}
I_{i-1} & 0 \\
^t\gamma & 1 \\
\end{bmatrix}=\begin{bmatrix}
Z''-\gamma^t\alpha'+\alpha'^t\gamma & \alpha'\\
-^t\alpha' & 0 \\
\end{bmatrix}$.
The map 
$$Sk_{i-1}(F)\times F^{i-1}\times F^{i-1}\times F^*\longrightarrow Sk_{i-1}(F)\times F^{i-1}$$
$$(Z'',\alpha',\beta',b_i)\mapsto (Z''-\gamma^t\alpha'+\alpha'^t\gamma, \alpha')$$
is a submersion because the Jacobian of this map contains an $i\times i$ identity matrix, due to that the coefficient of $Z''$ is 1 on both hand sides. The surjectivity is clear by the definition of $\varphi_i$. 
\end{proof}
\begin{lemma}
Denote $V_i=\{Z\in Sk_i(F): z_{i-1,i}\neq 0\}$ and let
$$V=\{Z\in Sk_{n+1}(F): \varphi_{n-i}\circ \varphi_{n-i+1}\circ\cdots \circ \varphi_n(Z)\in V_{n-i-1}, \forall 0\leq i \leq n-2 \}$$
where $\varphi_i: Sk_{i+1}(F)\longrightarrow Sk_i(F)$ as in Lemma 5.2, which is a surjective submersion. Then V is open dense in $Sk_{n+1}(F).$
\end{lemma}
\begin{proof} By the previous two lemmas, each $V_i$ is open dense in $Sk_i(F)$. Since the composition of surjective submersions is still a surjective submersion, the topology of $Sk_i(F)\hookrightarrow  Sk_{i+1}(F)$ is the induced topology. So the subset $V$, which is defined inductively, is a finite intersection of open dense subsets, therefore open dense.
\end{proof}
Based on the above discussion, we obtain
\begin{proposition} Let 
$N(F)'=\{n=\begin{bmatrix}
I & \alpha & (Z-\frac{\alpha^t\alpha}{2})J' \\
\ \ & 1 & -^t\alpha J' \\
\ \ & \ \ & I \\
\end{bmatrix}:\begin{bmatrix}
Z & \alpha \\
-^t\alpha & 0 \\
\end{bmatrix}\in V \}.$ Then $N(F)'\subset N(F)$ is open dense. Moreover, for $\forall n(Z,\alpha)\in N(F)'$, $\exists u\in U_n(F)$, such that 
$u\cdot n(Z,\alpha)=n(uZ^tu,u\alpha)$
where $$\begin{bmatrix}
uZ{^tu} & u\alpha \\
-^t(u\alpha) & 0 \\
\end{bmatrix}=\begin{bmatrix}
0 & a_1 & \ \ & \ \ & \ \ \\
-a_1 & 0 & \ \ & \ \ & \ \ \\
\ \ & \ \ & \ddots & \ \ & \ \ \\
\ \ & \ \ & \ \ & 0 & a_n \\
\ \ & \ \ & \ \ & -a_n & 0 \\
\end{bmatrix}$$
with $a_i\in F^*$. This gives a set of orbit representatives for the adjoint action of $U_{M}(F)\simeq U_n(F)$ on $N(F)'.$
\end{proposition}
\begin{proof}
First, by the previous argument, $N(F)'$ is open dense in $N(F)$ under the p-adic topology. Now take $u_n$ as in Lemma 5.2 and write $\tilde{Z}=\begin{bmatrix}
Z & \alpha \\
-^t\alpha & 0 \\
\end{bmatrix}$. Then we have $u_nZ^tu_n=\varphi_n(\tilde{Z})\in V_n$ and
$u_n\alpha=[0,\cdots,0,a_n]^t$ with $a_n\neq 0$ by the construction of $N(F)'$.
Now $u_nZ^tu_n\in V_n\subset Sk_n(F)$, by induction on $n$ we end up with some $u\in U_n(F)$ as stated in the lemma.

Let $R$ denote this orbit representatives, as we saw above it is homeomorphic to $(F^*)^n$. So we have a continuous surjective map:
$U_n(F)\times R\longrightarrow V$ given by
$(u,
(a_1,\cdots,a_n))\mapsto \begin{bmatrix}
u & \ \ \\
\ \ & 1 \\
\end{bmatrix}
\begin{bmatrix}
0 & a_1 & \ \ & \ \ & \ \ \\
-a_1 & 0 & \ \ & \ \ & \ \ \\
\ \ & \ \ & \ddots & \ \ & \ \ \\
\ \ & \ \ & \ \ & 0 & a_n \\
\ \ & \ \ & \ \ & -a_n & 0 \\
\end{bmatrix}
\begin{bmatrix}
^tu &  \ \ \\
\ \ & 1 \\
\end{bmatrix}$.
The map is clearly continuous. It has an inverse. In fact, the inverse map is just given by the process of finding the orbit representatives as we showed above, which is apparently continuous since all maps arising are again just matrix multiplications. Hence to show it is a homeomorphism, we only need to show that any two matrices of this form lie in different orbits. This follows easily by induction on the size of the matrix. 
Indeed, suppose $u=\begin{bmatrix}
u'  & \gamma \\
\ \ & 1 \\
\end{bmatrix}$ and let $\tilde{Z}=\begin{bmatrix}
0 & a_1 & \ \ & \ \ & \ \ \\
-a_1 & 0 & \ \ & \ \ & \ \ \\
\ \ & \ \ & \ddots & \ \ & \ \ \\
\ \ & \ \ & \ \ & 0 & a_n \\
\ \ & \ \ & \ \ & -a_n & 0 \\
\end{bmatrix}=\begin{bmatrix}
\tilde{Z}_1 & \alpha \\
-^t\alpha & 0 \\
\end{bmatrix}$
with $\alpha=[0,\cdots,0,a_n]^t$ and $\tilde{Z}_1$ is the principal $(n-1)\times(n-1)$ block of $\tilde{Z}$.
Now suppose $\tilde{Z}'$ is another such matrix with entries $a_i'$ and $\begin{bmatrix}
u  & \ \ \\
\ \ & 1 \\
\end{bmatrix}\tilde{Z}\begin{bmatrix}
^tu & \ \ \\
\ \ & 1 \\
\end{bmatrix}=\tilde{Z}',$ and similarly we define $\tilde{Z}'_1$ and $\alpha'$.
This implies that $u\alpha=\alpha'$, hence $u$ has to be the form $u=\begin{bmatrix}
u' & 0 \\
0 & 1 \\
\end{bmatrix}$. This gives that $\begin{bmatrix}
u' & \ \ \\
\ \ & 1 \\
\end{bmatrix}\tilde{Z}_1\begin{bmatrix}
^tu' & \ \ \\
\ \ & 1 \\
\end{bmatrix}=\tilde{Z}'_1$
where $\tilde{Z}_1$ and $\tilde{Z}'_1$ are of the same form as $\tilde{Z}$ and $\tilde{Z}'$ respectively, but of strictly smaller size, so by induction hypothesis, we derive that $u'=I_{n-1}$, which also means that $u=I$. This forces $\tilde{Z}=\tilde{Z}'$, so $a_i=a'_i$ for $1\leq i \leq n$.

Moreover, the action is simple, i.e., if $u\cdot Z=Z$, then $u=I$. To see this, just take $\tilde{Z}'=\tilde{Z}$ in the above argument, and a similar process gives $u=I$.
\end{proof}

Now we have a homeomorphism $U_{M}(F)\times R\simeq N(F)'\subset N(F)$ with $N(F)'\subset N(F)$ open dense. Recall that we have isomorphisms of algebraic groups $U_{M_H}\simeq U_M$, $N_H\simeq N$, given by identifying the corresponding root subgroups. So we obtain homeomorphisms of p-adic manifolds: $U_{M_H}(F)\simeq U_M(F)$ and $N_H(F)\simeq N(F) $. Denote the homeomorphic image of $N(F)'$ in $N_H(F)$ by $N_H(F)'$, then it's clear that $N_H(F)'\subset N_H(F)$ is also open dense. Moreover, the $U_{M_H}(F)$-action on $N_H(F)$  is compatible with the $U_M(F)$-action on $N(F)$. From now on we identify the p-adic manifolds: $U_{M_H}(F)\simeq U_M(F)$, $N_H(F)\simeq N(F)$, $N_H(F)'\simeq N(F)'$, and $U_{M_H}(F)\backslash N_H(F)\simeq U_M(F)\backslash N(F)$. We also identify $R$ as the orbit space representatives of $U_{M_H}(F)\backslash N_H(F)$.

Now let's discuss the invariant measure on the orbit space. Any measurable function $f$ on $N_H(F)$ can be viewed as a function on $U_{M_H}(F)\times R$. Let $du$ and $dn$ the Haar measure on $U_{M_H}(F)$ and $N_H(F)$ respectively. Let $da$ be the measure on $R$ such that the integration formula
$\int_{U_{M_H(F)}}\int_Rf(u\cdot a) du da= \int_{N_H(F)}f(n)dn$ holds. We also need to construct an invariant measure on $R$. When the dimension $n=2$, $U_{M_H}(F)\simeq U_2(F)=\{\begin{bmatrix}
1 & x \\
\ \ & 1 \\
\end{bmatrix}:x\in F\}\simeq F$,
$R\simeq \{\begin{bmatrix}
0 & a_1 & 0 \\
-a_1 & 0 & a_2\\
0 & -a_2 & 0
\end{bmatrix}: a_1, a_2\in F^*\}\simeq (F^*)^2$,
and 
$N_H(F)\simeq \{n(Z,\alpha):Z\in Sk_2(F),\alpha\in F^2\}\simeq F^3$. The action of $U_2(F)$ on $R$ is give by 
$$\begin{bmatrix}
1 & x & \ \ \\
0 & 1 & \ \ \\
\ \ & \ \ & 1 \\
\end{bmatrix}\begin{bmatrix}
0 & a_1 & 0 \\
-a_1 & 0 & a_2 \\
0 & -a_2 & 0 \\
\end{bmatrix}
\begin{bmatrix}
1 &  0 & \ \ \\
x & 1 & \ \ \\
\ \ & \ \ & 1 \\
\end{bmatrix}=\begin{bmatrix}
0 & a_1 & a_2x \\
-a_1 & 0 & a_2 \\
-a_2x & a_2 & 0 \\
\end{bmatrix}$$
So $$F\times (F^*)^2\simeq U_{M_H}(F)\times R\longrightarrow N_H(F)\simeq F^3$$ is given by
$$(x, a_1, a_2)\mapsto (a_1, a_2x, a_2).$$
So we can write $f(u\cdot a)=f(a_1,a_2x,a_2)$. Let $da=da_1\vert a_2 \vert da_2$, then
$$\int_{U_{M_H}(F)}\int_R f(u\cdot a)du da=\int_{(a_1,a_2)\in (F^*)^2}\int_{x\in F}f(a_1, a_2x, a_2)dx da_1\vert a_2 \vert da_2.$$
Let $x'=a_2x, a_1'=a_1,a_2'=a_2$, then $dx'=\vert a_2 \vert dx$. Then the above integral
$$=\int_F\int_{(F^*)^2}f(a'_1,x',a_2')\frac{dx'}{\vert a_2'\vert}da_1' \vert a_2'\vert da_2'=\int_F\int_{(F^*)^2}f(a_1',x',a_2')dx'da_1'da_2'$$

$$=\int_{F^3}f(a_1',x',a_2')dx'da_1'da_2'=\int_{N_H(F)}f(n)dn.$$
It is straightforward to show by induction on the dimension $n$ that the invariant measure on the space of orbits $R$ is given by $da=\prod_{i=1}^n\vert a_i\vert^{i-1} da_i=\prod_{i=1}^n\vert a_i \vert^i d^{\times}a_i.$

Next, we define $Z^0_{M_H}$ and consider its action on $U_{M_H}(F)\backslash N_H(F)$. 
\begin{lemma}$H=\textrm{GSpin}_{2n+1}$. Let $Z_H$ and $Z_{M_H}$ denote the centers of $H$ and $M_H$ respectively, then
$Z_H=\{e^*_0(\lambda): \lambda\in GL_1\}$ and 
$Z_{M_H}=\{e^*_0(\lambda)e^*_1(\mu)\cdots e^*_n(\mu):\lambda,\mu\in GL_1\}$. There exists an injection:
$\alpha^{\vee}:F^{\times}\hookrightarrow Z_H\backslash Z_{M_H} $ such that $\alpha(\alpha^{\vee}(t))=t$ for $\forall t\in F^*$.
\end{lemma}
\begin{proof} The structure of $Z_H$ and $Z_{M_H}$ follows from Proposition 2.3 of [2]. For the second part of the lemma,
take $\alpha^{\vee}:t\mapsto Z_H(e_1^*(t)\cdots e^*_n(t))$.
Then $\alpha^{\vee}$ is an injection, since if $Z_H(e_1^*(t)\cdots e_n^*(t))=Z_H$, then $e^*_1(t)\cdots e_n^*(t)\in Z_H,$
therefore $e^*_1(t)\cdots e_n^*(t)=e^*_0(\lambda)$ for some $\lambda\in GL_1$, but the cocharacters are independent since they form a basis for the cocharacter lattice, it forces
$e_1^*(t)=e^*_2(t)=\cdots=e^*_n(t)=e^*_0(\lambda)=1,$
this implies $t=1$. Moreover, since $\alpha=\alpha_n=e_n$, we have
$\alpha(\alpha^{\vee}(t))=e_n(e^*_1(t)\cdots e_n^*(t))=e_n(e^*_n(t))=t$.
\end{proof}

Let $Z_{M_H}^0=\{ \alpha^{\vee}(t):t\in F^*\}$ be the image of the map $\alpha^{\vee}$ we just constructed. For $z=\alpha^{\vee}(t)=\prod_{i=1}^ne^*_i(t)$ and $n(Z,\alpha)\in N_H(F)$ as before it's easy to see that $$\alpha^{\vee}(t) n(Z,\alpha)\alpha^{\vee}(t)^{-1}= n(t^2Z,t\alpha).$$ Therefore the $Z_{M_H}^0$-action on $N_H(F)$ induces an action
$Z_{M_H}^0\times R\longrightarrow R$, given by
$(t,(a_1,\cdots, a_n))\mapsto (t^2a_1,\cdots t^2 a_{n-1},t a_n).$

We also need to define a measure on the space of orbits $R'$ of $Z_{M_H}^0U_{M_H}\backslash N_H$ such that it is compatible with the measure on $R$ we constructed. We can take $a_n=1$ to identify $R'$ with $\{(a_1',\cdots,a_{n-1}',1):a_i'\in F^*\}$. By the measure on $R$ we can see that the measure on $R'$ is of this form $da'=\prod_{i=1}^{n-1}\vert a'_i \vert^{k_i}da'_i$ with $k_i\in \mathbb{Z}.$ Recall that $\rho$ is the half of the sum of positive roots in $N_H$, as we computed before
$\rho=\frac{n}{2}\sum_{i=1}^n e_i.$ So for $z=\alpha^{\vee}(t)$, we have $q^{\langle 2\rho, H_{M_H}(z)\rangle}=\vert n\sum_{i=1}^n e_i(\prod_{i=1}^n e^*(t))\vert=\vert t\vert^{n^2}$.
Then we should have $$\int_Rf(a)da=\int_{Z_{M_H}^0}\int_{R'}f(z\cdot a')q^{\langle 2\rho, H_{M_H}(z) \rangle}da'dz$$
$$=\int_{F^*\times R'}f(t^2a_1',\cdots,t^2a_{n-1}',t)\vert t\vert^{n^2-1}\prod_{i=1}^{n-1}\vert a'_i\vert^{k_i}da'_idt.$$
Let $a_i=t^2a'_i$ for $1\leq i\leq n-1$, and $a_n=t$. Then $da'_i=\vert t\vert^{-2}da_i$ and $da_n=dt$. So the above integral
$$=\int_{F^*\times R'}f(a_1,\cdots,a_{n-1},a_n)\vert a_n \vert^{n^2-1}\prod_{i=1}^{n-1}\vert t^{-2}a_i\vert^{k_i}\vert a_n\vert^{-2(n-1)}da_ida_n.$$
On the other hand, we should also have $$\int_{(F^*)^n}f(a_1, \cdots, a_n)\prod_{i=1}^n\vert a_i\vert^{i-1}da_i=\int_{R}f(a)da.$$
By comparing this with the above discussion we can see that it forces each $k_i=i-1$.
This means that $$da'=\prod_{i=1}^{n-1}\vert a_i'\vert^{i-1}da'_i$$ gives the desired measure on the space of orbits $R'$ of $Z_{M_H}^0U_{M_H}(F)\backslash N_H(F)$.

\subsection{A BRUHAT DECOMPOSITION}
Theorem 6.2 of [19] allows us to write the local coefficients as the Mellin transform of some partial Bessel functions, whose definitions rely on a Bruhat decomposition. We will study the Bruhat decomposition in this section.

As before $H=\textrm{GSpin}_{2n+1}$. Let $w_H$ and $w_{\theta}$ be the long Weyl group element of $H$ and $M_{\theta}=M_H$, respectively. We denote the length of $w$ by $l(w)$. Then $l(w_H)=n^2$ and $l(w_\theta)=\frac{n(n-1)}{2}$, since in general $l(w)$ is the number of positive roots that are mapped to negatives ones by $w$. Their reduced decompositions can be given as follows:
$$w_H=w_{\alpha_{n-1}}(w_{\alpha_{n-2}}w_{\alpha_{n-1}})\cdots (w_{\alpha_{2}}\cdots w_{\alpha_{n-1}})(w_{\alpha_1}\cdots w_{\alpha_{n-1}})$$
$$\cdot w_{\alpha_n}(w_{\alpha_{n-1}}w_{\alpha_{n}})\cdots(w_{\alpha_2}\cdots w_{\alpha_n})(w_{\alpha_1}\cdots w_{\alpha_n})$$
and 
$$w_\theta=w_{\alpha_{n-1}}(w_{\alpha_{n-2}}w_{\alpha_{n-1}})\cdots (w_{\alpha_2}\cdots w_{\alpha_{n-1}})(w_{\alpha_1}\cdots w_{\alpha_{n-1}})$$

In general there is a canonical way to pick the Weyl group representative $\dot{w}$ of $w\in W$ by a given splitting $\{u_{\alpha}: \mathbb{G}_m\rightarrow U_{\alpha}\}_{\alpha\in \Phi^+}$: Fix a reduced decomposition $w=\prod_{\alpha}w_{\alpha}$ with each $w_\alpha$ a simple reflection, there is a unique $y_\alpha\in \mathbb{G}_m$ such that $w_{\alpha}(1)w_{-\alpha}(y_\alpha) w_\alpha(1)$ normalizes the maximal torus. For each $w_\alpha$ pick  $\dot{w}_\alpha=u_{\alpha}(1)u_{-\alpha}(y_\alpha)u_{\alpha}(1)$ and let $\dot{w}=\prod_{\alpha}\dot{w}_\alpha$. This makes each $\dot{w}_\alpha$ the image of $\begin{bmatrix}
\ \ & 1 \\
-1 & \ \ \\
\end{bmatrix}$ under the homomorphism $\textrm{SL}_2\rightarrow H$ attached to the $\mathfrak{sl}_2$-triple $\{X_\alpha, H_\alpha, H_{-\alpha}\}$.

One can compute that we should pick $\dot{w}_{\alpha_i}=u_{\alpha_i}(1)u_{-\alpha_i}(-1)u_{\alpha_i}(1)$ for $1\leq i\leq n-1$ and $\dot{w}_{\alpha_n}=u_{\alpha_n}(1)u_{-\alpha_n}(-2)u_{\alpha_n}(1).$ Now we pick $\dot{w}_H$ and $\dot{w}_{\theta}$ as in the above process and let $\dot{w}_0=\dot{w}_H \dot{w}_\theta^{-1}$. 
Moreover, given $\psi :F\rightarrow \mathbb{C}^*$ a non-trivial additive character, recall that we can define a generic character of $U_H(F)$, which is still denoted by $\psi$, by setting
$\psi(u)=\psi(\sum_{\alpha\in \Delta}u_{\alpha})$. We can identify $u=m(u',1)\in U_{M_H}(F)\simeq U_n(F)$ with $m(u')\in U_{M}$, where $u'\in U_n$. Then a straightforward calculation shows that the generic character $\psi$ is compatible with the choice of the Weyl group representative $\dot{w}_0$, i.e., we have $\psi(\dot{w}_0u \dot{w}_0^{-1})=\psi(u)$.

Let $\overline{N}_H=\dot{w}_H N_H \dot{w}_H^{-1}$. We need to find some open dense subset of $N_H(F)$ such that the Bruhat decomposition 
$\dot{w}_0^{-1}n=mn'\bar{n}$
holds for $n$ lying in this open dense subset, where $m\in M_H$, $n'\in N_H$ and $\bar{n}\in \overline{N}_H$.

Observe that in this decomposition $m$ is uniquely determined by $n$. Since $n, n'$ and $\overline{n}$ are all in the derived group $H_D=\textrm{Spin}_{2n+1}$, so is $m$. Instead of doing this directly in $\textrm{Spin}_{2n+1}$(or in $\textrm{GSpin}_{2n+1}$), we first do it in $\textrm{SO}_{2n+1}$. We identify the Weyl group elements in $H=\textrm{GSpin}_{2n+1}$ and $\textrm{SO}_{2n+1}$. A direct computation in $\textrm{SO}_{2n+1}$ shows that we should pick
$$\dot{w}_H=\begin{bmatrix}
\ \ & \ \ & (-\frac{1}{2})J' \\
\ \ & (-1)^n & \ \ \\
(-2)\cdot \leftidx{^t}J' & \ \ & \ \ \\
\end{bmatrix},\dot{w}_{\theta}=\begin{bmatrix}
J' & \ \ & \ \ \\
\ \ & 1 & \ \ \\
\ \ & \ \ & J' \\
\end{bmatrix}.$$ Hence $\dot{w}_0=\dot{w}_H\dot{w}_\theta^{-1}=\begin{bmatrix}
\ \ & \ \ & (-\frac{1}{2})I \\
\ \ & (-1)^n & \ \ \\
(-1)^n2I & \ \ & \ \ \\
\end{bmatrix}$. Therefore $$\dot{w}_0^{-1}=\begin{bmatrix}
\ \ & \ \ & (-1)^n\frac{1}{2}I \\
\ \ & (-1)^n & \ \ \\
-2I & \ \ & \ \ \\
\end{bmatrix}=\begin{bmatrix}
(-\frac{1}{2})I & \ \ & \ \ \\
\ \ & 1  \ \ \\
\ \ & \ \ & -2I \\
\end{bmatrix}\cdot \begin{bmatrix}
\ \ & \ \ & (-1)^{n-1}I \\
\ \ & (-1)^n & \ \ \\
I & \ \ & \ \ \\
\end{bmatrix}.$$ Let $\tilde{w}_0^{-1}=\begin{bmatrix}
\ \ & \ \ & (-1)^{n-1}I \\
\ \ & (-1)^n & \ \ \\
I & \ \ & \ \ \\
\end{bmatrix}$, then the above formula shows that $$\dot{w}_0^{-1}=m(-\frac{1}{2}I)\tilde{w}_0^{-1}$$ To simplify our computation, let's first compute the decomposition $\tilde{w}_0^{-1}n=m(g)n'\overline{n}$ in $\textrm{SO}_{2n+1}$.
We have
$$\tilde{w}_0^{-1}n=\begin{bmatrix}
\ \ & \ \ & (-1)^{n-1}I \\
\ \ & (-1)^n & \ \ \\
I & \ \ & \ \ \\
\end{bmatrix}
\begin{bmatrix}
I & \alpha & X \\
\ \ & 1 & -\leftidx{^t}{\alpha}J' \\
\ \ & \ \ & I \\
\end{bmatrix}=
\begin{bmatrix}
\ \ & \ \ & (-1)^{n-1}I \\
\ \ & (-1)^n & (-1)^{n-1}\leftidx{^t}{\alpha}J' \\
I & \alpha & X \\
\end{bmatrix}$$
and if we assume $m(g)=\begin{bmatrix}
g & \ \ & \ \ \\
\ \ & 1 & \ \ \\
\ \ & \ \ & J'\leftidx{^t}g^{-1}J'^{-1}\\
\end{bmatrix}$ with $g\in GL_n$,  
$n'=\begin{bmatrix}
I & \beta & Y' \\
\ \ & 1 & -\leftidx{^t}{\beta}J' \\
\ \ & \ \ & I \\
\end{bmatrix}$
and $\bar{n}=\begin{bmatrix}
I & \ \ & \ \ \\
(-1)^{n}2\leftidx{^t}{\gamma} & 1 & \ \ \\
4\leftidx{^t}{J'}Z\leftidx{^t}{J'} & (-1)^{n-1}2\leftidx{^t}J'\gamma & I\\
\end{bmatrix}$. Let $\gamma'=-2\gamma$ and $Z'=4Z$, then 
$$m(g)n'\bar{n}=\begin{bmatrix}
g & \ \ & \ \ \\
\ \ & 1 & \ \ \\
\ \ & \ \ & J'\leftidx{^t}g^{-1}J'^{-1}\\
\end{bmatrix}
\begin{bmatrix}
I-(-1)^n\beta\leftidx{^t}{\gamma'}+Y'\leftidx{^t}J' Z' \leftidx{^t}J' & \beta+(-1)^nY' \leftidx{^t}J'\gamma' & Y' \\
(-1)^{n-1}\leftidx{^t}{\gamma'}-\leftidx{^t}{\beta}Z'\leftidx{^t}J' & 1-(-1)^n\leftidx{^t}{\beta}\gamma' & -\leftidx{^t}{\beta}J' \\
\leftidx{^t}J' Z' \leftidx{^t}J' & (-1)^n\leftidx{^t}J'\gamma' & I \\
\end{bmatrix}$$
$$=\begin{bmatrix}
g(I-(-1)^n\beta\leftidx{^t}{\gamma'}+Y'\leftidx{^t}J' Z' \leftidx{^t}J') & g(\beta+(-1)^nY' \leftidx{^t}J'\gamma') & g Y' \\
(-1)^{n-1}\leftidx{^t}{\gamma'}-\leftidx{^t}{\beta}Z'\leftidx{^t}J' & 1-(-1)^n\leftidx{^t}{\beta}\gamma' & -\leftidx{^t}{\beta}J' \\
(-1)^{n-1}J'{^tg^{-1}} Z' {^tJ'} & -J'{^tg^{-1}}\gamma' & J'{^tg^{-1}}J'^{-1} \\
\end{bmatrix}.$$

Assume that $\det(X)\neq 0$, then the equality $\tilde{w}_0^{-1}n=m(g)n'\bar{n}$ in our case is equivalent to the following conditions:

(1) $I-(-1)^n\beta{^t\gamma'}+Y'{^tJ'}Z'{^tJ'}=0$;
(2) $\beta+(-1)^nY'{^tJ'}\gamma'=0;$
(3) $gY'=I;$
(4)$(-1)^{n-1}{^t\gamma'}-{^t\beta}Z'{^tJ'}=0;$
(5) $1-(-1)^n{^t\beta}\gamma'=(-1)^n;$
(6) $(-1)^{n-1}{^t\alpha}J'=-{^t\beta}J';$
(7) $(-1)^{n-1}J'{^tg^{-1}}Z'{^tJ'}=I;$
(8) $-J'{^tg^{-1}}\gamma'=\alpha;$
(9) $J'{^tg^{-1}}J'^{-1}=X.$

We also recall that by the definition of $N_H(F)$, we also have

(i) $X{^tJ'}+J'{^tX}+\alpha{^t\alpha}=0\Longleftrightarrow {^tJ'}X+{^tX}J'+{^tJ'}\alpha{^t\alpha}J'=0$

We need to simplify this first. Note that

$(9)\Longleftrightarrow g=J'{^tX^{-1}}J'^{-1};$
$(6)\Longleftrightarrow \beta=(-1)^n\alpha$; 
$(7)\Longleftrightarrow Z'={^tJ'}X^{-1}J'^{-1};$

$(8)\Longleftrightarrow \gamma'=-{^tJ'}X^{-1}\alpha;$
$(3)\Longleftrightarrow Y'=g^{-1}=J'{^tX}J'^{-1}.$

Next, we have $(5)\Longleftrightarrow (-1)^n-{^t\beta}\gamma'=1\Longleftrightarrow {^t\beta}\gamma'=(-1)^n-1\Longleftrightarrow {^t\alpha}{^tJ'}X^{-1}\alpha=(-1)^n-1$
We call this formula (ii).

Also we have
$(4)\Longleftrightarrow (-1)^{n-1}\gamma'-J'{^tZ'}\beta=0\Longleftrightarrow (-1)^{n-1}(-{^tJ'}X^{-1}\alpha)-J'({^tJ'^{-1}}{^tX^{-1}}J')(-1)^n\alpha=0
\Longleftrightarrow {^tJ'}X^{-1}\alpha-J'J'{^tX^{-1}}J'\alpha=0\Longleftrightarrow ({^tJ'}X-(-1)^{n-1}{^tX}J')X^{-1}\alpha=0$.
We call the last formula (4').

Also notice that 
$(2)\Longleftrightarrow (-1)^n\alpha+(-1)^n(J'{^tXJ'^{-1}}){^tJ'}(-{^tJ'}X^{-1}\alpha)=0\Longleftrightarrow \alpha+J'{^tX}(-1)^{n-1}(-{^tJ'}X^{-1}\alpha)=0\Longleftrightarrow \alpha-(-1)^{n-1}J'{^tX}{^tJ'}X^{-1}\alpha=0\Longleftrightarrow \alpha-J'^{-1}{^tX}{^tJ'}X^{-1}\alpha=0\Longleftrightarrow {^tX^{-1}}J'\alpha-{^tJ'}X^{-1}\alpha=0\Longleftrightarrow ({^tX^{-1}}J'-{^tJ'}X^{-1})\alpha=0\Longleftrightarrow ({^tJ'}X-(-1)^{n-1}{^tX}J')X^{-1}\alpha=0 \Longleftrightarrow (4')$. So $(2)\Longleftrightarrow (4')\Longleftrightarrow (4).$

Next we show that $(i)+(ii)\Longrightarrow (4').$
Notice that
$(i)\Longleftrightarrow{^tJ'}X+{^tX}J'+{^tJ'}\alpha{^t\alpha}J'=0\Longleftrightarrow {^tJ'}X+{^tX}J'+(-1)^{n-1}J'\alpha{^t\alpha}J'=0\Longleftrightarrow {^tJ'}X+{^tX}J'+J'\alpha{^t\alpha}{^tJ'}=0,$
multiply this by $X^{-1}\alpha$ we obtain
${^tJ'}\alpha+{^tX}J'X^{-1}\alpha+J'\alpha((-1)^n-1)=0$.
When $n$ is even, this is equal to 
${^tJ'}\alpha+{^tX}J'X^{-1}\alpha=0,$ on the other hand in this case we have $(4')\Longleftrightarrow ({^tJ'}X+{^tXJ'})X^{-1}\alpha=0\Longleftrightarrow {^tJ'}\alpha+{^tX}J'X^{-1}\alpha=0;$
When $n$ is odd, this is saying that
${^tJ'}\alpha+{^tX}J'X^{-1}\alpha-2J'\alpha=0$, but since ${^tJ'}=(-1)^{n-1}J'=J'$ in this case, we have that this is the same as saying
${^tJ'}\alpha-{^tX}J'X^{-1}\alpha=0,$ while $(4')\Longleftrightarrow ({^tJ'}X-{^tX}J')X^{-1}\alpha=0\Longleftrightarrow {^tJ'}\alpha-{^tX}J'X^{-1}\alpha=0.$ 
Hence in both cases we have that $(i)+(ii)\Longrightarrow (4')$, and this is the same as saying that $(5)+(i)\Longleftrightarrow (i)+(ii)\Longrightarrow (2)\&(4).$
So we obtain that $(1)+(2)+\cdots+(9)+(i)\Longleftrightarrow (i)+(ii)+(1).$ 

We are left with (1).
We have 
$(1)\Longleftrightarrow I-\alpha(-{^t\alpha}{^tX^{-1}}J')+({^tJ'}{^tX}J'^{-1})$
$\cdot{^tJ'}({^tJ'}X^{-1}J'^{-1}){^tJ'}=0 \Longleftrightarrow I+\alpha{^t\alpha}{^tX^{-1}}J'+J'{^tX}J'X^{-1}=0,$ we call the last formula $(iii)$.

We show that if we pick $n\in N_H(F)'$, the open dense subset of $N_H(F)$ constructed in the last section, then both $(ii)$ and $(iii)$ are implied by $(i)$.

If we let 
$Y=X{^tJ'}$ and $Z=X{^tJ'}+\frac{\alpha{^t\alpha}}{2}=Y+\frac{\alpha{^t\alpha}}{2}$ as in the previous section in which we find orbit representatives for $U_{M_H}(F)\backslash N_H(F)$, then there exists $u\in U_n(F)$ such that 
$uZ{^tu}=\begin{bmatrix}
0 & a_1 & \ \ & \ \ & \ \ \\
-a_1 & 0 & \ \ & \ \ & \ \ \\
\ \ & \ \ & \ddots & \ \ & \ \ \\
\ \ & \ \ & \ \ & 0 & a_{n-1} \\
\ \ & \ \ & \ \ & -a_{n-1} & 0 \\
\end{bmatrix}$, we denote this matrix by $Z(a_1,\cdots, a_{n-1})$,  and we also have
$u\alpha=[0,\cdots,0,a_n]^t$, hence
$uY{^tu}=\begin{bmatrix}
0 & a_1 & \ \ & \ \ & \ \ \\
-a_1 & 0 & \ \ & \ \ & \ \ \\
\ \ & \ \ & \ddots & \ \ & \ \ \\
\ \ & \ \ & \ \ & 0 & a_{n-1}  \\
\ \ & \ \ & \ \ & -a_{n-1} & -\frac{a_n^2}{2} \\\end{bmatrix}$, we denote this matrix by $Y(a_1,\cdots, a_n)$. Then we see that
$(i)\Longleftrightarrow Y+{^tY}+\alpha{^t\alpha}=0\Longleftrightarrow u(Y+{^tY}+\alpha{^t\alpha}){^tu}=0\Longleftrightarrow uY{^tu} +{^t(uY{^tu})}+(u\alpha){^t(u\alpha)}=0;$

$(ii)\Longleftrightarrow {^t\alpha}Y^{-1}\alpha=-1-(-1)^{n-1}\Longleftrightarrow {^t(u\alpha)}(uY{^tu})^{-1}(u\alpha)=-1-(-1)^{n-1};$
$(iii)\Longleftrightarrow I+(-1)^{n-1}\alpha{^t\alpha}{^tY^{-1}}+{^tY}Y^{-1}=0\Longleftrightarrow u(I+(-1)^{n-1}\alpha{^t\alpha}{^tY^{-1}}+{^tY}Y^{-1})u^{-1}=0\Longleftrightarrow I+(-1)^{n-1}(u\alpha){^t(u\alpha)}{^t(uY{^tu})^{-1}}+{^t(uY{^tu})}(uY{^tu})^{-1}=0$.

Therefore, without loss of generality, we can assume that $Y=Y(a_1,\cdots,a_n)$
and 
$\alpha=[0,\cdots,0,a_n]^t$ with all $a_i\neq 0$ in this proof.
We work on the cases when the size of the matrix $n$ is even or odd separately.

Case 1: When $n$ is even;

Now we have that ${^tJ'}=J'^{-1}=(-1)^{n-1}J'=-J'.$ So
$(ii)\Longleftrightarrow {^t\alpha}Y^{-1}\alpha=0,$ notice that $\alpha$ is a vector with only the last entry non-zero, so only the last entry in $Y^{-1}$ contributes. Let $Y_{i,j}^*$ denote the $(i,j)$-th entry of the adjoint matrix of $Y$. Then we see that ${^t\alpha}Y^{-1}\alpha=a_n^2(\det Y^{-1})Y_{n,n}^*$. But since $n$ is even and therfore the $(n,n)$-th minor of $Y$ is an $(n-1)\times (n-1)$ skew-symmetric matrix of odd size, thus $Y_{n,n}^*=0$, hence ${^t\alpha} Y^{-1}\alpha=0;$ 
And we also have that 
$(iii)\Longleftrightarrow I-\alpha{^t\alpha}{^tY^{-1}}+{^tY}Y^{-1}=0\Longleftrightarrow I-Y^{-1}\alpha{^t\alpha}+{^tY^{-1}}Y=0.$
But $(i)\Longleftrightarrow Y+{^tY}+\alpha{^t\alpha}=0 \Longleftrightarrow {^tY^{-1}}Y+I+{^tY^{-1}}\alpha{^t\alpha}=0$,
so if we replace ${^tY^{-1}}Y$ by $-I-{^tY^{-1}}\alpha{^t\alpha}$ in the last formula for (iii) right above, then we have
$(iii)\Longleftrightarrow (Y^{-1}+{^tY^{-1}})\alpha{^t\alpha}=0$. But now $\alpha{^t\alpha}$ is a matrix with only the last entry non-zero and equals $a_n^2$, so only the last column of $Y^{-1}+{^tY}^{-1}$ contribute. For the same reason we have that $Y_{n,n}^*={^tY_{n,n}^*}=0$. On the other hand, for the matrix $Y$, we see that  $Y_{i,j}=-Y_{j,i}$ for all $(i,j)\neq (n,n)$, so we see that ${^tY_{i,n}^*}=(-1)^{n-1}Y_{i,n}^*=-Y_{i,n}^*$ for all $1\leq i\leq (n-1)$. This implies that $(Y^{-1}+{^tY^{-1}})\alpha{^t\alpha}=0.$

Case 2: When $n$ is odd.

Now
$(ii)\Longleftrightarrow {^t\alpha}Y^{-1}\alpha=-2$.
We see that $Y_{n,n}^*=\det Y_{n-1}$ where $Y_{n-1}$ is the principal $(n-1)$-th minor of $Y$, therefore one can easily prove by induction that
$\det Y_{n-1}=\prod_{k\ \ odd,k\neq n}a_k^2$ but on the other hand $\det Y=-\frac{1}{2}\prod_{k\ \ odd}a_k^2$, which can also be proved by induction on the size.
Therefore we have ${^t\alpha}Y^{-1}\alpha=(\det Y)^{-1}Y_{n,n}^*a_n^2=\frac{\prod_{k\ \ odd,k\neq n}a_k^2}{-\frac{1}{2}\prod_{k\ \ odd}a_k^2}\cdot a_n^2=-2.$ We also have 
$(iii)\Longleftrightarrow I+\alpha{^t\alpha}{^tY^{-1}}+{^tY}Y^{-1}=0\Longleftrightarrow I+Y^{-1}\alpha{^t\alpha}+{^tY^{-1}}Y=0$.
Again by $(i)$ we have ${^tY^{-1}}Y=-I-{^tY^{-1}}\alpha{^t\alpha}$,
so $(iii)\Longleftrightarrow (Y-{^tY^{-1}})\alpha{^t\alpha}=0$. But in this case $Y^*_{n,n}={^tY^*_{n,n}}=\prod_{k\ \ odd,k\neq n}a_k^2$, and ${^tY_{i,n}^*}=(-1)^{n-1}Y_{i,n}^*=Y_{i,n}^*$, therefore it shows that
$(Y-{^tY^{-1}})\alpha{^t\alpha}=0.$

From the above argument we see that in both cases if we pick $n=n(X,\alpha)\in N_H(F)'$, with $\det X\neq 0$ then $(i)\Longleftrightarrow (i)+(ii)+(iii)\Longleftrightarrow (i)+(1)+\cdots +(9).$

We have showed that for $n=n(X,\alpha)\in N_H(F)$, assume $\det(X)\neq 0$, then $\tilde{w}_0^{-1}n(X,\alpha)=m(J'\leftidx{^t}Y^{-1})n'\overline{n}$. Since $\dot{w}_0^{-1}=m(-\frac{1}{2}I)\tilde{w}_0^{-1}$, we see that 
$\dot{w}_0^{-1}n=m(-\frac{1}{2}I)m(J'\leftidx{^t}Y^{-1})n'\bar{n}=m(-\frac{1}{2}J'\leftidx{^t}Y^{-1})n'\overline{n}$ holds for $n\in N_H(F)'$, which already implies that $\det{X}\neq 0$ since $X=Y\leftidx{^t}J'^{-1}=YJ'$, and $\det(Y)=\det(Y(a_1,\cdots,a_n))\neq 0$. This gives the decomposition in $SO_{2n+1}$. 

The decomposition $\dot{w}_0^{-1}n=mn'\bar{n}$ in $\textrm{SO}_{2n+1}$ and $\textrm{Spin}_{2n+1}$ differ only by the $m$ part. Recall that at the end of section 5.1 we have $M_{H_D}=\{m(g,a)\in M_H\simeq\textrm{GL}_n\times \textrm{GL}_1: \det(g)a^2=1\}^\circ$, and the covering map $\varphi:M_{H_D}\rightarrow M\simeq \textrm{GL}_n$ is given by $m(g,a)\mapsto m(g)\mapsto g$. So for $n\in N_H(F)'$, we see that $\dot{w}_0^{-1}n=m(g,a(g) )n'\overline{n}$, where $g=(-\frac{1}{2})J'\leftidx{^t}Y^{-1}$, and $a(g)$ is uniquely determined by the relation $\det(g)\cdot a(g)^2=1$, since from the realization of $M_{H_D}$ in $M_H$ the $F$-points of $M_{H_D}$ is given by a pair $(g,a)\in \textrm{GL}_n(F)\times \textrm{GL}_1(F)$ such that $\det(g)=a^{-2}$ is a square in $F^\times$ and this $a$ is the unique square root of $\det(g)^{-1}$ that lies in the identity component of the $F$-points of the variety $\{(g,a)\in \textrm{GL}_n\times \textrm{GL}_1: \det(g)a^2=1\}$.  

If $Y=Y(a_1\cdots, a_n)$, we can see that $\det(g)=\det((-\frac{1}{2})J' \leftidx{^t}Y(a_1,\cdots,a_n)^{-1})=\frac{(-\frac{1}{2})^n}{\prod_{k\ \ odd}a_k^2}$ if $n$ is even, and $(-\frac{1}{2})^{n}\cdot\frac{-2}{\prod_{k\ \ odd}a_k^2}=\frac{(-\frac{1}{2})^{n-1}}{\prod_{k\ \ odd}a_k^2}$ if $n$ is odd. Hence $a(g)=\frac{(\frac{1}{2})^{\frac{n}{2}}}{\prod_{k\ \ odd}a_k}$ if $n$ is even, and $\frac{(\frac{1}{2})^{\frac{n-1}{2}}}{\prod_{k\ \ odd}a_k}$ if $n$ is odd. So we obtain the desired Bruhat decomposition in $\textrm{Spin}_{2n+1}$ and therefore in $H=\textrm{GSpin}_{2n+1}$.
\subsection{LOCAL COEFFICIENTS AND PARTIAL BESSEL FUNCTIONS}  

Now we are ready to apply Theoerem 6.2 of [19] to express the local coefficients as the Mellin transform of partial Bessel functions in our setting.

Recall that we have an injection $\alpha^{\vee}:F^*\hookrightarrow Z_H\backslash Z_{M_H}$ and $\alpha(\alpha^{\vee}(t))=t$ for $t\in F^*$ (Lemma 5.5). By the last section we also obtained that the decomposition 
$\dot{w}_0^{-1}n=mn'\bar{n}$ holds for $n\in N_H(F)'\subset N_H(F)$. Moreover, by the work of R. Sundaravaradhan in [22], we have that except for a set of measure zero on $N_H(F)$, 
$U_{M_H,n}=U_{M_H,m}'$, where $U_{M_H,n}=\{u\in U_{M_H}: unu^{-1}=n\},$ and $U_{M_H,m}'=\{u\in U_{M_H}: mum^{-1}\in U_{M_H}\ \ \& \ \ \chi(mum^{-1})=\chi(u)\}$.
The above two properties imply that the assumptions for Theorem 6.2 in [19] are satisfied. 

Let $\pi$ be a $\psi$-generic representation of $\textrm{GL}_n(F)$ and $\eta$ a character of $F^\times$, and $\lambda$ be a Whittaker functional attached to $\pi$. Since $U_{M_H}\simeq U_n$, $\psi$ can be viewed as a character of $U_{M_H}$. The representation $\sigma_\eta$ of $M_H(F)$ is also generic. Since $\psi(u)\lambda(v)=\lambda(\pi(u)v)=\lambda(\sigma_\eta(m(u,1)v))$, $\lambda$ can also be viewed as a Whittaker functional of $\sigma_\eta$.

Let $ \mathfrak{a}^*_{H,\mathbb{C}}=\mathfrak{a}_H^*\otimes_{\mathbb{R}}\mathbb{C}$, where $\mathfrak{a}_H^*=X(M_H)_F\otimes_{\mathbb
Z}\mathbb{R}$, and $\mathfrak{a}_H=\Hom(X(M_H)_F,\mathbb{R})$ is the real Lie algebra. The Harish-Chandra map 
$H_{M_H}:M_H\longrightarrow \mathfrak{a}_H$
is defined by 
$q^{\langle \chi, H_{M_H}(m)\rangle}=\vert \chi(m)\vert_F$ for all $\chi \in X(M_H)_F.$ Given $\mu\in \mathfrak{a}_{H,\mathbb{C}}^*$, let
$I(\mu, \sigma_\eta)=\Ind_{M_HN_H}^H((\sigma_\eta\otimes q^{\langle \mu, H_{M_H}(\cdot)})\otimes 1_{N_H})$ be the induced representation, and denotes its space by $V(\mu,\sigma_\eta).$ 
As before let $\sigma_{\eta,s}$ denote the representation $\sigma_\eta \otimes q^{\langle s\hat{\alpha}, H_{M_H}(\cdot)\rangle}$, where $\hat{\alpha}=\langle \rho, \alpha\rangle ^{-1}\rho=\frac{(\alpha,\alpha)}{2(\rho,\alpha)}\rho= \frac{(e_n,e_n)}{2\cdot\frac{n}{2}(\sum_{i=1}^ne_i,e_n)}$
$\cdot(\frac{n}{2}\sum_{i=1}^ne_i)=\frac{1}{2}\sum_{i=1}^ne_i.$
For $s\in \mathbb{C}$, define $I(s,\sigma_\eta)=I(s\hat{\alpha},\sigma_\eta)$ and let $V(s,\sigma_\eta)$ be its space. The local standard intertwining operator 
$A(s,\sigma_\eta):I(s,\sigma_\eta)\longrightarrow I(-s,w_0(\sigma_\eta))$
is defined by
$A(s.\sigma_\eta)f(h)=\int_{N_H}f(\dot{w}_0^{-1}nh)dn$ for $\forall h\in H$ and $f\in V(s,\sigma_\eta)$. We identify $\lambda$ as a Whittaker functional for $\sigma_\eta$, and
denote $\lambda_{\psi}(s,\sigma_\eta)$ the Whittaker functional for $I(s,\sigma_\eta)$ given by $\lambda$, defined as
$\lambda_{\psi}(s,\sigma_\eta)(f)=\int_{N_H}\langle f(\dot{w}_0^{-1}n),\lambda\rangle\cdot \psi^{-1}(n)dn.$ Then since $\psi$ is compatible with $\dot{w}_0$, 
$\lambda_{\psi}(-s,w_0(\sigma_\eta))\circ A(s,\sigma_\eta)$ defines another Whittaker functional for $I(s,\sigma_\eta)$. 
So by uniqueness of the local Whittaker functionals we obtain that the local coefficient $C_{\psi}(s,\sigma_\eta)$ is defined by
$\lambda_{\psi}(s,\sigma_\eta)=C_{\psi}(s,\sigma_\eta)\cdot \lambda_{\psi}(-s, w_0(\sigma_\eta))\circ A(s,\sigma_\eta).$

As in [19] we will choose $\overline{N}_0\subset \overline{N}_H(F)$ to be open compact so that $\alpha^{\vee}(t)\overline{N}_0\alpha^{\vee}(t)^{-1}$ depends only on $\vert t\vert$ for all $t\in F^*$. Define $\varphi_{\kappa}(X)=1$ if $\vert X_{i,j}\vert\leq q^{(i+j-1)\kappa}$ , and $0$ otherwise.

From the calculation of the decomposition $\dot{w}_0^{-1}n=mn'\bar{n}$ in the last section we see that if $n=n(X,\alpha)$ with $\det(X)\neq 0$, then $\bar{n}=\begin{bmatrix}
I & \ \ & \ \ \\
-^t(J'X^{-1}\alpha) & 1 & \ \ \\
X^{-1} & X^{-1}\alpha & I \\
\end{bmatrix}$, we denote $\begin{bmatrix}
I & \ \ & \ \ \\
-^t(J'\tilde{X}\alpha) & 1 & \ \ \\
\tilde{X} & \tilde{X}\alpha & I \\
\end{bmatrix}$ by $\bar{n}(\tilde{X},\alpha)$. Let 
$$\overline{N}_{0,\kappa}=\{\bar{n}=\bar{n}(\tilde{X},\alpha):  \varphi_{\kappa}(-\frac{1}{8}\varpi^{2(d+f)}\cdot\leftidx{^t}{\tilde{X}}{J'}^{-1})=1\},$$ where $d$ is the conductor of $\chi$ and $f$ is the conductor of $w_{\pi}^{-1}(w_0w_{\pi})$. And let $\varphi_{\overline{N}_0,\kappa}$ be the characteristic function of $\overline{N}_{0,\kappa}$.

Let $n\in N_H(F)'$ with $w_0^{-1}n=mn'\bar{n}$, and let $z\in Z_M^0=\{\alpha^{\vee}(t): t\in F^*\}$. As in (6.21) of [19], the partial Bessel function on $M_H(F)\times Z^0_{M_H}$ is defined by 
$$j_{\sigma_{\eta,s},\kappa}(m,z)=\int_{U_{M_H}} W_{\sigma_{\eta,s},v}(mu^{-1})\varphi_{\overline{N}_{0,\kappa}}(zu^{-1}\bar{n}uz^{-1})\psi^{-1}(u)du$$
where $W_{\sigma_{\eta,s},v}\in W(\sigma_{\eta,s})$ is a Whittaker model attached to the $\sigma_{\eta,s}$, with $v$ a fixed vector in the represenattion space. For partial Bessel functions for quasi-split groups, we refer the reader to [6].

In our case $m=m(g,a(g))$ with $\det(g)a(g)^2=1$, and $u=m(u',1)$ for $u'\in U_n$. Hence $W_{\sigma_{\eta,s},v}(m(g,a(g)))=\lambda(\sigma_{\eta,s}(m(g,a(g)))v)=\eta(a(g))^{-1}\vert \det(g)\vert^{\frac{s}{2}}\lambda(\pi(g)v)=\eta(a(g))^{-1}\vert \det(g)\vert^{\frac{s}{2}}W_{\pi,v}(g).$ Moreover, let $z=\alpha^{\vee}(\varpi^{d+f}u_{\alpha_n}(\dot{w}_0\bar{n}\dot{w}_0^{-1}))$, and define for $g\in \textrm{GL}_n(F)$,  
$$j_{\pi,\eta,\dot{w}_\theta,\kappa}(g)=j_{\sigma_{\eta,s},\kappa}(m,\alpha^{\vee}(\varpi^{d+f}u_{\alpha_n}(\dot{w}_0\bar{n}\dot{w}_0^{-1}))),$$ where $m=m(g,a(g)).$ This defines the partial Bessel function on $\textrm{GL}_n(F)$ in our case.
Now apply Theorem 6.2 in [19], we obtain
\begin{proposition} Let $\pi$ be an irreducible admissible $\psi$-generic representation of $\textrm{GL}_n(F)$, lifted as a $\psi$-generic representation $\sigma$ of $M_H(F)\simeq \textrm{GL}_n(F)\times \textrm{GL}_1(F)$ by pull-back through the projection on the $\textrm{GL}_n$-factor. $\eta:F^\times\rightarrow \mathbb{C}^{\times}$ is a fixed continuous character. Define the representation $\sigma_\eta$ as before. Suppose that $\omega_{\sigma_\eta}(w_0\omega_{\sigma_\eta}^{-1})$ is ramified as a character of $F^{\times}$. Then for all sufficiently large $\kappa$ we have
$$C_{\psi}(s,\sigma_\eta)^{-1}=\gamma(2\langle \hat{\alpha}, \alpha^{\vee}\rangle)s, \omega_{\sigma_\eta}(w_0w_{\sigma_\eta}^{-1})\circ \alpha^{\vee},\psi)^{-1}$$
$$\cdot\int_{Z_{M_H}^0 U_{M_H}\backslash N_H} j_{\pi,\eta,\dot{w}_\theta, \kappa}(g)\omega_{\sigma_{\eta,s}}^{-1}(\alpha^{\vee}(u_n))(w_0\omega_{\sigma_{\eta,s}})(\alpha^{\vee}(u_n))q^{\langle s \hat{\alpha}+\rho,H_M(m)\rangle}d\dot{n}$$
where off a set of measure zero, the decomposition $\dot{w}_0^{-1}n=mn'\bar{n}$ holds as in the previous section. Here $u_n=u_{\alpha_n}(\dot{w}_0\bar{n}\dot{w}^{-1}_0)\in U_{\alpha}$, $\gamma(2\langle \hat{\alpha}, \alpha^{\vee}\rangle s, \omega_{\sigma_\eta}(w_0\omega_{\sigma_\eta}^{-1})\circ \alpha^{\vee},\psi)$ is an abelian $\gamma$-factor depending only on $\omega_{\pi}$ and $\eta$. 
\end{proposition}
Let's simplify this formula. First recall that in our case $\alpha=e_n$, $\rho$ is the half of the sum of roots in $N_H$. The roots in $N_H$ are $e_i+e_j(1\leq i<j\leq n)$ and $e_i(1\leq i\leq n)$, so $\rho=\frac{1}{2}(\sum_{1\leq i<j\leq n}(e_i+e_j)+\sum_{i=1}^n e_i)=\frac{n}{2}\sum_{i=1}^n e_i$. 

We have $$\langle \rho, \alpha\rangle=\frac{2(\rho,\alpha)}{(\alpha,\alpha)}=\frac{2(\frac{n}{2}\sum_{i=1}^n e_i, e_n)}{(e_n,e_n)}=n.$$ So $\hat{\alpha}=\langle \rho, \alpha\rangle ^{-1}\rho=n^{-1}(\frac{n}{2}\sum_{i=1}^ne_i)=\frac{1}{2}\sum_{i=1}^ne_i$.
 Since $\alpha^{\vee}=\sum_{i=1}^ne_i^*,$ so we have for $\forall t\in F^*$,
$t^{\langle \hat{\alpha}, \alpha^{\vee}\rangle}=\hat{\alpha}(\alpha^{\vee}(t))=\frac{1}{2}\sum_{i=1}^ne_i(\prod_{i=1}^ne_i^*(t))=t^{n/2}.$
Therefore $\langle \hat{\alpha},\alpha^{\vee}\rangle=\frac{n}{2}.$
This implies that
$q^{\langle s\hat{\alpha}, H_{M_H}(m)}=q^{\langle s\hat{\alpha}, H_{M_H}(m(g,a(g)))}=\vert \det(g)\vert^{s/2}.$ 
Then
$\omega_{\sigma_{\eta,s}}(m(g,a(g)))$
$=\omega_{\sigma_{\eta}}(m(g,a(g)))\vert \det g\vert^{s/2}=\eta^{-1}(a(g))\vert \det(g)\vert^{\frac{s}{2}}\omega_\pi(g).$

Secondly, since we have $w_0=w_H\cdot w_{\theta}$, where $\theta=\Delta-\{\alpha_n\}=\Delta-\{\alpha\},$ and 
$w_H: e_i\mapsto -e_i$, $w_{\theta}: e_i\mapsto e_{n+1-i}$, we obtain
$w_0^{-1}\cdot\prod_{i=1}^ne_i^*(t)\cdot w_0=\prod_{i=1}^n(-e^*_{n+1-i}(t))=\prod_{i=1}^n(-e^*_i(t)).$
This implies that 
$$\omega_{\sigma_\eta}(w_0\omega_{\sigma_\eta}^{-1})(\alpha^{\vee}(t))=\omega_{\sigma_\eta}(\prod_{i=1}^ne_i^*(t))\cdot \omega_{\sigma_\eta}^{-1}(w_0^{-1}\cdot \prod_{i=1}^ne_i^*(t)\cdot w_0)=$$$$\omega_{\sigma_\eta}(\prod_{i=1}^ne^*_i(t))\cdot \omega_{\sigma_\eta}^{-1}(\prod_{i=1}^n(-e^*_i(t)))=\omega_{\sigma_\eta}(\prod_{i=1}^ne^*_i(t))\cdot \omega_{\sigma_\eta}(\prod_{i=1}^ne^*_i(t))=\omega_{\sigma_\eta}^2(\alpha^{\vee}(t))=\omega_{\pi}^2(t),$$ since $\eta$ is trivial on the $\textrm{GL}_n$-component of $M_H$.
So $\omega_{\pi}(w_0 \omega_{\pi}^{-1})\circ \alpha^{\vee}=\omega_{\pi}^2.$

Similarly 
$$\omega_{\sigma_{\eta,s}}^{-1}(w_0 \omega_{\sigma_{\eta,s}})(\alpha^{\vee}(t))=\omega_{\sigma_{\eta,s}}^{-1}(\prod_{i=1}^ne^*_i(t))\cdot \omega_{\sigma_{\eta_s}}(\prod_{i=1}^n(-e^*_i(t)))=\omega_{\sigma_{\eta_s}}^{-2}(\prod_{i=1}^ne^*_i(t))$$$$=\omega_{\pi}^{-2}(\alpha^{\vee}(t))\cdot \vert t^{n}\vert^{-(s/2)\cdot 2}\cdot =\omega_{\pi}^{-2}(t)\cdot \vert t\vert ^{-ns}.$$ So
$\omega_{\sigma_{\eta,s}}^{-1}(w_0\omega_{\sigma_{\eta,s}})\circ \alpha^{\vee}=\omega_{\pi}^{-2}(\cdot) \vert \cdot \vert ^{-ns}.$

Finally
$$q^{\langle s\hat{\alpha}+\rho,H_{M_H}(m)\rangle}=\vert(\frac{s}{2}\sum_{i=1}^ne_i+\frac{n}{2}\sum_{i=1}^ne_i)(m(g,a(g)))\vert=$$
$$\vert \frac{(s+n)}{2}\sum_{i=1}^ne_i(m(g,a(g))))\vert=\vert\det(g)\vert^{\frac{s+n}{2}}.$$

From the above discussion we obtain a simplified version of the local coefficient formula in our case, namely
\begin{proposition} Let $\pi$ be an irreducible admissible $\psi$-generic representation of $\textrm{GL}_n(F)$, lifted as a $\psi$-generic representation $\sigma$ of $M_H(F)\simeq \textrm{GL}_n(F)\times \textrm{GL}_1(F)$ by pull-back through the projection on the $\textrm{GL}_n$-factor. $\eta:F^\times\rightarrow \mathbb{C}^{\times}$ is a fixed continuous character. Define the representation $\sigma_\eta$ as before. Suppose that $\omega_{\sigma_\eta}(w_0\omega_{\sigma_\eta}^{-1})$ is ramified as a character of $F^{\times}$. Then for all sufficiently large $\kappa$ we have
$$C_{\psi}(s,\sigma_\eta)^{-1}=\gamma(ns, \omega_{\pi}^2,\psi)^{-1}$$$$\cdot\int_{Z^0_{M_H}U_{M_H}\backslash N_H}j_{\pi,\eta,\dot{w}_\theta,\kappa}(g)\omega_{\pi}^{-2}(u_n)\vert u_n\vert^{-ns}\vert \det(g)\vert^{\frac{s+n}{2}}d\dot{n}. $$
where off a set of measure zero, the decomposition $\dot{w}_0^{-1}n=mn'\bar{n}$ holds as in the previous section. Here $u_n=u_{\alpha_n}(\dot{w}_0\bar{n}\dot{w}^{-1}_0)\in U_{\alpha_n}=U_{\alpha}$. And $\gamma(ns, \omega_{\pi}^2,\psi)$ is an abelian $\gamma$-factor depending only on $\omega_{\pi}$. 
\end{proposition}
In the proof of stability, we also need an integral formula for the local coefficient $C_{\psi}(s,(\sigma_\eta \otimes \chi))^{-1}$ for a sufficiently ramified character $\chi$ of $F^{\times}$, viewed as a character of $M_H(F)$ by $\chi(m(g,a))=\chi(\det(g))$. Therefore it is important to be able to choose $\kappa$ or equivalently, $\overline{N}_0\subset \overline{N}_H(F)$ to be independent of $\chi$.

To make this work, as in the proof of Theorem 6.2 in [19] and the corresponding discussion in [7], if we fix an irreducible generic representation $\pi'$ of $G$ such that $\omega_{\sigma_\eta'}$ is ramified, where $\sigma'$ is the lift of $\pi'$, $\sigma_\eta'$ is defined in the same way as $\sigma_\eta$. Then $\overline{N}_0$ is chosen to satisfy (1) $\exists f\in V(s,\sigma_\eta')$ such that $f$ is supported in $P_H\overline{N}_0$; (2) $\overline{N}_0$ is large enough such that $\alpha^{\vee}(t)\overline{N}_0\alpha^{\vee}(t)^{-1}$ depends only on $\vert t\vert $ for all $t\in F^{\times}.$ Note that here (2) does not depend on $\pi'$. For (1), as in the proof of Theorem 6.2 in [19], there exist $f\in V(s,\sigma_\eta') $ s.t. $f$ is compactly supported modulo $P_H$. Fix such an $f$ and choose $\overline{N}_0$ sufficiently large such that it contains the support of $f$, then $f$ is supported in $P_H\overline{N}_0$. 

Now let's get back to our case. We fix a character $\chi_0$ of $F^{\times}$ such that $\omega_{\sigma_\eta}\chi_0^n=\eta^{-1}\omega_\pi\chi_0^n=\omega_{\sigma_\eta\otimes \chi_0}$ is ramified. Then we take $\kappa_0$ such that both conditions (1) and (2) above are satisfied for $\overline{N}_{0,\kappa_0}$ and $f_{\chi_0}\in V(s,\sigma_\eta\otimes\chi_0 )$. Also note that if $\kappa\ge \kappa_0$, we have $\overline{N}_{0,\kappa_0}\subset \overline{N}_{0,\kappa}$. Therefore (1) and (2) hold for $\sigma_\eta\otimes \chi_0$ and all $\kappa\ge \kappa_0.$ Let $\chi$ be any other character of $F^{\times}$ such that $\omega_{\sigma_\eta}\chi^n$ is ramified. Then as discussed above we can choose $f_{\chi}\in V(s,\sigma_\eta\otimes\chi )$ which is supported in $P_H\overline{N}_{0,\chi}$ for some open compact $\overline{N}_{0,\chi}\subset\overline{N}_H$. Now if $\overline{N}_{0,\chi}\subset\overline{N}_{0,\chi_0}$, then Proposition 5.7 holds for $\sigma_\eta\otimes \chi$ and all $\kappa\ge \kappa_0$. While if not, note that $\alpha^{\vee}(t)=\prod_{i=1}^ne_i^*(t)\in M_H,$ then $R(\alpha^{\vee}(t)^{-1})f$ will be supported in $P_H(\alpha^{\vee}(t)^{-1}\overline{N}_{0,\chi}\alpha^{\vee}(t))$. To see this, note that for $\bar{n}(\tilde{X},\alpha)\in \overline{N}_H(F)$, we have $\alpha^{\vee}(t)^{-1}\bar{n}(\tilde{X},\alpha)\alpha^{\vee}(t)=\bar{n}(t^2\tilde{X}, t\alpha).$ Therefore if we take $\vert t\vert $ sufficiently small, we will have
$\alpha^{\vee}(t)^{-1}\overline{N}_{0,\chi}\alpha^{\vee}(t)\subset \overline{N}_{0,\kappa_0}$. So if we take such a $t$ and replace $f$ with $f'_{\chi}=R(\alpha^{\vee}(t)^{-1})f_{\chi}$, we see that $f'_{\chi}$ will be supported in $P_H\overline{N}_{0,\kappa_0}$ and Proposition 5.7 holds for $\sigma_\eta\otimes \chi$ and for all $\kappa\ge \kappa_0$. Now we obtain a stronger version of Proposition 5.7.

\begin{proposition} Let $\pi$ be an irreducible admissible $\psi$-generic representation of $\textrm{GL}_n(F)$, lifted as a $\psi$-generic representation $\sigma$ of $M_H(F)\simeq \textrm{GL}_n(F)\times \textrm{GL}_1(F)$ by pull-back through the projection on the $\textrm{GL}_n$-factor. $\eta:F^\times\rightarrow \mathbb{C}^{\times}$ is a fixed continuous character. Define the representation $\sigma_\eta$ as before. Suppose that $\omega_{\sigma_\eta}(w_0\omega_{\sigma_\eta}^{-1})$ is ramified as a character of $F^{\times}$. Then there exist a $\kappa_0$ such that for all $\kappa\ge \kappa_0$ and all $\chi$ such that $\omega_{\sigma_\eta}\chi^n$ is ramified, we have
$$C_{\psi}(s,\sigma_\eta\otimes \chi)^{-1}=\gamma(ns, (w_{\pi}\chi)^{2n},\psi)^{-1}\int_{Z^0_{M_H}U_{M_H}\backslash N_H}j_{\pi\otimes \chi,\eta,\dot{w}_\theta,\kappa}(g)(\omega_{\pi}\chi^n)^{-2}(u_n)$$$$\cdot \vert u_n\vert^{-ns}\vert \det(g)\vert^{\frac{s+n}{2}}d\dot{n}. $$
where off a set of measure zero, the decomposition $\dot{w}_0^{-1}n=mn'\bar{n}$ holds as in the previous section. Here $u_n=u_{\alpha_n}(\dot{w}_0\bar{n}\dot{w}^{-1}_0)\in U_{\alpha_n}=U_{\alpha}$. And $\gamma(ns,(\omega_{\pi}\chi)^{2n},\psi)$ is an abelian $\gamma$-factor depending only on $\omega_{\pi}$ and $\chi$. 
\end{proposition}
Next, we use our orbit space representatives and measure to further simplify the integral in the local coefficient formula.
Recall that we have the decomposition 
$\dot{w}_0^{-1}n=mn'\bar{n}$
holds for $n$ lying in the open dense subset $N_H(F)'$ of $N_H(F)$. Now for $n=n(X,\alpha)$,
let $Y=X^tJ'=(Z-\frac{\alpha^t\alpha}{2})J'^tJ'=Z-\frac{\alpha^t\alpha}{2}$. Then by section 5.2 on orbit space and measure, if $n\in N_H(F)'$, then $Z$ can be taken as $Z(a_1,\cdots,a_{n-1})$ and $\alpha$ can be taken as $[0,\cdots,0,a_n]^t$, consequently $Y$ can be given as 
$Y(a_1,\cdots,a_n)$(see the notations on Page 22). Also recall that the calculation of the decomposition $w_0^{-1}n=mn'\bar{n}$ gives 
$m=m(g,a(g))$
where $g=(-\frac{1}{2})J'^tY^{-1}$ and $a(g)=\frac{(\frac{1}{2})^{\frac{n}{2}}}{\prod_{k\ \ odd}a_k}$ if $n$ is even and $a(g)=\frac{(\frac{1}{2})^{\frac{n-1}{2}}}{\prod_{k\ \ odd}a_k}$ if $n$ is odd.

We have seen that in the decomposition $\dot{w}_0^{-1}n=mn'\bar{n}$, if $n=n(X,\alpha)$, then the corresponding
$\bar{n}=\bar{n}(X^{-1},\alpha)=\begin{bmatrix}
I & \ \ & \ \ \\
-^t(J'X^{-1}\alpha) & 1 & \ \ \\
X^{-1} & X^{-1}\alpha & I \\
\end{bmatrix}.$
So $$\dot{w}_0\bar{n}\dot{w}_0^{-1}=\begin{bmatrix}
\ \ & \ \ & (-\frac{1}{2})I \\
\ \ & (-1)^n & \ \ \\
(-1)^{n}2I & \ \ & \ \ \\
\end{bmatrix}
\begin{bmatrix}
I & \ \ & \ \ \\
-^t(J'X^{-1}\alpha) & 1 & \ \ \\
X^{-1} & X^{-1}\alpha & I \\
\end{bmatrix}$$$$
\cdot \begin{bmatrix}
\ \ & \ \ & (-1)^{n}\frac{1}{2}I \\
\ \ & (-1)^n & \ \ \\
-2I & \ \ & \ \ \\
\end{bmatrix}
=\begin{bmatrix}
I & (-1)^{n-1}\frac{1}{2}X^{-1}\alpha & (-1)^{n-1}\frac{1}{4}X^{-1} \\
\ \ & 1 & -\frac{1}{2}^t(J'X^{-1}\alpha)' \\
\ \ & \ \ & I \\
\end{bmatrix}$$
So $u_n=u_{\alpha_n}(\dot{w}_0\bar{n}\dot{w}_0^{-1})$ is the last entry of $(-1)^{n-1}\frac{1}{2}X^{-1}\alpha$. Since only the last entry of $\alpha$ is non-zero, $u_n=(-1)^{n-1}\frac{1}{2}(\det X)^{-1}X_{n,n}^*a_n$, where $X^*_{n,n}$ is the $(n,n)$-th entry of the adjoint matrix of $X$. Since $X=Y{^tJ'^{-1}}=YJ'$, $Y$ is the matrix given as above, it is not hard to see that $X_{n,n}^*=(-1)^{n-1}\prod_{i=1}^{n-1}a_i.$ Therefore we have that $u_n=\frac{1}{2}(\det X)^{-1}\prod_{i=1}^n a_i.$ Also notice that $X=YJ'$ and $\det{J'}=1$, so $\det(X)=\det(Y)$. Hence $u_n=\frac{1}{2}(\det Y)^{-1}\prod_{i=1}^na_i.$ 

Next, we work on $zu^{-1}\bar{n}uz^{-1}.$ Let $z_0=\varpi^{d+f}u_n=\frac{1}{2}\varpi^{d+f}(\det Y)^{-1}\prod_{i=1}^na_i,$
let $t=(\det Y)^{-1}\prod_{i=1}^na_i\in F^{\times}$, then $z_0=\frac{1}{2}\varpi^{d+f}t$. Let $u=m(u_0,1)$ and $z=\alpha^{\vee}(z_0)=m(z_0I,1)=$ with $u_0\in U_n(F)\subset GL_n(F).$ Since $Y=X^tJ'$, so $X^{-1}=^tJ'Y^{-1}$, therefore $\bar{n}(X^{-1},\alpha)=\bar{n}(^tJ'Y^{-1},\alpha)=\begin{bmatrix}
I & \ \ & \ \ \\
-{^t\alpha}{^tY^{-1}} & 1 & \ \ \\
{^tJ'}Y^{-1} & {^tJ'}Y^{-1}\alpha & I \\
\end{bmatrix}$.
Then a direct calculation shows that $u^{-1}\bar{n}(^tJ'Y^{-1},\alpha)u=\bar{n}({^tJ}'{^tu}_0Y^{-1}u_0,u_0^{-1}\alpha)$.
This implies that
$zu^{-1}\bar{n}(^tJ'Y^{-1},\alpha)uz^{-1}=\bar{n}(z_0^{-2}\cdot\leftidx{^t}J'{^tu}_0Y^{-1}u_0,z_0u_0^{-1}\alpha)$.

We have $z_0=\frac{1}{2}\varpi^{d+f}t$, with $t=(\det Y)^{-1}\prod_{i=1}^na_i\in F^{\times}.$ Let $Y'=t^2Y$ and $\alpha'=t\alpha$. Recall that 
$\overline{N}_{0,\kappa}=\{\bar{n}=\bar{n}(\tilde{X},\alpha):  \varphi_{\kappa}(-\frac{1}{8}\varpi^{2(d+f)}\cdot\leftidx{^t}{\tilde{X}}{J'}^{-1})=1\}.$
Therefore
$\varphi_{\overline{N}_{o,\kappa}}(zu^{-1}\bar{n}uz^{-1})=\varphi_{\kappa}(-\frac{1}{8}\varpi^{2(d+f)}\cdot(\frac{1}{2}\varpi^{d+f}t)^{-2}\cdot\leftidx{^t}(^tJ'\leftidx{^t}u_0Y^{-1}u_0)J'^{-1})$
$=\varphi_\kappa(-\frac{1}{2}t^{-2}(^tu_0\leftidx{^t}Y^{-1}u_0J')J'^{-1})=\varphi_{\kappa}(-\frac{1}{2}t^{-2}\cdot\leftidx{^t}u_0\leftidx{^t}Y^{-1}u_0)=\varphi_\kappa(-\frac{1}{2}\leftidx{^t}u_0\leftidx{^t}Y'^{-1}u_0)$. We pick the long Weyl group representative of $G=\textrm{GL}_n$ by $\dot{w}_G=J'$, then
$$j_{\pi,\eta,\dot{w}_\theta,\kappa}(g)=j_{\pi,\eta,\dot{w}_\theta,\kappa}(-\frac{1}{2}\dot{w}_G\leftidx{^t}Y^{-1})$$$$=\int_{U_{M_H}}W_{\sigma_{\eta,s},v}(m(-\frac{1}{2}\dot{w}_G{^tY^{-1}},a(g))u)\varphi_{\overline{N}_{0,\kappa}}(zu^{-1}\bar{n}uz^{-1})\psi^{-1}(u)du$$$$=\int_{U_n} \eta(a(g))^{-1}\vert \det(g)\vert^{\frac{s}{2}} W_{{\pi},v}(-\frac{1}{2}\dot{w}_G{^tY^{-1}}u_0)\varphi_{\kappa}(-\frac{1}{2}\leftidx{^t}u_0{^tY'^{-1}}u_0)\psi^{-1}(u_0)du_0$$$$=\eta(a(g))^{-1}\vert \det(g)\vert^{\frac{s}{2}}\int_{U_n} W_{\pi,v}(gu)\varphi_{\kappa}(\leftidx{^t}u \dot{w}^{-1}_G g' u)\psi^{-1}(u)du.$$
where $g'=-\frac{1}{2}\dot{w}_G\leftidx{^t}Y'^{-1}$(so $g=t^2g'$), $U_n$ is the upper triangular unipotent matrices of size $n$ in $\textrm{GL}_n$. We also used the fact that $W_{\pi,v}(g)=\lambda(\pi(g)v)$, therefore $W_{\sigma_{\eta,s},v}(m(g,a(g)))=\lambda(\sigma_{\eta,s}(m(g,a(g))))=\eta(a(g))^{-1}\vert \det(g)\vert^{\frac{s}{2}}\lambda(\pi(g)v)=\eta(a(g))^{-1}\vert \det(g)\vert^{\frac{s}{2}}W_{\pi,v}(g)$. 

Moreover, substitute $u_n=\frac{1}{2}(\det Y)^{-1}\prod_{i=1}^na_i$ into the local coefficient formula, and use the orbit space measure we constructed earlier. After some simplifications, we obtain
\begin{proposition}
Let $\pi$ be an irreducible admissible $\psi$-generic representation of $\textrm{GL}_n$, lifted as a $\psi$-generic representation $\sigma$ of $M_H(F)\simeq \textrm{GL}_n(F)\times \textrm{GL}_1(F)$ by pull-back through the projection on the $\textrm{GL}_n$-factor. $\eta:F^\times\rightarrow \mathbb{C}^{\times}$ is a fixed continuous character. Define the representation $\sigma_\eta$ as before. Suppose that $\omega_{\sigma_\eta}(w_0\omega_{\sigma_\eta}^{-1})$ is ramified as a character of $F^{\times}$. Then for all sufficiently large $\kappa$, we have
$$C_{\psi}(s,\sigma_\eta)^{-1}=\gamma(ns, \omega_{\pi}^2,\psi)^{-1}\int_{F^{\times}\backslash R}j_{\pi,\eta,\dot{w}_\theta,\kappa}(-\frac{1}{2}\dot{w}_G\leftidx{^t}Y^{-1})$$$$\cdot\omega_{\pi}(4\det(Y)^2\prod_{i=1}^n a_i^{-2}) \vert\frac{1}{2}\vert^{\frac{n(n-s)}{2}}\vert \det(Y)\vert^{\frac{2ns-s-n}{2}} \prod_{i=1}^n\vert a_i \vert ^{i-1-ns}da_i$$
In addition, there exists a constant $\kappa_0$ such that for all $\kappa\ge \kappa_0$ and all $\chi$ such that $\eta^{-1}\omega_{\pi}\chi^n$ is ramified, we have
$$C_{\psi}(s,\sigma_\eta\otimes\chi)^{-1}=\gamma(ns, (\omega_{\pi}\chi^n)^2,\psi)^{-1}\int_{F^{\times}\backslash R}j_{\pi,\eta,\dot{w}_\theta,\kappa}(-\frac{1}{2}\dot{w}_G\leftidx{^t}Y^{-1})$$$$\cdot (\omega_{\pi}\chi^n)(4\det(Y)^2\prod_{i=1}^n a_i^{-2})\vert \frac{1}{2}\vert^{\frac{n(n-s)}{2}} \vert \det(Y)\vert^{\frac{2ns-s-n}{2}} \prod_{i=1}^n\vert a_i \vert ^{i-1-ns}da_i.$$
\end{proposition}

\subsection{Partial Bessel integrals}
For the proof of the stability of local coefficients, it is important to relate partial Bessel functions with partial Bessel integrals, which have nice asymptotic expansions under some conditions. 

Let $\textbf{G}$ be a split reductive group over $F$, and $G=\textbf{G}(F)$. Fix a Borel subgroup $\textbf{B}=\textbf{AU}$ and let $B$, $A$, $U$ denote the groups of their $F$-points respectively. Suppose $\Theta:G\longrightarrow G$ is an involution defined over $F$, i.e., $\Theta^2=1$ and $\Theta\neq 1$. Let $\pi$ be a $\psi$-generic supercuspidal representation of $G$ with its central character $\omega_\pi$. Let $f\in \mathcal{M}(\pi)$ be a matrix coefficient of $\pi$. Then $f\in C^\infty_c(G;\omega_\pi)$, the space of smooth functions on $G$ with compact support modulo the center $Z_G$ such that $f(zg)=\omega_\pi(z)f(g)$ for $z\in Z_G$ and $g\in G$. We associate $f$ with the Whittaker function $W^f(g)=\int_Uf(u'g)\psi^{-1}(u')du'$. The integral convergences since the coset $UZg$ is closed in $G$ and $f\in C^\infty_c(G;\omega_\pi)$. We can normalize it by choosing $f\in \mathcal{M}_\pi$ such that $W^f(e)=1$, where $e\in G$ is the identity element.

We define the twisted centralizer of $g\in G$
by
$$U_g=\{u\in U: \Theta(u^{-1})gu=g \}.$$
Suppose $G=Z_GG'$, write $g=zg'$ with $z\in Z_G$, $g\in G'$. Then we define the partial Bessel integral

$$B^G_{\tilde{\varphi}}(g,f)=\int_{U_g\backslash U}W^f(gu)\tilde{\varphi}(\Theta(u^{-1})g'u)\psi^{-1}(u)du,$$ where $\tilde{\varphi}$ is some cut-off function. Note that the above definitions can also be applied to any Levi subgroup $\textbf{M}$ of $\textbf{G}$.

If we apply the above settings to the case $\textbf{G}=\textrm{GL}_n$, $\Theta(g)=\dot{w}_G \leftidx{^t}g^{-1}\dot{w}_G^{-1}$, and $\tilde{\varphi}=L_{\dot{w}_G}\varphi$, where $L_{s}\varphi(g)=\varphi(s^{-1}g)$ is the left translation of $\varphi$, we obtain
$$B^G_{\varphi}(g, f)=\int_{U_g\backslash U} W^{f}(gu)\varphi(\leftidx{^t}u\dot{w}^{-1}_G g' u)\psi^{-1}(u)du,$$ which is the definition of partial Bessel integrals in [7]. And in this case the twisted centralizer of $g$ is given by
$$U_g=\{u\in U: \leftidx{^t}u\dot{w}_G^{-1}gu=\dot{w}_G^{-1}g\}$$
We will only use this definition for partial Bessel integrals and twisted centralizers in the rest part of the paper.

On the other hand, it is not hard to see by induction on the size $n$ that if $g=-\frac{1}{2}\dot{w}_G \leftidx{^t}Y^{-1}$ for $Y=Y(a_1,\cdots,a_n)$ with $(a_1\cdots,a_n)\in (F^{\times})^n$ as in the last part of section 5.4, the twisted centralizer $U_g$ is trivial. Hence the partial Bessel integral
$$B^G_{\varphi}(g, f)=\int_U W^{f}(gu)\varphi(\leftidx{^t}u\dot{w}^{-1}_G g' u)\psi^{-1}(u)du,$$
where $g=zg'$, $z\in Z$. Now choose $f\in \mathcal{M}(\pi)$ such that $W_{\pi,v}=W^f$, and $W^f(e)=1$. Take $\varphi=\varphi_{\kappa}$. From the calculations right before Proposition 5.9, we have $$j_{\pi,\eta,\dot{w}_\theta,\kappa}(g)=\eta(a(g))^{-1}\vert \det(g)\vert^{\frac{s}{2}}\int_{U_n} W_{\pi,v}(gu)\varphi_{\kappa}(\leftidx{^t}u \dot{w}^{-1}_G g' u)\psi^{-1}(u)du.$$ 

Therefore we obtain 
\begin{proposition}
Let $f\in \mathcal{M}(\pi)$ such that $W^f(e)=1$, and let $\varphi=\varphi_{\kappa}$, then
$$j_{\pi,\eta,\dot{w}_\theta,\kappa}(g)=\eta(a(g))^{-1}\vert\det(g)\vert^{\frac{s}{2}}\cdot B^G_{\varphi}(g,f),$$
for $g=-\frac{1}{2}\dot{w}_G\leftidx{^t}Y^{-1}$, where $Y=Y(a_1,\cdots,a_n)$ with all $a_i\in F^{\times}.$
\end{proposition}

Now we have successfully related our partial Bessel functions with partial Bessel integrals, whose asymptotic expansions will lead to the proof of stability.

\section{ANALYSIS OF PARTIAL BESSEL INTEGRALS}

Let $\textbf{G}$ be a split connected reductive group over $F$. Fix a Borel subgroup $\textbf{B}=\textbf{AU}$, and let $\textbf{U}^-$ be the unipotent group generated by all the negative roots. We use $G$, $B$, $A$, $U$, $U^-$ to denote their groups of $F$-points respectively. Denote the Weyl group of $\textrm{G}$ by $W$. We begin by stating some basic facts and properties.
\begin{itemize}
    \item {\textbf{\textrm{B(G)}.}} Define the subset of $W$ that supports Bessel functions by $B(G)=\{w\in W: \alpha\in \Delta \ \ s.t. \ \ w\alpha>0 \Rightarrow w\alpha\in \Delta \}$, or equivalently,
$B(G)=\{w\in W: w_G w=w_M \ \ \textrm{for} \ \ \textrm{some} \ \ \textrm{standard} \ \ \textrm{Levi} \ \ M\subset G\}$. We take the representatives $\dot{w}$ of $w\in B(G)$ so that $\dot{w}=\dot{w}_G\dot{w}_M^{-1}$.
Then there is a one-to-one correspondence between elements in $B(G)$ and Levi subgroups standard parabolic subgroups of $G$. To be precise, to a $w\in B(G)$ we associate $\theta_w^+=\{\alpha\in \Delta: w\alpha>0\}\subset \Delta$
which determines a standard parabolic subgroup $P_w=M_wN_w$, such that $M_w=Z_G(\cap_{\alpha\in \theta_w^+}\ker \alpha)$. We also have that $\theta^+_w=\theta_{w_M}^-=\Delta_M\subset \Delta$, where $w_M$ is the long Weyl group element of $M$.

\item
$\textbf{U}_w^+, \textbf{U}_w^-.$ For each $w\in W$ we define two unipotent subgroups $U^+_w$ and $U_w^-$ of $U$ to be
$U_w^+=\{u\in U: wuw^{-1}\in U\}$
and $U_w^-=\{u\in U:wuw^{-1}\in U^-\}.$ In other words, $U_w^+$(resp. $U_w^-$) is generated by those roots that are made positive(resp. negative) by $w$. One can see that 
$U^+_w=U\cap w^{-1}Uw$, $U^-_w=U\cap w^{-1} U^-w,$ and $U=U^+_wU^-_w$.
Moreover, if $w\in B(G)$, suppose $\dot{w}=\dot{w}_G\dot{w}_M^{-1}$, so $w$ associates the Levi $M=M_w$ of $G$. Let $U_M=U\cap M$, then $U_M$ is the standard maximal unipotent subgroup of $M$. If we denote $N_M$ to be the unipotent radical of the corresponding parabolic, i.e., $P_M=MN_M$. Then $U=U_MN_M$. Now for $w=w_M$, we can see that
$U_{w_M}^+=N_M, U_{w_M}^-=U_M$ and for $w=w_G$, we have
$U_{w_G}^+=\{e\}, U_{w_G}^-=U$. In general for $w=w_Gw_M$ we have
$U_w^+=U_M, U_{w}^-=N_M.$

\item\textbf{Bessel distance} For $w,w'\in B(G)$ with $w>w'$ we define the Bessel distance as follows:
$d_B(w,w')=max\{m: \exists w_i\in B(G)\ \ s.t \ \ w=w_m>w_{m-1}>\cdots>w_0=w' \}$. And if we denote $\Delta_{M_w}$ to be the set of simple roots associated with the standard Levi $M_w$, we have $\Delta_{M_w}\subset \Delta_{M_w'}$ and $d_B(w,w')=\vert \Delta_{M_{w'}}-\Delta_{M_w}\vert$.

\item \textbf{Bruhat order} For $w\in W$ we denote the Bruhat cell by $C(w)=UwAU$, we define the Bruhat order on $W$ by $w\leq w'\Longleftrightarrow C(w)\subset \overline{C(w')}.$

\item\textbf{The relevant torus $\textbf{A}_w$}.
For $w\in B(G)$, define $A_w=\{a\in A: a \in \cap_{\alpha\in \theta_w^+}\ker \alpha\}^\circ\subset A$, which is also the center $Z_{M_w}$ of $M_w$.

\item \textbf{The relevant Bruhat cell} $\textbf{C}_{r}(\dot{w})$. We call $C_r(\dot{w})=U\dot{w}A_wU_w^-$ the relevant part of the Bruhat cell $C(w)$. Note that $C_r(\dot{w})$ depends on the choice of the representative $\dot{w}$ of $w$.
\item \textbf{Transverse tori} Let $w,w'\in B(G)$ and let $M=M_w$ and $M'=M_{w'}$ be their associated Levi subgroups respectively. Suppose $w'\leq w$. Then $M\subset M'$ and $A_{w'}\supset A_w$. Let $A^{w'}_w=A_w\cap M^d_{w'}=Z_M\cap (M')^d$. Note that in particular $A^w_w=Z_M \cap M^d$ is finite since $M$ is reductive and in general we have that $M^d\cap R(M)=M^d\cap Z^0$ is finite, where $Z^0$ is the connected component of $Z$ and $R(M)$ is the radical of $M$. In the case of $G=GL_n$ the center is connected, and $A^w_w$ consists of certain roots of unity on the diagonal blocks of $M$. Similarly $A^{w'}_w\cap A_{w'}=A^{w'}_{w'}$ is finite and the subgroup $A^{w'}_w A_{w'}\subset A_w$ is open and of finite index. So this decomposition is essentially a "transfer principal" for relevant tori, from the larger one $A_w$ to the smaller one $A_{w'}$ which differs by the transverse torus $A^{w'}_w$, on which the germ functions live on, as we will see later.
\end{itemize}

Here are some useful properties of $B(G)$:

\textbf{1,} For $w,w'\in B(G)$. Then $w'\leq w\Longleftrightarrow M_w\subset M_{w'}\Longleftrightarrow A_w\supset A_{w'}.$ (Lemma 5.1 in [7])

\textbf{2,} For each $w\in B(G)$, say $\dot{w}=\dot{w}_G\dot{w}_M^{-1}$. Then for all $u\in U_w^+=U_M$, we have $\psi(\dot{w}u\dot{w}^{-1})=\psi(u)$, where $\psi$ is the generic character. (Proposition 5.1 in [7])

\textbf{3,} Let 
$\Omega_w=\bigsqcup_{w\leq w'} C(w'),$ we see that $\Omega_w$ is invariant under the two-sided action of $U\times U$ and as in Lemma 5.2 in [7], $\Omega_w$ is an open subset of $G$ and $C(w)$ is closed in $\Omega_w$. 

As stated in [7] we also have:
\begin{lemma}
Suppose $w\in B(G)$ is associated with a standard Levi $M$ of $G$, then we have 
$\Omega_{w}\simeq U^-_{w^{-1}}\times \dot{w}M\times U_w^-.$ This decomposition is unique.
\end{lemma}
Suppose $\pi$ is a generic representation of $M(F)$. Let $C^{\infty}_c(\Omega_w;w_{\pi})$ denote the space of smooth functions of compact support modulo the center $Z$, so $\forall g\in \Omega_w$ and $z\in Z$, $f(zg)=w_{\pi}(z)f(g)$. Since $\Omega_w$ is open in $G$, we have $C_c^{\infty}(\Omega_w;w_{\pi})\subset C^{\infty}_c(G;w_{\pi}).$ 

\begin{lemma} There is a surjective map: $C^{\infty}_c(M;w_{\pi})\rightarrowdbl  C^{\infty}_c(\Omega_{w'};w_{\pi})$ given by
$h=h_f\mapsto f$
where $h(m)=h_f(m)=\int_{U^-_{w'}}\int_{U^-_{w'^{-1}}}f(x^-\dot{w}mu^-)\psi^{-1}(x^-u^-)dx^-du^-.$
\end{lemma}
\begin{proof}
See Lemma 5.9 [7].
\end{proof}
\subsection{Partial and full Bessel integrals} 

Let $w\in B(G)$ and $g=u_1\dot{w}au_2\in C_r(\dot{w})$, the relevant cell associated to $w$, which depends on the choice of the representative $\dot{w}$ of $w$.  Let $M=M_w$ be the Levi subgroup of $G$ such that $w=w_Gw_M.$ We have 
\begin{lemma} For $g=u_1\dot{w}au_2\in C_r(\dot{w})$ with $w=w_Gw_M\in B(G)$, then $$U_g\subset u_2^{-1}U^+_w u_2=u_2^{-1} U_M u_2$$
\end{lemma}
\begin{proof}
$u\in U_g\Longleftrightarrow ^tu\dot{w}_G^{-1}u_1\dot{w}_G\dot{w}_M^{-1}au_2u=\dot{w}_G^{-1}u_1\dot{w}_G\dot{w}_M^{-1}au_2$. Let $\overline{u_1}=\dot{w}_G^{-1}u_1\dot{w}_G\in U^-$, then this is equivalent to
$ (\overline{u_1})^{-1}{^tu} \overline{u_1}\dot{w}_M^{-1}au_2uu_2^{-1}=\dot{w}_M^{-1}a$, which is the same as $ (\overline{u_1})^{-1}{^tu}\overline{u_1}=\dot{w}_Mau_2u^{-1}u_2^{-1}a^{-1}\dot{w}_M$.

Notice that $(\overline{u_1})^{-1}{^tu}\overline{u_1}\in U^-$, and $au_2u^{-1}u_2^{-1}a^{-1}\in U$. This implies that 
$au_2u^{-1}u_2^{-1}a^{-1}\in U_{w_M}^-=U_M$.
Therefore $ u_2u^{-1}u_2^{-1}\in a^{-1}U_Ma=U_M$ since $a\in A_w$.
So $u^{-1}\in u^{-1}_2U_Mu_2$, thus $u\in u^{-1} U_M u_2=u_2^{-1}U_{w}^+u_2$.
\end{proof}
Next, we will show an equality that relates partial Bessel integrals with full Bessel integrals. 

First, decompose $U=u_2^{-1}Uu_2=(u^{-1}_2U^+_wu_2)(u_2^{-1}U_w^-u_2)$ and for $u\in U$, write $u=u'^+(u_2^{-1}u^-u_2)$ with $u'^+=u_2^{-1}u^+u_2$ where $u^+\in U^+_w$, and $u^-\in U^-_w.$
Since by lemma 6.3, $U_g\subset u_2^{-1}U_w^+u_2$, we have 
$$B^G_{\varphi}(g,f)=\int_{U_g\backslash u_2^{-1}U^+_wu_2}\int_{U^-_w}\int_U f(xgu'^+u_2^{-1}u^-u_2)$$$$\cdot\varphi(^t{(u_2^{-1}u^-u_2)}^t{u'^+}\dot{w}_G^{-1}g'u'^+u_2^{-1}u^-u_2) \psi^{-1}(x)\psi^{-1}(u'^+u_2^{-1}u^-u_2)dxdu^-du'^+$$
$$=\int_{U_g\backslash u_2^{-1}U^+_wu_2}\int_{U^-_w}\int_U f(xu_1\dot{w}a(u_2u'^+u_2^{-1})u^-u_2)$$$$\cdot\varphi(^t{(u_2^{-1}u^-u_2)}^t{u'^+}\dot{w}_G^{-1}u_1\dot{w}a'(u_2u'^+u_2^{-1}) u^-u_2) \psi^{-1}(x)\psi^{-1}(u'^+u_2^{-1}u^-u_2)dxdu^-du'^+$$
$$=\int_{U_g\backslash u_2^{-1}U^+_wu_2}\int_{U^-_w}\int_U f(xu_1\dot{w}au^+u^-u_2)\varphi(^t{u_2}^t{u^-}^t{u^+}^t{u_2^{-1}}\dot{w}_G^{-1}u_1\dot{w}a'u^+u^-u_2)$$
$$\cdot \psi^{-1}(x)\psi^{-1}(u_2^{-1}u^+u^-u_2)dxdu^-du^+.$$

Now since $a\in A_w$, we have $au^+=u^+a$. So the above integral
$$=\int_{U_g\backslash u_2^{-1}U^+_wu_2}\int_{U^-_w}\int_U f(xu_1(\dot{w}u^+\dot{w}^{-1})\dot{w}au^-u_2)\varphi(^t{u_2}^t{u^-}^t{u^+}^t{u_2^{-1}}\dot{w}_G^{-1}u_1\dot{w}a'u^+u^-u_2)$$
$$\cdot \psi^{-1}(x)\psi^{-1}(u_2^{-1}u^+u^-u_2)dxdu^-du^+.$$

Let $x'=xu_1(\dot{w}u^+ \dot{w}^{-1})$ and $u'^-=u^-u_2$, then $dx'=dx$ and $du'^-=du^-$.

After this change of variable we have the above integral 
$$=\int_{U_g\backslash u_2^{-1}U^+_wu_2}\int_{U^-_w}\int_U f(x'\dot{w}au'^-)\varphi(^t{u'^-}^t{u^+}^t{u_2^{-1}}\dot{w}_G^{-1}u_1\dot{w}a'u^+u'^-)$$
$$\cdot \psi^{-1}(x'(u_1\dot{w}u^+\dot{w}^{-1})^{-1})\psi^{-1}(u_2^{-1}u^+u'^-)dxdu'^-du^+$$
$$=\psi(u_1)\psi(u_2)\int_{U_g\backslash u_2^{-1}U^+_wu_2}\int_{U^-_w}\int_U f(x'\dot{w}au'^-)\varphi(^t{u'^-}^t{u^+}^t{u_2^{-1}}\dot{w}_G^{-1}u_1\dot{w}a'u^+u'^-)$$
$$\cdot \psi^{-1}(x')\psi(\dot{w}u^+\dot{w}^{-1})\psi^{-1}(u^+)\psi^{-1}(u'^-)dxdu'^-du^+.$$

By compatibility of $\psi$ and $\dot{w}$, we have
$\psi(\dot{w}u^+\dot{w}^{-1})=\psi(u^+)$, so
$$B^G_{\varphi}(g,f)=\psi(u_1)\psi(u_2)\int_{U_g\backslash u_2^{-1}U^+_wu_2}\int_{U^-_w}\int_U f(x'\dot{w}au'^-)$$$$\cdot\varphi(^t{u'^-}^t{u^+}^t{u_2^{-1}}\dot{w}_G^{-1}u_1\dot{w}a'u^+u'^-) \psi^{-1}(x')\psi^{-1}(u'^-)dxdu'^-du^+$$

Now take $f\in C_c^{\infty}(\Omega_w;w_{\pi})$. Since $g$ is fixed, $a$ is fixed. Since by Lemma 5.2 of [7], $C(w)$ is closed in $\Omega_w$, there exists open compact subsets $U_1\subset U$ and $U_2\subset U_w^-$ such that the support of the function
$(x,u^-)\mapsto f(x\dot{w}au^-)$ lies in $U_1\times U_2$. Take $N$ large enough such that $\varphi=\varphi_N$ is invariant under the left and right action of $U_2$ as in Lemma 4.2 of [7], i.e., 
$\varphi(^tugu)=\varphi(g)$ for all $u\in U_2$. Then
we have 
$\varphi(^t{u'^-}^t{u^+}^t{u_2^{-1}}\dot{w}_G^{-1}u_1\dot{w}a'u^+u'^-)=\varphi(^t{u^+}^t{u_2^{-1}}\dot{w}_G^{-1}u_1\dot{w}a'u^+)$.

Define $$\tilde{\varphi}^G_M(g')=\int_{U_g\backslash u_2^{-1}U^+_wu_2}\varphi(^t{u^+}\leftidx{^t}{u_2^{-1}}u_1\dot{w}a'u^+)du^+,$$
then $$\tilde{\varphi}^G_M(g')=\int_{U_g\backslash u_2^{-1}U^+_wu_2}\varphi(^t{u^+}\leftidx{^t}{u_2^{-1}}\dot{w}_G^{-1}g'u_2^{-1}u^+)du^+$$
$$=\int_{U_g\backslash u_2^{-1}U^+_wu_2}\varphi(^t{u_2^-}\leftidx{^t}{u'^+}\dot{w}_G^{-1}g'u'^+u_2^{-1})du'^+.$$
So we have
$$B^G_{\varphi}(g,f)=\psi(u_1)\psi(u_2)\tilde{\varphi}^G_M(g')\int_{U_w^-}\int_Uf(x\dot{w}au^-)\psi^{-1}(x)\psi^{-1}(u^-)dxdu^-$$
$$=\psi(u_1)\psi(u_2)\tilde{\varphi}^G_M(g')B^G(\dot{w}a,f)=\tilde{\varphi}^G_M(g')B^G(g,f)$$

We just showed the following result:
\begin{lemma}For $w\in B(G)$ and any $g=u_1\dot{w}au_2\in C_r(\dot{w})$, $g'=u_1\dot{w}a'u_2$ where $a=za'$, $z\in Z$ and $a'\in A'$, we have 
$$B^G_{\varphi}(g,f)=\tilde{\varphi}^G_M(g')B^G(g,f).$$
where
$$B^G(g,f)=\int_{U\times U_{w}^-}f(xgu^-)\psi^{-1}(x)\psi^{-1}(u^-)dxdu^-$$ is the full Bessel integral and  $\tilde{\varphi}^G_M(g')$ as defined above.
\end{lemma}
\subsection{Twisted centralizer and transfer principle}
For $\textbf{G}=\textrm{GL}_n$, $G=\textbf{G}(F)$, and $f\in C^{\infty}_c(G;w_{\pi})$, we defined the partial Bessel integral as $$B^G_\varphi(g,f)=\int_{U_g\backslash U}W^f(gu)\varphi(^tu\dot{w}_G^{-1}g'u)\psi^{-1}(u)du$$
$$=\int_{U_g\backslash U}\int_U f(xgu)\varphi(^tu\dot{w}_G^{-1}g'u)\psi^{-1}(x)\psi^{-1}(u)dxdu,$$ where $\varphi$ is the characteristic function of some compact neighborhood of zero in $\textrm{Mat}_n(F)$. Now for any Levi subgroup $M$ of $G$, we define the twisted centralizer of $m\in M$ in $U_M=U\cap M$ to be
$U_{M,m}=\{u\in U_M: \leftidx{^t}u\dot{w}_M^{-1}mu=\dot{w}_M^{-1}u\}$. Let $h\in C^{\infty}_c(M;w_{\pi})$, the space of smooth functions of compact support modulo $Z$ on $M$, satisfying $h(zm)=w_{\pi}(z)h(m),$ for $z\in Z=Z_G$. The partial Bessel integral on $M$ is then given by
$$B^M_{\varphi}(m,h)=\int_{U_{M,m}\backslash U_M}\int_{U_M}h(xmu)\varphi(^tu\dot{w}_M^{-1}m'u)\psi^{-1}(xu)dxdu,$$
where $m'$ is obtained by $m$ from the decomposition $Z_M=ZA_M'$, i,e., if $m\in U_M \dot{w} A_M U^-_{M,w}$, then $m'\in U_M\dot{w} A_M' U^-_{M,w},$ $z\in Z$ and $m=zm'$.

Now Let $L\subset M \subset G$ be standard Levi subgroups of G, as before let $w_G$, $w_M$ and $w_L$ be the long Weyl group elements of $G, M$ and $L$ respectively. And let $\dot{w}_G,\dot{w}_M,$ and $\dot{w}_L$ be their representatives chosen to be compatible with $\psi$ as before. Now denote $w^M_L=\dot{w}_M\cdot \dot{w}_L^{-1}$, similarly if $M$ is replaced by $G$.

Take $g\in C_r(w^G_L),$ the relevant cell for $w^G_L$. Suppose $g=u_1 \dot{w}^G_L a u_2$ is the Bruhat decomposition of $g$, where $a\in A_{w^G_L}=Z_L$. Decompose $u_1=u_1^-u_1^+\in U^-_{(w')^{-1}} U^+_{(w')^{-1}}=U$, also $u_2=u_2^+u_2^-\in U^+_{w'}U_{w'}^-=U_M N_M=U$, where $w'=w^G_M$. Therefore $g=u_1w'au_2=u_1^-u_1^+w'au_2^+u_2^-=u_1^-w'(w'^{-1})u_1^+w'au_2^+u_2^-$. Since $C_r(w^G_L)\subset \Omega_{w'}$, by Lemma 6.1, $g$ has a unique decomposition $g=u_1^-w'mu_2^-$,
 $u_1^-\in U_{(w')^{-1}}^-$, and $u_2^-\in U_{w'}^-$. On the other hand, since $w'(w'^{-1}u_1^+w')w'^{-1}=u_1^+\in U$, so by definition $(w'^{-1})u_1^+w'\in U_{w'}^+=U_M\subset M$. Therefore $(w'^{-1})u_1^+w'au_2^+\in M$. Now compare the two decompositions and by uniqueness of Lemma 6.1, we see that $m=w'^{-1}u_1^+w'au_2^+$.

Now we prove the following transfer principal for partial Bessel integrals: 

\begin{proposition}(\textbf{Transfer principle for partial Bessel integrals})For any given $g\in C_r(w^G_L)$, suppose $g=u_1^- w' m u_2^-$, then 
$$B^G_{\varphi}(g, f)=\psi(u_1^-)\psi(u_2^-) B^M_{\varphi}(u_1^-,u_2^-,m, h_f).$$
where
$$B_\varphi^M(u_1^-,u_2^-, m,h_f)=\int_{U_{M,m}\cap n_0U_{M,m}n_0^{-1}\backslash U_M}\int_{U_M}h_f(x'mu')$$
$$\cdot\varphi(^t{u'}^t{n_0}\dot{w}_M^{-1}m'u')\psi^{-1}(x')\psi^{-1}(u')dx'du'$$ and
$h_f\mapsto f$ through the surjective map: 
$C_c^{\infty}(\Omega_{w'};w_{\pi})\twoheadrightarrow C_c^{\infty}(M; w_{\pi}),$ and $n_0=\leftidx{^t}(\overline{u_1^-})(u_2^-)^{-1}\in N_M$.
\end{proposition}

To prove this, we first need to deal with the twisted centralizers in the above two partial Bessel integrals.

\begin{lemma} Suppose that we have a chain of standard Levi subgroups $L\subset M\subset G$ with associated Weyl group elements $w^G_L\in B(G)$ and $w^M_L\in B(M)$ respectively. Then for $g\in C_r(w^G_L)$ with $g=u_1w^G_La u_2=u_1^-w'mu_2^-\in C_r(w^G_L)\subset \Omega_{w'}\simeq U^-_{(w')^{-1}}\times \dot{w}'M\times U_{w'}^-,$ where $a\in A_{w^G_L}=Z_L$ and $w'=w^G_M,$ $u=u_1^-u_1^+\in U^-_{(w')^{-1}} U^+_{(w')^{-1}}=U$, also $u_2=u_2^+u_2^-\in U^+_{w'}U_{w'}^-=U_M N_M=U.$

Then the twisted centralizer of $g$ and $m$ satisfies
$$U_g=(\leftidx{^t}(\overline{u_1^-})^{-1}U_{M,m} \leftidx{^t}{\overline{u_1^-}})\cap ((u_2^-)^{-1}U_{M,m}u_2^-)$$
where $\overline{u_1^-}=\dot{w}_G^{-1}u_1^- \dot{w}_G.$
\end{lemma}

\begin{proof}
We have $g=u_1 w^G_L a u_2 =u^-_1 u_1^+w' w^M_L a u_2^+u_2^-=u_1^-w'(w'^{-1}u_1^+ w' w^M_L a u_2^+)u_2^-=u_1^-w'mu_2^-$ where $m=w'^{-1}u_1^+w'w^M_Lau_2^+$. Notice that we have $w'^{-1}U^+_{(w')^{-1}}w'=U^+_{w'}=U_M.$ The above decomposition is unique by Lemma 6.1.

Now we show that $\overline{u^-_1}=\dot{w}_G^{-1}u_1^-\dot{w}_G\in N_M^-$, or equivalently, $^t\overline{u_1^-}\in U^-_{(w')}=N_M$. To see this, since $u_1^-\in U^-_{(w')^{-1}}\subset U$, $\overline{u^-_1}=\dot{w}_G^{-1}u^-_1\dot{w}_G\in U^-$. On the other hand, we have that $w'^{-1}u_1^-w'=\dot{w}_M\dot{w}_G^{-1}u_1^-\dot{w}_G \dot{w}_M^{-1}=\dot{w}_M\overline{u^-_1}\dot{w}_M^{-1}\in U^-$ by the definition of $u_1^-$. Taking transpose and using the fact that $^t\dot{w}_M=\dot{w}_M^{-1}$ by the way we choose the Weyl group representatives, we see that this is the same as saying $\dot{w}_M\leftidx{^t}{\overline{u_1^-}}\dot{w}_M^{-1}\in U$, this shows that $^t\overline{u^-_1}\in U_{\dot{w}_M}^+=N_M$.

Next, we see that $$u\in U_g\Longleftrightarrow\dot{w}_G^tu\dot{w}_G^{-1}gu=g$$
$$\Longleftrightarrow\dot{w}_G\leftidx{^t}{u^-}\leftidx{^t}{u^+}\dot{w}_G^{-1}u_1^-w'mu_2^-u^+u^-=u_1^-w'mu_2^-$$
$$\Longleftrightarrow \dot{w}_G\leftidx{^t}{u^-}\leftidx{^t}{u^+}\dot{w}_G^{-1}u^-_1\dot{w}_G\dot{w}_M^{-1}mu_2^-u^+u^-=u_1^-w'mu_2^- \cdots\cdots (w'=\dot{w}_G\dot{w}_M^{-1})$$
$$\Longleftrightarrow \dot{w}_G\leftidx{^t}{u^-}\leftidx{^t}{u^+}\overline{u^-_1}\dot{w}_M^{-1}mu^-_2 u^+u^-=u_1^-w'mu_2^-\cdots\cdots(\overline{u_1^-}=\dot{w}_G^{-1}u^-_1\dot{w}_G\in N_M)$$
$$\Longleftrightarrow \dot{w}_G\leftidx{^t}{u^-}\leftidx{^t}{u^+}\overline{u_1^-}(\leftidx{^t}{u^+})^{-1}\leftidx{^t}{u^+}\dot{w}_M^{-1}mu_2^-u^+u^-=u_1^-w'mu_2^-$$
$$\Longleftrightarrow \dot{w}_G\leftidx{^t}{u^-}(\leftidx{^t}{u^+}\overline{u_1^-}(\leftidx{^t}{u^+})^{-1})\dot{w}_G^{-1}(\dot{w}_G\dot{w}_M^{-1})\dot{w}_M\leftidx{^t}{u^+}\dot{w}_M^{-1}mu^+(u^+)^{-1}u_2^-u^+u^-=u_1w'mu_2^-$$
$$\Longleftrightarrow (\dot{w}_G\leftidx{^t}{u^-}(\leftidx{^t}{u^+}\overline{u_1^-}(\leftidx{^t}{u^+})^{-1})\dot{w}_G^{-1})w'(\dot{w}_M\leftidx{^t}{u^+}\dot{w}_M^{-1}mu^+)((u^+)^{-1}u_2^-u^+u^-)=u_1^-w'mu_2^-$$ We call the last equality $(A)$.
Now notice that $\leftidx{^t}{u^+}\overline{u_1^-}(\leftidx{^t}{u^+})^{-1}=\leftidx{^t}{(({u^+})^{-1}\leftidx{^t}{\overline{u_1^-}}u^+)}$, and $(({u^+})^{-1}\leftidx{^t}{\overline{u_1^-}}u^+)\in N_M$ since we showed that $\leftidx{^t}{\overline{u_1^-}}\in N_M$ and $u^+\in U_M$, $U_M$ normalizes $N_M$. So we have $\leftidx{^t}{u^+}\overline{u_1^-}(\leftidx{^t}{u^+})^{-1}\in N_M^-$. 

Next, we claim that $\dot{w}_G\leftidx{^t}{u^-}(\leftidx{^t}{u^+}\overline{u_1^-}(\leftidx{^t}{u^+})^{-1})\dot{w}_G^{-1}\in U_{(w')^{-1}}^-$.
To see this, notice that this is equivalent to $ w'^{-1}\dot{w}_G\leftidx{^t}{u^-}(\leftidx{^t}{u^+}\overline{u_1^-}(\leftidx{^t}{u^+})^{-1})\dot{w}_G^{-1} w'\in U^-$, which is the same as saying $\dot{w}_M\leftidx{^t}{u^-}\leftidx{^t}{u^+}\overline{u_1^-}(\leftidx{^t}{u^+})^{-1}\dot{w}_M^{-1}\in U^-$, since $ w'^{-1}=\dot{w}_M\dot{w}_G^{-1}.$ Also note that $^t{u^-}\in N_M^-$ and $\leftidx{^t}{u^+}\overline{u_1^-}(\leftidx{^t}{u^+})^{-1}\in N_M^-,$ and it is not hard to see that $\dot{w}_MN^-_M\dot{w}_M^{-1}\subset U^-$, so the claim follows.  

Moreover, clearly we have $\dot{w}_M\leftidx{^t}{u^+}\dot{w}_M^{-1}mu^+\in M$ and $(u^+)^{-1}u_2^-u^+u^-\in N_M$.

Summarize what we obtained so far, we have 
$\dot{w}_G\leftidx{^t}{u^-}(\leftidx{^t}{u^+}\overline{u_1^-}(\leftidx{^t}{u^+})^{-1})\dot{w}_G^{-1}\in U_{(w')^{-1}}^-$,
$\dot{w}_M\leftidx{^t}{u^+}\dot{w}_M^{-1}mu^+\in M$ and 
$(u^+)^{-1}u_2u^+u^-\in U_{w'}^-$.
In addition, by the uniqueness of the decomposition
$\Omega_{w'}=U^-_{(w')^{-1}}\times w'M\times U_{w'}^-$ as in Lemma 6.1 and equality $(A)$, the following three equalities hold at the same time:

$(a),\ \ \dot{w}_G\leftidx{^t}{u^-}(\leftidx{^t}{u^+}\overline{u_1^-}(\leftidx{^t}{u^+})^{-1})\dot{w}_G^{-1}=u^-_1;$

$(b), \ \ \dot{w}_M\leftidx{^t}{u^+}\dot{w}_M^{-1}mu^+=m$;

$(c), \ \ (u^+)^{-1}u_2^-u^+u^-=u^-_2$. 

Notice that $(a)\Longleftrightarrow \leftidx{^t}{u}\overline{u_1^-}(\leftidx{^t}{u^+})^{-1}=\overline{u_1^-}\Longleftrightarrow (u^+)^{-1}\leftidx{^t}{\overline{u_1^-}}u=\leftidx{^t}{\overline{u_1^-}}\Longleftrightarrow u^+=\leftidx{^t}{\overline{u_1^-}}u(\leftidx{^t}{\overline{u_1^-}})^{-1}$, hence $\leftidx{^t}{\overline{u_1^-}}u(\leftidx{^t}{\overline{u_1^-}})^{-1}=u^+\in U_M$.
On the other hand, from $(b)$ we see that $u^+\in U_{M,m}$, so $(a)$\&$(b)$ implies that $\leftidx{^t}{\overline{u_1^-}}u(\leftidx{^t}{\overline{u_1^-}})^{-1}\in U_{M,m}$. Since we started with $u\in U_g$, we see that $U_g\subset (\leftidx{^t}{\overline{u_1^-}})^{-1} U_{M,m} \leftidx{^t}{\overline{u_1^-}}. $
Similarly, $(c)\Longleftrightarrow u_2^-u(u_2^-)^{-1}=u^+\Longrightarrow u_2^-u(u_2^-)^{-1}=u^+\in U_M$ and again by $(b)$ we have $u^+\in U_{M,m}$, therefore $u_2^-u(u_2^-)^{-1}\in U_{M,m}$. So $(b)\&(c)$ implies that $U_g\subset (u_2^-)^{-1}U_{M,m}u_2^-$. We conclude that
$U_g\subset(\leftidx{^t}(\overline{u_1^-})^{-1}U_{M,m} \leftidx{^t}{\overline{u_1^-}})\cap ((u_2^-)^{-1}U_{M,m}u_2^-)$.

Conversely, if $u=\leftidx{^t}(\overline{u_1^-})^{-1}u'\leftidx{^t}{\overline{u_1^-}}=(u_2)^{-1}u''u_2^-$ with $u',u''\in U_{M,m}$, we see that 
$u^+u^-=u=u'(u')^{-1}\leftidx{^t}(\overline{u_1^-})^{-1}u'\leftidx{^t}{\overline{u_1^-}}=u'((u')^{-1}\leftidx{^t}(\overline{u_1^-})^{-1}u')\leftidx{^t}{\overline{u_1^-}}$.
Since $u'\in U_M,(u')^{-1}\leftidx{^t}(\overline{u_1^-})^{-1}u'\in N_M$, $U=U_M\times N_M$ and $U_M\cap N_M=\{1\}$, $u^+=u'$ and $u^-=(u')^{-1}\leftidx{^t}(\overline{u_1^-})^{-1}u'\leftidx{^t}{\overline{u_1^-}}.$ Replace $\leftidx{^t}{\overline{u_1^-}}$ by $u_2^-\in N_M$ in the above argument we also obtain $u^+=u''$. This implies $(b)$. 

Moreover, from $u=\leftidx{^t}(\overline{u_1^-})^{-1}u'\leftidx{^t}{\overline{u_1^-}}=(u_2)^{-1}u''u_2^-$, we see that 
$\leftidx{^t}(\overline{u_1^-})u\leftidx{^t}{\overline{u_1^-}}^{-1}=u'=u^+\Longleftrightarrow (a)$ and $u_2^-u(u_2^-)^{-1}=u''=u^+\Longleftrightarrow (c).$ 
Since $u\in U_g$ is equivalent to $(a),(b),(c)$ to hold at the same time, hence it proves the reverse inclusion 
$U_g\supset (\leftidx{^t}(\overline{u_1^-})^{-1}U_{M,m} \leftidx{^t}{\overline{u_1^-}})\cap ((u_2^-)^{-1}U_{M,m}u_2^-)$.

So we finally obtain that  
$U_g=(\leftidx{^t}(\overline{u_1^-})^{-1}U_{M,m} \leftidx{^t}{\overline{u_1^-}})\cap ((u_2^-)^{-1}U_{M,m}u_2^-)$.
\end{proof}

Remark: From the above argument, $u=\leftidx{^t}(\overline{u_1^-})^{-1}u'\leftidx{^t}{\overline{u_1^-}}=(u_2)^{-1}u''u_2^-\in U_g=(\leftidx{^t}(\overline{u_1^-})^{-1}U_{M,m} \leftidx{^t}{\overline{u_1^-}})\cap ((u_2^-)^{-1}U_{M,m}u_2^-)$ automatically implies that $$u'=u''\in U_{M,m}\cap \textrm{Cent}(\leftidx{^t}(\overline{u_1^-})u_2^{-1}).$$

Now we can show the proposition based on the above lemma:
\begin{proof}(Proposition 6.5)
For any given $g\in C_r(w^G_L)$, $$g=u_1w^G_L a u_2=u_1^- w' m u_2^-\in C_r(w^G_L)\subset \Omega_{w'}=U_{w'^{-1}}^-\times w'M\times U_{w'}^-$$
By lemma 6.6, $U_g= (\leftidx{^t}(\overline{u_1^-})^{-1}U_{M,m} \leftidx{^t}{\overline{u_1^-}})\cap ((u_2^-)^{-1}U_{M,m}u_2^-)$. 
To simplify the notations, we denote $n=\leftidx{^t}({\overline{u_1^-}})^{-1}$ and $n_0=\leftidx{^t}{\overline{u_1^-}}(u_2^-)^{-1}$, then they both lie in $N_M$. Since $n\in U$, we have $U=nUn^{-1}=(nU_Mn^{-1})\times(nN_Mn^{-1})$.
For $f\in C^{\infty}_c(\Omega_{w'};w_{\pi})$ we have
$B_{\varphi}^G(g,f)=\int_{U_g\backslash U}\int_U f(xgu)\varphi(^tu\dot{w}_G^{-1}g'u)\psi^{-1}(xu)dxdu$. Make a change of variable $u\mapsto nun^{-1}$, and decompose $U$ as $U=nUn^{-1}=(nU_Mn^{-1})\times(nN_Mn^{-1})$. Then $U_g=(\leftidx{^t}(\overline{u_1^-})^{-1}U_{M,m} \leftidx{^t}{\overline{u_1^-}})\cap ((u_2^-)^{-1}U_{M,m}u_2^-)=n(U_{M,m}\cap n_0U_{M,m}n_0^{-1})n^{-1}.$ We can rewrite the integral as
$$B^G_{\varphi}(g,f)=\int_{n(U_{M,m}\cap n_0U_{M,m}n_0^{-1})n^{-1}\backslash nU_Mn^{-1}}\int_{U_{w'}^-}\int_{U^+_{(w')^{-1}}}\int_{U^-_{(w')^{-1}}}$$
$$f(x^-x^+u_1^-w'mu_2^-nu^+u^-n^{-1})\varphi(^t{n^{-1}}^t{u^-}^t{u^+}^tn\dot{w}_G^{-1}u_1^-\dot{w}_G\dot{w}_M^{-1}m'u_2^-nu^+u^-n^{-1})$$
$$\cdot \psi^{-1}(x^-x^+)\psi^{-1}(nu^+u^-n^{-1})dx^-dx^+du^-du^+$$
$$=\int_{U_{M,m}\cap n_0U_{M,m}n_0^{-1}\backslash U_M}\int_{N_M}\int_{U_{(w')^{-1}}^+}\int_{U_{(w')^{-1}}^-}f(x^-x^+u_1^-(x^+)^{-1}w'(w')^{-1} x^+$$
$$\cdot w'mu^+(u^+)^{-1}u_2^-nu^+u^-n^{-1})\varphi(^t{n^{-1}}^t{u^-}^t{u^+}(\overline{u_1^-})^{-1}\overline{u_1^-}\dot{w}_M^{-1}m'u_2^-nu^+u^-n^{-1})$$
$$\cdot \psi^{-1}(x^-x^+)\psi^{-1}(nu^+u^-n^{-1})dx^-dx^+du^-du^+$$
$$=\int_{U_{M,m}\cap n_0U_{M,m}n_0^{-1}\backslash U_M}\int_{N_M}\int_{U_{(w')^{-1}}^+}\int_{U_{(w')^{-1}}^-}f(x^-x^+u_1^-(x^+)^{-1}w'(w')^{-1} x^+$$
$$\cdot w'mu^+(u^+)^{-1}u_2^-nu^+u^-n^{-1})\varphi(^t{n^{-1}}^t{u^-}^t{u^+}\dot{w}_M^{-1}m'u_2^-nu^+u^-n^{-1})$$
$$\cdot \psi^{-1}(x^-x^+)\psi^{-1}(nu^+u^-n^{-1})dx^-dx^+du^-du^+.$$

Now let $x'=w'^{-1}x^+w'$, then $x'\in U_M$, and by compatibility we have $\psi(x')=\psi(x^+)$. Moreover, let $y^-=x^-x^+u_1^-(x^+)^{-1}$, then since $U_{(w')^{-1}}^+$ normalizes $U_{(w')^{-1}}^-$, we see that $x^+u_1^-(x^+)^{-1}\in U_{(w')^{-1}}^-$. As a result, we have $y^-\in U_{(w')^{-1}}^-$. Let $v^-=(u^+)^{-1}u_2^-nu^+u^-n^{-1}\in N_M$. And also let $u'=u^+$. Then since all variables live in unipotent subgroups therefore are all unimodular, we see that $dy^-=dx^-$, $dv^-=du^-$, and $du'=du^+$.

After making the above change of variables, the above integral 
$$=\int_{U_{M,m}\cap n_0U_{M,m}n_0^{-1}\backslash U_M}\int_{N_M}\int_{U_{(w')^{-1}}^-}\int_{U_M}f(y^-w'x'mu'v^-)$$
$$\cdot \varphi(^t{v^-}^t{u'}^t{(u_2^-)^{-1}}^t{n^{-1}}\dot{w}_M^{-1}m'u'v^-)\psi(u_1^-)\psi(u_2^-)\psi^{-1}(y^-)\psi^{-1}(x')\psi^{-1}(v^-)\psi^{-1}(u')$$
$$dx'dy^-dv^-du'.$$

Since here $f\in C^{\infty}_c(\Omega_{w'};w_{\pi})$, the decomposition $\Omega_{w'}=U^-_{(w')^{-1}}\times w' M \times U_{w'}^-$ implies that there exists open compact subsets $U_1\subset U_{(w')^{-1}}^-$, and $U_2 \subset U_{w'}^- $ such that 
$f(y^-w'x'mu'v^-)\neq 0\Longrightarrow y^-\in U_1, v^-\in U_2 $.
Therefore we can take $N$ large enough, such that
$\varphi=\varphi_N$ is invariant under large open compact subgroups of $U_{w'}^-$, as in Lemma 4.2 [7]. Consequently,
$\varphi(^t{v^-}^t{u'}^t{(u_2^-)^{-1}}^t{n^{-1}}\dot{w}_M^{-1}m'u'v^-)=\varphi(^t{u'}^t{(u_2^-)^{-1}}^t{n^{-1}}\dot{w}_M^{-1}m'u').$

So now we have
$$B^G_{\varphi}(g,f)=\int_{U_{M,m}\cap n_0U_{M,m}n_0^{-1}\backslash U_M}\int_{N_M}\int_{U_{(w')^{-1}}^-}\int_{U_M}f(y^-w'x'mu'v^-)$$
$$\cdot\varphi(^t{u'}^t{(u_2^-)^{-1}}^t{n^{-1}}\dot{w}_M^{-1}m'u')\psi(u_1^-)\psi(u_2^-)\psi^{-1}(y^-)\psi^{-1}(x')\psi^{-1}(v^-)\psi^{-1}(u')$$
$$dx'dy^-dv^-du'$$
$$=\int_{U_{M,m}\cap n_0U_{M,m}n_0^{-1}\backslash U_M}\int_{U_{w'}^-}\int_{U_{(w')^{-1}}^-}\int_{U_M}f(y^-w'x'mu'v^-)$$
$$\cdot\varphi(^t{u'}^t{n_0}\dot{w}_M^{-1}m'u')\psi(u_1^-)\psi(u_2^-)\psi^{-1}(y^-)\psi^{-1}(x')\psi^{-1}(v^-)\psi^{-1}(u')$$
$$dx'dy^-dv^-du'.$$

Now by Lemma 6.2, there exists an $h=h_f\in C^{\infty}_c(M;w_{\pi})$ such that
$$h(m)=h_f(m)=\int_{U^-_{w'}}\int_{U^-_{(w')^{-1}}}f(x^-\dot{w}mu^-)\psi^{-1}(x^-u^-)dx^-du^-.$$
This implies that 
$$B^G_{\varphi}(g,f)=\psi(u_1^-)\psi(u_2^-)\int_{U_{M,m}\cap n_0U_{M,m}n_0^{-1}\backslash U_M}\int_{U_M}h_f(x'mu')$$
$$\cdot\varphi(^t{u'}^t{n_0}\dot{w}_M^{-1}m'u')\psi^{-1}(x')\psi^{-1}(u')dx'du'$$
$$=\psi(u_1^-)\psi(u_2^-)B_\varphi^M(u_1^-,u_2^-, m,h_f),$$

where $$B_\varphi^M(u_1^-,u_2^-, m,h_f)=\int_{U_{M,m}\cap n_0U_{M,m}n_0^{-1}\backslash U_M}\int_{U_M}h_f(x'mu')$$
$$\cdot\varphi(^t{u'}^t{n_0}\dot{w}_M^{-1}m'u')\psi^{-1}(x')\psi^{-1}(u')dx'du'.$$

One can check that this integral is well-defined. Suppose $v'\in U_{M,m}\cap n_0 U_{M,m}n_0^{-1}$, then by the remark after the previous lemma, we see that for $u\in U_M$, $v'\in \textrm{Cent}(n_0)$, 
$\leftidx{^t}u'\leftidx{^t}v' \leftidx{^t}{n_0}^{-1}\dot{w}_M^{-1}m' v'u'=\leftidx{^t}u'\leftidx{^t}v' \leftidx{^t}{n_0}^{-1}\leftidx{^t}v'^{-1}\leftidx{^t}v'\dot{w}_M^{-1}m' v'u'=\leftidx{^t}u' \leftidx{^t}{n_0}\dot{w}_M^{-1}m'u'$.
In particular, if $n_0=1$, i.e., $\leftidx{^t}(\overline{u_1^-})=u_2^-$, we have
$$B^G_\varphi(g,f)=\psi(u_1^-)\psi(u_2^-)B^M_\varphi(m, h_f).$$

\end{proof}

\subsection{Small cell Analysis}
The philosophy to prove supercuspidal stability is to analyze the asymptotic behavior of the partial Bessel integrals through looking at the contribution of each Bruhat cell inductively. In this section we will analyze the small cell of both $G$ and its Levi subgroups. 

The following lemmas(lemma 6.7, 6.8, 6.9), which were proved in [7], show that the non-zero contributions are only from the relevant parts of those Bruhat cells that support Bessel functions. We will use them, together with the transfer principal(proposition 6.5) to obtain the asymptotic expansion for partial Bessel integrals.
\begin{lemma}
Let $w\in B(G)$ and $f\in C_c^{\infty}(\Omega_w;\omega_{\pi}).$ Suppose $B^G_{\varphi}(\dot{w}a,f)=0$ for all $a\in A_w$. Then there exists $f_0\in C^{\infty}(\Omega'_{\dot{w}};\omega_{\pi})$, where $\Omega'_{\dot{w}}=\Omega_w-C_r(\dot{w})$, such that for sufficiently large $\varphi$ depending only on $f$, we have  $B^G_{\varphi}(g,f)=B^G_{\varphi}(g,f_0)$ for all $g\in G$.
\end{lemma}
\begin{proof} See Lemma 5.12, [7].
\end{proof}
\begin{lemma}
Let $w\in B(G)$ and $f\in C^{\infty}_c(\Omega_w;\omega_{\pi})$, $\Omega_w^\circ=\Omega_w-C(w)$. Suppose $B^G(\dot{w}a,f)=0$ for all $a\in A_w$. Then there exists $f_0\in C^{\infty}_c(\Omega_w^\circ,\omega_{\pi})$ such that, for all sufficiently large $\varphi$ depending only on $f$, we have $B^G_{\varphi}(g,f)=B^G_{\varphi}(g,f_0)$ for all $g\in \Omega_w$.
\end{lemma}
\begin{proof}
See Lemma 5.13, [7]. 
\end{proof}
\begin{lemma}
Let $w=w_Gw_M\in B(G)$. Let $\Omega_{w,0}$ and $\Omega_{w,1}$ be $U\times U $ and $A$-invariant open sets of $\Omega_w$ such that $\Omega_{w,0}\subset \Omega_{w,1}$ and $\Omega_{w,1}-\Omega_{w,0}$ is a union of Bruhat cells $C(w')$ such that $w'$ does not support a Bessel function, i.e, $w'\notin B(G)$. Then for any $f_1\in C^{\infty}_c(\Omega_{w,1};\omega_{\pi})$, there exists $f_0 \in C^{\infty}_c(\Omega;\omega_{\pi})$ such that, for all sufficiently large $\varphi$ depending only on $f_1$, we have $B^G_{\varphi}(g,f_0)=B^G_{\varphi}(g,f_1)$ for all $g\in G$.
\end{lemma}
\begin{proof}
See Lemma 5.14, [7].
\end{proof}

Now let's work on the inductive process of the asymptotic expansion of partial Bessel integrals. We begin with the analysis of the small cell of $G$. Consider $e$ as a Weyl group element, then 
$M_e=G, A_e=Z_G=Z$, and $U_e^+=U.$ We also have
$\Omega_e=\bigsqcup_{e\leq w}C(w)=G$. Take the representative of $e$ to be $\dot{e}=I.$ Take $f\in \mathcal{M}(\pi)\subset C^{\infty}_c(G;\omega_{\pi})$ with $W^f(e)=1.$
We also fix an auxiliary function $f_0\in C^{\infty}_c(G;\omega_{\pi})$ such that $W^{f_0}(e)=1$. Decompose $G=G^dA_e=G^dZ$, where $G^d$ is the derived group of $G$. Since $G^d\cap Z$ is finite, if we write $g=g_1c$ for $g\in G$ and $g_1\in G^d$, $c\in Z$, then there are only finitely many such decompositions and they differ by elements in the transverse torus $A^e_e$. In the case of $G=GL_n$, $A^e_e$ consists of diagonal matrices whose entries are n-th roots of unity, and notice that there is no such decomposition if $\det(g)$ is not an $n$-th power in $F^\times$.
Now let 
$$f_1(g)=\sum_{g=g_1c}f_0(g_1)B^G(\dot{e}c, f)=\sum_{g=g_1c}f_0(g_1)\omega_{\pi}(c)$$ if $\det(g)$ is an $n$-th power in $F^\times$, and $f_1(g)=0$ otherwise. Then $f_1(g)\in C_c^{\infty}(G;\omega_{\pi})$, since the subgroup of all $g\in G$ such that $\det(g)$ is an $n$-th power in $F^\times$ is open in $G$.
We have 
\begin{lemma}
$B^G_{\varphi}(\dot{e}a,f_1)=B^G_\varphi(\dot{e}a, f)$ for all $a\in A_e=Z.$
\end{lemma}
\begin{proof}
See Lemma 5.15, [7].
\end{proof}
\begin{proposition}
Fix an auxiliary function $f_0\in C^{\infty}_c(G;\omega_{\pi})$ with $W^{f_0}(e)=1$. Then for each $f\in C_c^{\infty}(G;\omega_{\pi})$ with $W^f(e)=1$ and for each $w'\in B(G)$ with $d_B(e,w')=1$, there exists a function $f_{w'}\in C^{\infty}_c(\Omega_{w'};\omega_{\pi})$ such that for any $w\in B(G)$ and any $g=u_1\dot{w}au_2 \in C_r(\dot{w})$ we have 
$$B^G_{\varphi}(g,f)=\sum_{w'\in B(G), d_B(w',e)=1}B^G_{\varphi}(g,f_{w'})+\sum_{a=bc}\omega_\pi(c)B^G_{\varphi}(u_1\dot{w}bu_2,f_0)$$ where $a=bc$ runs over the possible decompositions of $a\in A_w$ with $b\in A^e_w$ and $c\in A_e=Z.$
\end{proposition}
\begin{proof} We construct $f_1$ from $f_0$ as above. By Lemma 6.10, $B^G_{\varphi}(\dot{e}a,f-f_1)=0$ for all $a\in A_e=Z$. We have $C_r(e)=A_eU=ZU\subset C(e)=AU$ and $\Omega_e^{\circ}=\Omega_e-C(e)=G-AU=\bigsqcup_{w\neq e}C(w)$. Then by Lemma $6.8$, there exists an $f_2'\in C^{\infty}_c(\Omega_e^{\circ};\omega_{\pi})$ such that $B^G_{\varphi}(g,f-f_1)=B^G_{\varphi}(g,f'_2)$ for all $g\in G$.
Let 
$\Omega_1=\bigcup_{w\in B(G), w\ne e}\Omega_w=\bigcup_{w'\in B(G), d_B(w',e)=1}\Omega_{w'}=\bigsqcup_{w''\ge w'\in B(G), d_B(w',e)=1}C(w'')$ and $\Omega_0=\Omega_e^{\circ}=G-C(e)=\bigsqcup_{w\neq e}C(w).$ So $\Omega_0 - \Omega_1$ is a union of Bruhat cells $C(w)$ such that $w\notin B(G)$, since $d_B(w',e)=1$ in the definition of $\Omega_1$.

By Lemma 6.9, there exists $f_2\in C^{\infty}_c(\Omega_1,\omega_{\pi})$ such that for sufficiently large $\varphi$ we have
$B^G_{\varphi}(g,f_2)=B^G_{\varphi}(g,f'_2)=B^G_{\varphi}(g,f-f_1)$ for all $g\in G$. Then we use a partition of unity argument, to get 
$f_2=\sum_{w'\in B(G),d_B(w',e)=1}f_{w'}$ with $f_{w'}\in C^{\infty}_c(\Omega_{w'};\omega_{\pi}).$ Thus for any $w\in B(G)$ and any $g\in C_r(\dot{w})$ we have
$$B^G_{\varphi}(g,f)=B^G_{\varphi}(g,f_1)+\sum_{w'\in B(G),d_B(w',e)=1}B^G_{\varphi}(g,f_{w'}).$$

Now we work with $B^G_{\varphi}(g, f_1)$ for $g\in C_r(\dot{w})$. We have 

$$B^G_{\varphi}(g,f_1)=\int_{U_g\backslash U}\int_Uf_1(xgu)\varphi(^tu\dot{w}_G^{-1}g'u)\psi^{-1}(x)\psi^{-1}(u)dxdu$$
$$=\int_{U_g\backslash u_2^{-1}U^+_w u_2}\int_{U_w^-}\int_Uf_1(xgu'^+u_2^{-1}u^-u_2)\varphi(^t{(u_2^{-1}u^-u_2)}^t{u'^+}\dot{w}_G^{-1}g'u'^+u_2^{-1}u^-u_2)$$
$$\cdot \psi^{-1}(x)\psi^{-1}(u'^+u_2^{-1}u^-u_2)dxdu^-du'^+.$$

Since $f_1(g)=\sum_{g=g_1c}f_0(g_1)B^G(\dot{e}c, f)=\sum_{g=g_1c}f_0(g_1)\omega_{\pi}(c),$ we need to decompose 
$xgu'^+u_2^{-1}u^-u_2=g_1c$ with $g_1\in G^d$ and $c\in Z$.
Write $g=u_1\dot{w}au_2$, then 
$g_1=xu_1\dot{w}ac^{-1}u_2u'^+u_2^{-1}u^-u_2\in G^d$.
So $1=\det(g_1)=\det(ac^{-1}).$ This says that $b=ac^{-1}\in A_w^e=SL_n(F)\cap Z_L$, where $L=L_w$ is the Levi given by $w=w_Gw_L\in B(G)$. We decompose $A_w=ZA_{w'}$, then $a=za'$ and 
$a'=(bc)'=b'.$ Therefore we have
$$f_1(xgu'^+u_2^{-1}u^-u_2)=\sum_{a=bc}f_0(xu_1\dot{w}bu_2u'^+u_2^{-1}u^-u_2)\omega_{\pi}(c).$$
So eventually we have
$$B^G_{\varphi}(g,f_1)=\int_{U_g\backslash U}\int_Uf_1(xgu)\varphi(\leftidx{^t}u\dot{w}_G^{-1}g'u)\psi^{-1}(x)\psi^{-1}(u)dxdu$$
$$=\int_{U_g\backslash u_2^{-1}U^+_w u_2}\int_{U_w^-}\int_Uf_1(xgu'^+u_2^{-1}u^-u_2)\varphi(\leftidx{^t}{(u_2^{-1}u^-u_2)}\leftidx{^t}{u'^+}\dot{w}_G^{-1}g'u'^+u_2^{-1}u^-u_2)$$
$$=\omega_{\pi}(c)\sum_{a=bc}\int_{U_g\backslash u_2^{-1}U^+_w u_2}\int_{U_w^-}\int_Uf_0(xu_1\dot{w}bu_2u'^+u_2^{-1}u^-u_2)\varphi(\leftidx{^t}{(u_2^{-1}u^-u_2)}\leftidx{^t}{u'^+}\dot{w}_G^{-1}$$
$$\cdot u_1\dot{w}bu_2u'^+u_2^{-1}u^-u_2)\psi^{-1}(x)\psi^{-1}(u'^+u_2^{-1}u^-u_2)dxdu^-du'^+$$
$$=\sum_{a=bc}\omega_{\pi}(c)\int_{U_g\backslash U}\int_U f(xu_1\dot{w}bu_2u)\varphi(\leftidx{^t}u\dot{w}_G^{-1}u_1\dot{w}bu_2u)\psi^{-1}(x)\psi^{-1}(u)dxdu$$
$$=\sum_{a=bc}\omega_{\pi}(c)B^G_{\varphi}(u_1\dot{w}bu_2,f_0).$$
\end{proof}
A very similar process works for Levi subgroups $M\subset G$. If $w'=w_Gw_M\in B(G)$, then $A^{w'}_{w'}=Z_M\cap M^d$, which is also finite. In the case $G=\textrm{GL}_n$, $M^d\simeq SL_{n_1}\times \cdots\times SL_{n_t}$ for some $t\ge 1$, and $A^{w'}_{w'}=A_w\cap (M_{w'})^d$ consists of $n_i$-th roots of unity in the $i$-th block of $M$.

Let's analyze the small cell of $M$. For $h\in C_c^{\infty}(M;\omega_{\pi})$, and $c\in Z_M=A_{w'}$, define the Bessel integral on $M$ by 
$B^M(c,h)=\int_{U_M}h(xc)\psi^{-1}(x)dx$.
Take $h_0\in C_c^{\infty}(M;\omega_{\pi})$, such that $B^M(e, h_0)=\frac{1}{\kappa_M}$, where $\kappa_M=\vert Z\cap A_{w'}^{w'}\vert <\infty$, and $B^M(b, h_0)=0$ for $b\in A^{w'}_{w'}$ but $b\notin Z\cap A^{w'}_{w'}$. Decompose $M=M^d Z_M$, where $M^d\cap Z_M=A^{w'}_{w'}$ is finite. Define $h_1$ on $M$
by
$h_1(m)=\sum_{m=m'c}h_0(m')B^M(c,h)$ with $m'\in M^d$ and $c\in Z_M=A_{w'}$. Similar to the case for $G$, if $m=\textrm{diag}\{m_1,m_2,\cdots, m_r\}$, $\det(m_i)$ is not an $n_i$-th power on each block, then $h_1(m)=0$. We have

\begin{lemma}
$B^M_{\varphi}(a, h_1)=B^M_{\varphi}(a,h)$
 for all $a\in Z_M=A_{w'}$.
\end{lemma}
\begin{proof}
See Proposition 5.4, [7].
\end{proof}

Now suppose $g\in C_r(\dot{w}_G)$ with $g=u_1\dot{w}_G a u_2$, then for $w'=\dot{w}_G\dot{w}_M^{-1}$ we have 
$$C_r(\dot{w}_G)\subset \Omega_{w'}=U_{w'^{-1}}^-\times w'M\times U_{w'}^-.$$ We further decompose $g$ as $=u_1^-u_1^+\dot{w}_Gau_2^+u_2^-$ with $u_1^-\in U_{w'^{-1}}^-$, $u_1^+\in U_{w'^{-1}}^+$, $u_2^+\in U_{w'}^+$, $u_2^-\in U_{w'}^-$, $u_1=u_1^-u_1^+$, $u_2=u_2^+u_2^-$. Then
$$g=u_1^-w'(w')^{-1}u_1^+w'\dot{w}_Mau_2^+u_2^-=u_1^-w'mu_2^-$$ where $m=(w')^{-1}u_1^+w'\dot{w}_Mau_2^+\in C_r^M(\dot{w}_M)$, the relevant cell of $\dot{w}_M$ in $M$, and $a\in A_{w_G}=A$. Recall that
$$B^M_{\varphi}(u_1^-,u_2^-,m,h_1)=\int_{U_{M,m}\cap n_0U_{M,m}n_0^{-1}\backslash U_M\times U_M}h_1(xmu)\varphi(^tu\leftidx{^t}n_0\dot{w}_M^{-1}m'u)\psi^{-1}(xu)dxdu$$
where $m'=(w')^{-1}u_1^+w'\dot{w}_Ma'u_2^+$.
Here $a=za'$ is the decomposition of $a\in A=ZA'$. It follows that $B^M_{\varphi}(u_1^-,u_2^-,m,h_1)=\omega_{\pi}(z)B^M_{\varphi}(u_1^-,u_2^-,m', h_1)$.

Since $h_1(m)=\sum_{m=m_1c}h_0(m_1)B^M(m',h)$
with $m_1\in M^d$ and $c\in Z_M$, to compute the above integral, we need to decompose $xm'u=m_1c$. This gives
$xw'^{-1}u_1^+w'\dot{w}_M^{-1}a'u_2^+uc^{-1}=m_1\in M^d$.
Since $x,w',\dot{w}_M, u,u_1^+,u_2^+\in M^d$, it suffices to decompose $a'=bc$ for $b\in A\cap M^d$ and $c\in Z_M$. Now we can write
$$h_1(m')=\sum_{a'=bc}h_0(xw'^{-1}u^+_1w'\dot{w}_Mbu_2^+u)B^M(c,h).$$
Decompose $b=z_bb'$ and $c=z_cc'$, with $z_b,z_c\in Z$, $b'\in A'$ and $c'\in Z_M'$.
Then $a'=bc=z_bz_cb'c'\Longrightarrow a'=b'c'$, and $z_bz_c=1.$
As $h,h_0\in C^{\infty}_c(M;\omega_{\pi})$, we have
$$h_0(xw'^{-1}u^+_1w'\dot{w}_Mbu^+_2u)B^M(c,h)=\omega_{\pi}(z_bz_c)h_0(xw'^{-1}u_1^+w'\dot{w}_Mb'u_2^+u)B^M(c',h)$$
$$=h_0(xw'^{-1}u_1^+w'\dot{w}_Mb'u_2^+u)B^M(c',h).$$
Thus
$$B^M_\varphi(u_1^-,u_2^-,m',h_1)=\int_{U_{M,m}\cap n_0U_{M,m}n_0^{-1}\backslash U_M\times U_M}\sum_{a=bc}h_0(xw'^{-1}u_1^+w'\dot{w}_Mb'u_2^+u)B^M(c',h)$$
$$\cdot \varphi(^tu \leftidx{^t}n_0\dot{w}_M^{-1}w'^{-1}u_1^+w'\dot{w}_Mb'u^+_2uc')\psi^{-1}(xu)dxdu$$
$$=\sum_{a'=b'c'}B^M(c',h)\int_{U_{M,m}\cap n_0 U_{M,m}n_0^{-1}\backslash U_M\times U_m}h_0(xw'^{-1}u_1^+w'b'u^+_2u)$$
$$\cdot \varphi(^tu \leftidx{^t}n_0\dot{w}_M^{-1}w'^{-1}u_1^+w'\dot{w}_Mb'u^+_2uc')\psi^{-1}(xu)dxdu.$$

Now since $a'=b'c'$, $c'\in Z_M'\subset Z_M$, let
$m_{b'}=w'^{-1}u_1^+w'\dot{w}_Mb'u^+_2$, then
$$m'=w'^{-1}u_1^+w'\dot{w}_Ma'u_2^+=w'^{-1}u_1^+w'\dot{w}_Mb'c'u_2^+=m_{b'}c'$$
Meanwhile we have 
$U_{M,m'}=\{u\in U_M: {^tu}\dot{w}_M^{-1}m'u=\dot{w}_M^{-1}m'\}=\{u\in U_M: \dot{w}_M{^tu}\dot{w}_M^{-1}m'u=m'\}=\{u\in U_M: \dot{w}_M{^tu}\dot{w}_M^{-1}m_{b'}c'u=m_{b'}c'\}=\{u\in U_M: \dot{w}_M{^tu}\dot{w}_M^{-1}m_{b'}uc'=m_{b'}c'\}=\{u\in U_M:\dot{w}_M{^tu}\dot{w}_M^{-1}m_{b'}u=m_{b'}\}=U_{M,m_{b'}}$

So we obtain
$$B^M_{\varphi}(u_1^-,u_2^-,m',h_1)=\sum_{a'=b'c'}B^M(c',h)\int_{U_{M,m_{b'}}\cap n_0U_{M,m_{b'}}n_0^{-1}\backslash U_M\times U_M}h_0(xm_{b'}u)$$$$\cdot\varphi({^tu}\leftidx{^t}n_0\dot{w}_M^{-1}m_{b'}uc')\psi^{-1}(xu)dxdu$$
$$=\sum_{a'=b'c'}B^M(c',h)B^M_{\varphi^{c'}}(u_1^-,u_2^-,m_{b'},h_0)$$
where $\varphi^{c'}(m)=\varphi(mc')$ for $c'\in Z_M'$.

In particular, when $n_0=1$, we have
$$B^M_\varphi(m,h_1)=\sum_{a=bc}B^M(c, h)B^M_{\varphi^{c}}(m_{b'}, h_0).$$
\subsection{Uniform smoothness} The key to prove supercuspidal stability is that the asymptotic expansions of partial Bessel integrals have two parts, one part depends only on the central character of $\pi$, the other is a uniform smooth function on certain torus. Therefore under highly ramified twist, the uniform smooth part becomes zero. We study the uniform smoothness in this section.
\begin{definition}
A smooth function $B$ on a torus $T\subset A$ is \textbf{uniformly smooth} if there exists a fixed open compact subgroup $T_0\subset T$ such that $B(tt_0)=B(t)$ for $t_0\in T_0$ and all $t\in T$.
\end{definition}
\begin{proposition}
For $g(a)=u_1^-(a)w'm(a)u_2^-(a)\in C_r^G(\dot{w}_G)$ with $m=m(a)=\tilde{u}_1(a)\dot{w}_Ma \tilde{u}_2(a)\in C^M_r(\dot{w}_M)$, $a\in A^{w'}_{\dot{w}_G}A_{w'}\subset A_{\dot{w}_G}=A$, $u_1^-(a)$, $u_2^-(a)$, $\tilde{u}_1(a)$ and $\tilde{u}_2(a)$ are rational functions(as morphisms of algebraic varieties) of $a$. Let $a=bc$ be a fixed decomposition with $b\in A^{w'}_{w_G}$ and $c\in A_{w'}$. Then all decompositions are of the form $a=(b\zeta^{-1})(\zeta c)$ with $\zeta\in A^{w'}_{w'}=A^{w'}_{w_G}\cap A_{w'}$, a finite set with appropriate roots of unity on the diagonal. Moreover, if $c=c'z$ with $c'\in A_{w'}'=Z_M'$ and $z\in Z$, then for each fixed $b, z$,
$$B^M_{\varphi}(u_1^-(a), u_2^-(a),m(a),h_1)=\omega_{\pi}(z)B^M_{\varphi}(u_1^-(bc'z), u_2^-(bc'z),\tilde{u}_1(bc'z)\dot{w}_Mbc'\tilde{u}_2(bc'z), h_1)$$ is uniformly smooth as a function of $c'\in Z_M'$.
\end{proposition}
\begin{proof}
First fix one decomposition $a=bc$. To simplify the notation, we denote $u_i^-=u_i^-(a)$ and $\tilde{u}_i=\tilde{u}_i(a)$. Then we have 
$$B^M_{\varphi}(u_1^-,u_2^-,m,h_1)=\sum_{a=bc}B^M(c,h)B^M_{\varphi^c}(u_1^-,u_2^-,m_{b'}, h_0)$$$$=\sum_{\zeta}B^M(\zeta c, h)B^M_{\varphi^{\zeta c}}(u_1^-,u_2^-,\tilde{u}_1\dot{w}_Mb\zeta^{-1}\tilde{u}_2,h_0)$$
Since $\vert\zeta\vert=1$, so we have $\varphi^{\zeta c}=\varphi^c$. This implies that 
$B^M_{\varphi}(u_1^-,u_2^-,m,h_1)=\sum_{\zeta}B^M(\zeta c, h)B^M_{\varphi^c}(u_1^-,u_2^-,\tilde{u}_1\dot{w}_Mb\zeta^{-1}\tilde{u}_2,h_0).$

Now $$B^M(\zeta c, h)=\int_{U_M}h(x\zeta c)\psi^{-1}(x)dx=\omega_{\pi}(\zeta_1 z)\int_U h(x\zeta'c')\psi^{-1}(x)dx$$
where $\zeta=diag(\zeta_1 I_{n_1},\cdots, \zeta_t I_{n_t})$ and $\zeta'=diag(I_{n_1},\zeta_1^{-1}\zeta_2I_{n_2}\cdots, \zeta_1^{-1}\zeta_{t}I_{n_t}).$ Since $h\in C^{\infty}_c(M;\omega_{\pi})$, $x\zeta c\in A_MU_M=B_M$ and $C^M(e_M)=B_M$ is closed in $M$, there exists compact subsets $U_1\subset U$, $K''\subset A'$ s.t.
$h(x\zeta'c')\neq 0 \Longrightarrow x\in U_1, \zeta'c'\in K'.$ Moreover, since $Z_M'\subset A'$ is closed and $\zeta'c'\in Z_M'$, there exists a further compact subset $K''\subset Z_M'$ s.t.
$h(x\zeta'c')\neq 0 \Longrightarrow x\in U_1, \zeta'c'\in K''.$
Write $a=bc=bc'z$, we see that 
$$B^M_{\varphi}(u_1^-,u_2^-,m,h_1)=\omega_{\pi}(z)\sum_{\zeta}B^M(\zeta c',h)B_{\varphi^c}^M(u_1^-,u_2^-,\tilde{u}_1\dot{w}_Mb\zeta^{-1}\tilde{u}_2, h_0)$$ is zero unless $c'\in \bigcup_{\zeta'}(\zeta')^{-1}K''$, which is compact since it is a finite union of compact subsets.

So $$B^M_{\varphi}(u_1^-, u_2^-,m,h_1)=B^M_{\varphi}(u_1^-,u_2^-, \tilde{u}_1\dot{w}_Ma\tilde{u}_2, h_1)=B^M_{\varphi}(u_1^-,u_2^-,\tilde{u}_1\dot{w}_Mbc'\tilde{u}_2,h_1)$$
$$=\omega_{\pi}(z)B^M_{\varphi}(u_1^-,u_2^-,\tilde{u}_1\dot{w}_Mbc'\tilde{u}_2, h_1)$$ has compact support on $c'\in Z_M'$, depending only on $h$ through the choice of $K''$ and $A_{M^d}\cap Z_M$. Thus independent of $a$ and $b$.

Since $h$ is smooth and its support in $c'$ is  compact, for each fixed $b, z$, there exists uniform compact subset $\Omega_{b,z}\subset Z_M'$ s.t.
$h(x\zeta c'c_1)=h(x\zeta c')$, $u_i^-(bzc'c_1)=u_i^-(bzc')$, $\tilde{u}_i(bzc'c_1)=\tilde{u}_i(bzc')$ $(i=1,2)$ for all $c_1\in \Omega_{b,z}$, $x\in U_1$, and $c'\in Z_M'.$ Shrinking $\Omega_{b,z}$ if necessary, we may assume that $\Omega_{b,z}\subset Z_M'(\mathcal{O}_F)$, so $\varphi^{cc'}=\varphi^c$ for all $c_1\in \Omega_{b,z}$.  
So we have proved that 
$$B^M_\varphi(u_1^-(ac_1), u_2^-(ac_1), m(ac_1), h_1)$$$$=B^M_{\varphi}(u_1^-(bzc'c_1), u_2^-(bzc'c_1),\tilde{u}_1(bzc'c_1)\dot{w}_Mbzc'c_1\tilde{u}_2(bzc'c_1),h_1)$$$$=B^M_{\varphi}(u_1^-(bc'z), u_2^-(bc'z),\tilde{u}_1(bc'z)\dot{w}_Mbc'z\tilde{u}_2(bc'z),h_1)$$$$=\omega_{\pi}(z)B^M_{\varphi}(u_1^-(bc'z), u_2^-(bc'z),\tilde{u}_1(bc'z)\dot{w}_Mbc'\tilde{u}_2(bc'z),h_1)$$
$$=B^M_\varphi(u_1^-(a), u_2^-(a), m(a), h_1)$$
for all $c_1\in \Omega_{b,z}$, $a=bc.$

Finally note that since $A_{w_G}^{w'}A_{w'}\subset A_{w_G}=A$ is open of finite index, one can extend $B^M_{\varphi}(u_1^-(a), u_2^-(a), m(a), h_1)$ on all of $A$.
\end{proof} 

\subsection{Asymptotic expansions} We are ready to establish a more general version of the asymptotic expansion formula for partial Bessel integrals as in [7]. The formula that will be established works for all elements in the relevant Bruhat cells.

The following proposition is the key to prove the main results in this section. 
\begin{proposition}

Let $w'=\dot{w}_G\dot{w}_M^{-1}\in B(G)$, and $f_{w'}\in C_c^{\infty}(\Omega_{w'};\omega_{\pi})$. There exists $f_{1,w'}\in C^{\infty}_c(\Omega_{w'};\omega_{\pi})$, such that

(1), $\exists$ a family of functions $ \{f_{w''}\}_{w''\in B(G)}$ with $d_B(w'',w')=1$, $w''>w'$, such that $f_{w''}\in C_c^{\infty}(\Omega_{w''};\omega_{\pi})$, and for $\forall w\in B(G)$ and $g\in C_r^G(w)$, we have
$$B^G_{\varphi}(g,f_{w'})=B^G_{\varphi}(g,f_{1,w'})+\sum_{w''\in B(G),w''>w', d_B(w'',w')=1} B^G_{\varphi}(g,f_{w''});$$

(2), Let $g=u_1(a)\dot{w}_G a u_2(a)\in C^G_r(\dot{w}_G)$, where $u_i(a)$'s are rational functions(as algebraic varieties) of $a$. Write $u_1(a)=u_1^-(a)u_1^+(a)\in U_{(w')^{-1}}^-U_{(w')^{-1}}^+=U$ and $u_2(a)=u_2^+(a)u_2^-(a)\in U_{w'}^+U_{w'}^-=U$, then $u_i^{\pm}(a)$'s are all rational functions of $a\in A$. Then $g=u^-_1(a)w' m(a) u^-_2(a)$ and 
$m(a)=w'^{-1}u_1^+(a)w'\dot{w}_Mau_2^+(a)=\tilde{u}_1(a)\dot{w}_Ma\tilde{u}_2(a)$ where $\tilde{u}_1=w'^{-1}u_1^+w'$, $\tilde{u}_2=u^+_2.$ And we have 
$$B^G_{\varphi}(g,f_{1,w'})
=\omega_{\pi}(z)B^G_{\varphi}(u_1(bc'z)\dot{w}_G bc' u_2(bc'z),f_{1,w'})$$
is uniformly smooth as a function of $c'\in A'_{w'}=Z'_M$ for each fixed $b$ and $z$.
\end{proposition}
\begin{proof}
Take $h=h_{f_{w'}}\in C^{\infty}_c(M,\omega_{\pi})$ which maps to $f_{w'}$ under the surjective map $C^{\infty}_c(M;\omega_{\pi})\twoheadrightarrow C^{\infty}_c(\Omega_{w'}, \omega_{\pi})$ in Lemma 6.2. Construct $h_1$ based on $h$ as Lemma 6.12 such that $B^M_{\varphi}(a,h_1)=B^M_{\varphi}(a,h)$ for all $a\in Z_M=A_{w'}$. We have $h_1\in C^{\infty}_c(M;\omega_{\pi})$. Let $f_1$ be the image of $h_1$ under the map $C^{\infty}_c(M;\omega_{\pi})\twoheadrightarrow C^{\infty}_c(\Omega_{w'}, \omega_{\pi})$.
Then by the transfer principal of partial Bessel integrals (Proposition 6.5), we have for Levi subgroups $L$, $M$ of $G$ with $A\subset L \subset M\subset G$, and $g=u_1\dot{w}^G_Lau_2=u_1^-w'mau_2^-\in C_r(w^G_L)$,
$$B^G_{\varphi}(g, f_1)=\psi(u_1^-)\psi(u_2^-)B^M_{\varphi}(u_1^-,u_2^-,m,h_1).$$

Apply this with the case when $L=M$, and $g=w'a$, $a\in A_{w^G_M}=A_{w'}=Z_M$, then $u_1^-=u_2^-=1$. So we have
$$B^G_{\varphi}(w'a,f_1)=B^M_{\varphi}(a, h_1)=B^M_{\varphi}(a,h)=B^G(w'a,f_{w'})$$
by Lemma 6.12. So $B^G_{\varphi}(w'a,f_{w'}-f_1)=0$ for all $a\in A_{w'}=Z_M$ and $f_{w'}-f_1\in C^{\infty}_c(\Omega_{w'};\omega_{\pi}).$ Therefore by Lemma 6.7, Lemma 6.8, and Lemma 6.9, in addition with a partition of unity argument, we can find a family of functions $\{f_{w'}: w''\in B(G), w''>w', d_B(w'',w')=1, f_{w''}\in C^{\infty}_c(\Omega_{w''};\omega_{\pi})\}$ such that for any $w\in B(G)$ and any $g\in C_r(\dot{w})$, we have
$$B^G_{\varphi}(g, f_{w'})=B^G_{\varphi}(g,f_1)+\sum_{w''\in B(G), w''>w', d_B(w'',w')=1}B^G_{\varphi}(g,f_{w''}).$$ Moreover for each $f_{w''}$ we have $w''=w^G_{M''}$, this will be used for induction later.

On the other hand if we apply the transfer principal (Proposition 6.5) for partial Bessel integrals to the case $L=A$,  then for $g=u_1\dot{w}_Gau_2=u_1^-w'mu_2^-\in C_r(\dot{w}_G)=C(\dot{w}_G)$, where $m=w'^{-1}u_1^+w'\dot{w}_Mau_2^+\in C^M_r(\dot{w}_M)=C^M(\dot{w}_M)$, we obtain that 
$$B^G_{\varphi}(g, f_1)=B^G_{\varphi}(u_1\dot{w}_Ga u_2, f_1)=\psi(u_1^-)\psi(u_2^-)B^M_{\varphi}(u_1^-,u_2^-,m,h_1)$$

If we decompose $a\in A^{w'}_{w_G}A_{w'}$ as $a=bc$, and assume that $u_1=u_1(a)=u_1^-(a)u_1^+(a)$, $u_2=u_2(a)=u_2^+(a)u_2^-(a)$ are rational maps in $a$, then $g=g(a)=u_1(a)\dot{w}_Gau_2(a)$ is rational in $a$ as well. Then by proposition 6.14 we have
$$B^G_{\varphi}(g,f_1)=B^G_{\varphi}(g(a), f_1)=B^G_{\varphi}(u_1(a)\dot{w}_Ga u_2(a), f_1)$$
$$=\psi(u_1^-(a))\psi(u_2^-(a))B^M_{\varphi}(u_1^-(a), u_2^-(a), m(a),h_1)$$$$=\omega_{\pi}(z)\psi(u_1^-(bc'z)u_2^-(bc'z))B^M_{\varphi}(u_1^-(bc'z), u_2^-(bc'z), w'^{-1}u_1^+(bc'z)w'\dot{w}_Mbc'u_2^+(bc'z), h_1)$$
is compactly supported in $c'\in A'_{w'}=Z'_M$, and therefore $B^G_{\varphi}(g(bc'z),f_1)$ is uniformly smooth as a function of $c'\in Z_M'$ for each fixed $b,z$. 
\end{proof}
Next we are going to perform an induction on the Bessel distance $d_B(w,e)$, to obtain the following main proposition for our final proof of supercuspidal stability:
\begin{proposition}
Fix an auxiliary function $f_0\in C^{\infty}_c(G;\omega_{\pi})$ with $W^{f_0}(e)=1$.
Let $f\in M(\pi)$ with $W^f(e)=1$, and $m\in \mathbb{Z}$ with $1\leq m\leq d_B(w_G, e)+1$. Then

(1) there exists a function $f_{1,e}\in C^{\infty}_c(G;\omega_{\pi})$;

(2) for each $w'\in B(G)$ with $1\leq d_B(w',e)$ there exists $f_{1,w'}\in C^{\infty}_c(\Omega_{w'};\omega_{\pi})$, and for each $w''\in B(G)$ with $d_B(w'',e)=m$ there exists a function $f_{w''}\in C^{\infty}_c(\Omega_{w''};\omega_{\pi})$ such that for sufficiently large $\varphi$ we have

(a) for any $w\in B(G)$ we have $$B^G_{\varphi}(g, f)=B^G_{\varphi}(g,f_{1,e})+\sum_{1\leq d_B(w',e)<m}B^G_{\varphi}(g, f_{1,w'})+\sum_{d_B(w'',e)=m}B^G_{\varphi}(g, f_{w''})$$
for $\forall g\in C_r(\dot{w})$;

(b) for each $w\in B(G)$,$B^G_{\varphi}(g, f_{1,e})$ depends only on the auxiliary function $f_0$ and $w_{\pi}$ for all $g\in C_r(\dot{w})$;

(c) for each $w'\in B(G)$ with $1\leq d_B(w',e)<m$, and $g=g(a)=u_1(a)w^G_M a u_2(a)\in C_r(\dot{w})$, parameterized by $a\in A$ and such that $u_i(a) $'s are both rational functions of $a\in A$, we have that 
$$B^G_{\varphi}(g(a), f_{1,w'})=w_{\pi}(z)B^G_{\varphi}(u_1(bc'z)\dot{w}_Gbc' u_2(bc'z), f_{1,w'})$$ is uniformly smooth as a function of $c'\in A'_{w'}=Z'_M$ for each fixed $b,z$, where $B^G_{\varphi}(g(a),f_{1,w'})$ defined apriori on $a=bc=bc'z \in A^{w'}_{\dot{w}_G}A_{w'}\subset A_{\dot{w}_G}=A$ and finally extended on all $a\in A$.
\end{proposition}
\begin{proof} 
First we fix an auxiliary function $f_0\in C^{\infty}_c(G;\omega_{\pi})$ with $W^{f_0}(e)=1$. Take $f\in M(\pi)\subset C^{\infty}_c(G,\omega_{\pi})$ normalized such that $W^f(e)=1$. Then by Proposition 6.11, we have the following result:

There exists $f_{1,e}\in C^{\infty}_c(G;\omega_{\pi})$ and, for each $w'\in B(G)$ with $d_B(w',e)=1$, there exists a function $f_{w'}\in C^{\infty}_c(\Omega_{w'};\omega_{\pi})$ such that for sufficiently large $\varphi$,

(i) For any $w\in B(G)$, we have
$$B^G_{\varphi}(g, f)=B^G_{\varphi}(g,f_{1,e})+\sum_{w'\in B(G), d_B(w',e)=1}B^G_{\varphi}(g, f_{w'})$$ for all $g\in C_r(\dot{w})$, the relevant cell attached to $w$;

(ii) For each $w\in B(G)$, the partial Bessel integral $B^G_{\varphi}(g, f_{1,e})$ in (i) depends only on the auxiliary function $f_0$ and the central character $\omega_{\pi}$ for all $g\in C_r(\dot{w})$. (This can be seen directly from the expansion formula for $B^G_{\varphi}(g,f_{1,e})$ as in the proof of Proposition 6.11.)

By proposition 6.15, we also have that for each $f_{w'}\in C^{\infty}_c(\Omega_{w'};\omega_{\pi})$, there exists $f_{1,w'}\in C^{\infty}_c(\Omega_{w'};\omega_{\pi})$ such that for sufficiently large $\varphi$,

(i) There exists a family of functions $\{f_{w',w''}\}\in C^{\infty}_c(\Omega_{w''}; \omega_{\pi})$, parameterized by $w''\in B(G)$ with $w''>w'$ and $d_B(w'',w')=1$ such that for any $w\in B(G)$ and any $g\in C_r(\dot{w})$, we have
$$B^G_{\varphi}(g,f_{w'})=B^G_{\varphi}(g,f_{1,w'})+\sum_{w''\in B(G),w''>w', d_B(w'',w')=1} B^G_{\varphi}(g,f_{w',w''});$$

(ii) Let $g=u_1(a)\dot{w}_G a u_2(a)\in C^G_r(\dot{w}_G)=C^G(\dot{w}_G)$, where $u_i(a)$'s are rational functions of $a\in A$. Write $u_1(a)=u_1^-(a)u_1^+(a)\in U_{w'^{-1}}^-U_{w'^{-1}}^+=U$ and $u_2(a)=u_2^+(a)u_2^-(a)\in U_{w'}^+U_{w'}^-=U$, then $u_i^{\pm}(a)$'s are all rational functions of $a\in A$, then $g=u^-_1(a)w' m(a) u^-_2(a)$ and 
$m(a)=w'^{-1}u_1^+(a)w'\dot{w}_Mau_2^+(a)=\tilde{u}_1(a)\dot{w}_Ma\tilde{u}_2(a)$ where $\tilde{u}_1=w'^{-1}u_1^+w'$, $\tilde{u}_2=u^+_2.$ And we have 
$$B^G_{\varphi}(g,f_{1,w'})=w_{\pi}(z)B^G_{\varphi}(u_1(bc'z)\dot{w}_G bc' u_2(bc'z),f_{1,w'})$$
is uniformly smooth as a function of $c'\in A'_{w'}=Z'_M$ for each fixed $b,z$.

Combine the above two results we obtain that
for any $w\in B(G)$,
$$B^G_{\varphi}(g,f_{w'})=B^G_{\varphi}(g,f_{1,w'})+\sum_{ d_B(w',e)=1} B^G_{\varphi}(g,f_{1,w'})$$$$+\sum_{d_B(w'',w')=d_B(w',e)=1}B^G_{\varphi}(g, f_{w',w''})$$
$$=B^G_{\varphi}(g,f_{w'})=B^G_{\varphi}(g,f_{1,w'})+\sum_{ d_B(w',e)=1} B^G_{\varphi}(g,f_{1,w'})+\sum_{d_B(w'',e)=2}B^G_{\varphi}(g, f_{w',w''})$$ for any $g\in C_r(\dot{w})$. 

Let $f_{w''}=\sum_{d_B(w'',w')=1}f_{w',w''},$ then we see that $f_{w''}\in C^{\infty}_c(\Omega_{w''};\omega_{\pi})$. Hence for any $w''\in B(G)$ with $d_B(w'',e)=2$, there exist $f_{w''}\in C^{\infty}_c(\Omega_{w'};\omega_{\pi})$ such that for sufficiently large $\varphi$

(i) for any $w\in B(G)$ and $g\in C_r(w)$ we have
$$B^G_{\varphi}(g,f_{w'})=B^G_{\varphi}(g,f_{1,w'})+\sum_{ d_B(w',e)=1} B^G_{\varphi}(g,f_{1,w'})+\sum_{d_B(w'',e)=2}B^G_{\varphi}(g, f_{w''});$$

(ii) for each $w\in B(G)$, $B^G_{\varphi}(g, f_{1,e})$ depends only on the auxiliary function $f_0$ and the central character $\omega_{\pi}$ for all $g\in C_r(\dot{w})$;

(iii) for $g=u_1(a)\dot{w}_G a u_2(a)\in C^G_r(\dot{w}_G)=C^G(\dot{w}_G)$, parameterized by $a$, where $u_i(a)$'s are rational functions of $a$, we have 
$$B^G_{\varphi}(g,f_{1,w'})=w_{\pi}(z)B^G_{\varphi}(u_1(bc'z)\dot{w}_G bc' u_2(bc'z),f_{1,w'})$$
is uniformly smooth as a function of $c'\in A'_{w'}=Z'_M$ for each fixed $b,z$.

We proceed by induction on $m=d_B(w,e)$ with $w\in B(G)$, and use Proposition 6.15 on each step, we obtain the statements in the Proposition.
\end{proof}

Now if we apply Proposition 6.16 to the case when $m=d_B(w_G, e)+1$, we obtain a final result that we need for the proof of supercuspidal stability in our case:
\begin{proposition}
Fix an auxiliary function $f_0\in C^{\infty}_c(G;\omega_{\pi})$ with $W^{f_0}(e)=1$.
Let $f\in M(\pi)$ with $W^f(e)=1$, Then

(1) there exists a function $f_{1,e}\in C^{\infty}_c(G;\omega_{\pi})$;

(2) for each $w'\in B(G)$ with $1\leq d_B(w',e)$ there exists $f_{1,w'}\in C^{\infty}_c(\Omega_{w'};\omega_{\pi})$ such that for sufficiently large $\varphi$ we have

(a) $$B^G_{\varphi}(g, f)=B^G_{\varphi}(g,f_{1,e})+\sum_{1\leq d_B(w',e)}B^G_{\varphi}(g, f_{1,w'})$$
for $g\in C_r(\dot{w}_G)=C(\dot{w}_G)$;

(b) $B^G_{\varphi}(g, f_{1,e})$ depends only on the auxiliary function $f_0$ and $w_{\pi}$ for all $g\in C(\dot{w}_G)$;

(c) for each $w'\in B(G)$ with $1\leq d_B(w',e)$, and $g=g(a)=u_1(a)\dot{w}_G a u_2(a)\in C(\dot{w}_G)$, parameterized by $a\in A$ and such that $u_i(a) $'s are both rational functions of $a\in A$, we have that 
$$B^G_{\varphi}(g(a), f_{1,w'})=w_{\pi}(z)B^G_{\varphi}(u_1(bc'z)\dot{w}_Gbc' u_2(bc'z), f_{1,w'})$$ is uniformly smooth as a function of $c'\in A'_{w'}=Z'_M$ for each fixed $b,z$.
 \end{proposition}
\section{SUPERCUSPIDAL STABILITY}

Now we have all the ingredients for the final proof of supercuspidal stability in our case.
First recall that we have reduced Proposition 3.4 to the proof of the stability of local coefficient, since the adjoint action $r: \leftidx{^L}M_H\longrightarrow \textrm{GL}(\leftidx{^L}{\mathfrak{n}_H})$ is irreducible. And from Langlands-Shahidi method, $C_{\psi}(s,\pi)=\gamma(s, \pi, \textrm{Sym}^2\otimes \eta,\psi)$. We wrote the local coefficients as the Mellin transform of the partial Bessel functions $j_{\pi,\eta,\dot{w}_\theta,\kappa}(g)$, where $g=-\frac{1}{2}\dot{w}_G\leftidx{^t}Y^{-1}$. By an appropriate choice of orbit space representatives of the space $U_{M_H}\backslash N_H$, we can pick $Y=Y(a_1,\cdots, a_n)$. Then by induction on $n$ we can show that such $g$ lies in the big cell. Let $g=u_1 \dot{w}_G a u_2$ be its Bruhat decomposition. Since $g\mapsto u_1$, $g\mapsto a$, $g\mapsto u_2$ are all morphisms of algebraic varieties, we see that here the entries of $a$, $u_1=u_1(a)$, and $u_2=u_2(a)$ are all rational functions of $(a_1,a_2,\cdots,a_n)\in (F^\times)^n.$ We have $g=u_1 \dot{w}_G a u_2=u_1(a)\dot{w}_Ga u_2(a)\in C_r(\dot{w}_G)=C(\dot{w}_G)\subset \Omega_{w'}$, write $$g=u_1 \dot{w}_G a u_2=u_1^-u_1^+\dot{w}_G a u_2^+u_2^-=u_1^-w'mu_2^-,$$ where $m=(w')^{-1}u_1^+w' \dot{w}_M a u_2^+\in C_r^M(\dot{w}_M)$ with 
$u_1^-\in U^-_{(w')^{-1}}$, $u_1^+\in U_{(w')^{-1}}^+$, $u_2^+\in U_{w'}^+$, $u_2^-\in U_{w'}^-$, $u_1=u_1^-u_1^+$, $u_2=u_2^+u_2^-.$
Since $u_1(a)$ and $u_2(a)$ are both rational functions of $a$, the projection maps 
$u_i(a)\mapsto u_i^{\pm}(a)$ are rational maps, so $u_i^{\pm}(a)$'s are all rational functions of $a$. 
So we can apply Proposition 6.14 to our case with $\tilde{u}_1(a)=(w')^{-1}u_1^+(a)w'$, $\tilde{u}_2(a)=u_2^+(a)$. Now we see that the conditions for Proposition 6.17 are all satisfied for our $g$. 

By Proposition 5.9,
$$C_{\psi}(s,\sigma_\eta)^{-1}=\gamma(ns, \omega_{\pi}^2,\psi)^{-1}\int_{F^{\times}\backslash R}j_{\pi,\eta,\dot{w}_\theta,\kappa}(-\frac{1}{2}\dot{w}_G\leftidx{^t}Y^{-1})$$$$\cdot\omega_{\pi}(4\det(Y)^2\prod_{i=1}^n a_i^{-2}) \vert\frac{1}{2}\vert^{\frac{n(n-s)}{2}}\vert \det(Y)\vert^{\frac{2ns-s-n}{2}} \prod_{i=1}^n\vert a_i \vert ^{i-1-ns}da_i$$
In the Bruhat decomposition $g=-\frac{1}{2}\dot{w}_G\leftidx{^t}Y^{-1}=u_1(a)\dot{w}_G au_2(a)$ if we write $a=\textrm{diag}\{ d_1,\cdots, d_n\}$,
then a direct calculation shows that 
$$d_1=\frac{\prod_{j\ \ even}a_j^2}{\prod_{k\ \ odd}a_k^2}, d_2=\frac{\prod_{k\neq 1,\ \ odd}a_k^2}{4\prod_{j\ \ even}a_j^2},d_3=\frac{\prod_{j\neq 2,\ \ even}a_j^2}{\prod_{k\neq 1, \ \ odd}a_k^2},$$ $$d_4=\frac{\prod_{k\neq 1, 3,\ \ odd}a_k^2}{4\prod_{j\neq 2,\ \ even}a_j^2},\cdots,d_n$$
and $d_n=\frac{1}{4a_n^2}$ if $n$ is even, $d_n=\frac{1}{a_n^2}$ if $n$ is odd.
And no matter $n$ is even or odd we have $d_i\cdot d_{i+1}=\frac{1}{4a_i^2}$ for all $1\leq i \leq n-1$. Recall that the action of $F^{\times}$ on $R\simeq (F^{\times})^n$ is given by 
$t\cdot (a_1,\cdots,a_n)=(t^2 a_1, t^2 a_2,\cdots t^2 a_{n-1}, t a_n)$. From the above observation, it is clear that this action is equivalent to the action of $F^{\times}$ on $A=\{\textrm{diag}\{d_1,\cdots,d_n):d_i\in F^{\times}\}$ by
$t\cdot \textrm{diag}(d_1, d_2,\cdots,d_n)=\textrm{diag}(\frac{d_1}{t^2},\frac{d_2}{t^2},\cdots,\frac{d_n}{t^2}).$
Thus the action of $F^{\times }$ on $R$ translates into the action of $Z$ on $A$. Meanwhile the change of variable $(a_1,\cdots,a_n)\mapsto (d_1,\cdots,d_n)$ translates the measure given by the $a_i$'s into a unique measure given by the $d_i$'s, with the determinant of the Jacobian matrix a rational function of the $d_i$'s. Recall that by the computation at the end of section 5.3, $\det(g)=\det(Y)^{-1}=\frac{(-\frac{1}{2})^n}{\prod_{k\ \ odd}a_k^2}$, if $n$ is even;  $\det(g)=\det(Y^{-1})=\frac{(-\frac{1}{2})^{n-1}}{\prod_{k\ \  odd}a_k^2}$ if $n$ is odd. In both cases 
$\det(Y^{-1})\in (F^{\times})^2$. 
On the other hand, $\det(Y)^2=\frac{1}{(d_1\cdots d_n)^2}=\frac{1}{d_1(d_1d_2)(d_2d_3)\cdots(d_{n-1}d_n)d_n}
$. The last expression is equal to $\frac{1}{d_1}(4a_1^2)(4a_2^2)\cdots(4a_{n-1}^2)(4a_n^2)$ if $n$ is even, and $\frac{1}{d_1}(4a_1^2)(4a_2^2)\cdots (4a_{n-1}^2)a_n^2$ if $n$ is odd. Therefore $\det(Y)^2\prod_{i=1}^n a_{i}^{-2}=\frac{4^n}{d_1}$ if $n$ is even and $\frac{4^{n-1}}{d_1}$ if $n$ is odd. Meanwhile, 
$\prod_{i=1}^n\vert a_i\vert^{i-1-ns}=\prod_{i=1}^n \vert a_i^2\vert ^{\frac{i-1-ns}{2}}=\prod_{i=1}^{n-1}(\vert\frac{1}{4d_i\cdot d_{i+1}}\vert^{\frac{i-1-ns}{2}})\cdot\vert\frac{1}{4d_n}\vert^\frac{n-1-ns}{2}=\vert\frac{1}{2}\vert^{\frac{n(n+1)}{2}-ns-1}\cdot\prod_{i=1}^{n-1}(\vert\frac{1}{d_i\cdot d_{i+1}}\vert^{\frac{i-1-ns}{2}})\cdot\vert\frac{1}{d_n}\vert^\frac{n-1-ns}{2}$ if $n$ is even, and $\prod_{i=1}^n\vert a_i\vert^{i-1-ns}=\prod_{i=1}^{n-1}(\vert\frac{1}{4d_i\cdot d_{i+1}}\vert^{\frac{i-1-ns}{2}})\cdot\vert\frac{1}{d_n}\vert^\frac{n-1-ns}{2}=\vert\frac{1}{2}\vert^{\frac{n(n-1)}{2}-ns-1}\cdot \prod_{i=1}^{n-1}(\vert\frac{1}{d_i\cdot d_{i+1}}\vert^{\frac{i-1-ns}{2}})\cdot\vert\frac{1}{d_n}\vert^\frac{n-1-ns}{2}$ if $n$ is odd. Let $\nu(n,s)=\frac{n(n-s)}{2}+\frac{n(n+1)}{2}-ns-1$ if $n$ is even and $\frac{n(n-s)}{2}+\frac{n(n-1)}{2}-ns-1$ if $n$ is odd.

Let $A=A'Z$, which gives $d'_i=d_i/d_1$, ($1\leq i\leq n$), then since $d_1'=1$, $\omega_\pi(4\det(Y')^2\prod_{i=1}^na_i'^{-2})=\omega_\pi(4^{n+1})$ if $n$ is even and $\omega_\pi(4^n)$ if $n$ is odd, denote this number by $c_{\pi}$. 
From the above observations we see that there exists complex numbers $\tau(i,s)$, which are of the form $\tau(i,s)=p_i+sq_i$, $s\in \mathbb{C}$ with $p_i,q_i\in \mathbb{Q}$ depending only on $1\leq i \leq n$, such that
$$C_{\psi}(s,\sigma_\eta)^{-1}=c_\pi\vert\frac{1}{2}\vert^{\nu(n,s)}\gamma(ns, w_{\pi}^2,\psi)^{-1}\int_{A'}j_{\pi,\eta,\dot{w}_\theta,\kappa}(g'(a))\prod_{i=2}^n\vert d_i'\vert^{\tau(i,s)}\prod_{i=2}^2 d^{\times}d'_i$$
where $g'=g(a')=u_1(a')\dot{w}_G a' u_2(a')$ with $a=a'z$, and $a'=\textrm{diag}\{d'_1,\cdots,d'_n\}$.

Now let's prove Proposition 3.4.

\begin{proof}(\textbf{Proof of Proposition 3.4}) If we are given two irreducible supercuspidal representations $\pi_1$ and $\pi_2$ of  $\textrm{GL}_n(F)$ with the same central character $w_{\pi_1}=w_{\pi_2}$, lift them to representations of $M_H(F)$ and denote them by $\sigma_1$ and $\sigma_2$ respectively, then by Proposition 5.9 and the above argument,
$$C_{\psi}(s,\sigma_{1,\eta}\otimes \chi)^{-1}-C_{\psi}(s,\sigma_{2,\eta}\otimes \chi)^{-1}=c_\pi\vert\frac{1}{2}\vert^{\nu(n,s)}\gamma(ns,(w_{\pi}\chi^n)^2,\psi)^{-1}D_{\chi}(s)$$
where 
$$D_{\chi}(s)=\int_{A'} (j_{\pi_1\otimes\chi, \eta,\dot{w}_\theta,\kappa}(g(a'))-j_{\pi_2\otimes\chi,\eta,\dot{w}_\theta,\kappa}(g(a')))\prod_{i=2}^n\vert d'_i\vert^{\tau(i,s)}\prod_{i=2}^nd^{\times}d'_i$$

Pick $f_i\in M(\pi_i)$ such that $W^{f_i}(e)=1$, for $i=1,2$, and such that for $g=-\frac{1}{2}\dot{w}_G\leftidx{^t}Y^{-1}=g(a)=u_1(a)\dot{w}_G a u_2(a).$ By Proposition 5.10,
$$j_{\pi_i,\eta,\dot{w}_\theta,\kappa}(g(a),f_i)=\eta(a(g))^{-1}\vert\det(g)\vert^{\frac{s}{2}}B^G_{\varphi}(g(a), f_i).$$

For convenience let $J_{\pi_i,\eta,\dot{w}_\theta,,\kappa}(g,f_i)=\eta(a(g))\vert\det(g)\vert^{-\frac{s}{2}}\cdot j_{\pi_i,\eta,\dot{w}_\theta,,\kappa}(g,f_i)$. We may also assume that $\kappa $ is sufficiently large so that Proposition 6.17 holds for both $f_1$ and $f_2$ with the same auxiliary function $f_0$. Then apply Proposition 6.17 (2)(a), we have
$$J_{\pi_1,\eta,\dot{w}_\theta,\kappa}(g(a'))-J_{\pi_2,\eta,\dot{w}_\theta,\kappa}(g(a'))=B^G_{\varphi}(g(a'),f_1)-B^G_{\varphi}(g(a'), f_2)$$
$$=B^G_{\varphi}(g(a'),f_{1,1,e})-B^G_{\varphi}(g(a'),f_{2,1,e})+\sum_{1\leq d_B(w',e)}(B^G_{\varphi}(g(a'),f_{1,1,w'})-B^G_{\varphi}(g(a'),f_{2,1,w'}))$$

Now since both $B^G_{\varphi}(g(a'),f_{1,1,e})$ and $B^G_{\varphi}(g(a'),f_{2,1,e})$ depend only on the auxiliary function $f_0$, the central character $\omega_{\pi}=\omega_{\pi_1}=\omega_{\pi_2}$, and $\eta$, we see that 
$$B^G_{\varphi}(g(a'),f_{1,1,e})-B^G_{\varphi}(g(a'),f_{2,1,e})=0.$$ So we are left with
$$J_{\pi_1,\eta,\dot{w}_\theta,\kappa}(g(a'))-J_{\pi_2,\eta,\dot{w}_\theta,\kappa}(g(a'))=\sum_{1\leq d_B(w',e)}(B^G_{\varphi}(g(a'),f_{1,1,w'})-B^G_{\varphi}(g(a'),f_{2,1,w'}))$$

Meanwhile, notice that 
$j_{\pi\otimes\chi, \eta,\dot{w}_\theta,\kappa}(g)=\chi(\det(g))j_{\pi,\eta,\dot{w}_\theta,\kappa}(g)$.
So we have
$$j_{\pi_1\otimes\chi, \eta,\dot{w}_\theta,\kappa}(g(a'))-j_{\pi_2 \otimes\chi,\eta,\dot{w}_\theta,\kappa}(g(a'))$$$$=\chi(\det(a'))(j_{\pi_1,\eta,\dot{w}_\theta,\kappa}(g(a'))-j_{\pi_2,\eta,\dot{w}_\theta,\kappa}(g(a'))).$$
Moreover, since $\det(g')=\det(a')=\frac{d_1\cdots d_n}{d_1^n}$, and as we saw before both $d_1\cdots d_n$ and $d_1$ are in $(F^\times)^2$, so $\det(g')\in (F^\times)^2$. Recall that at the end of section 5.1, we have $M_{H_D}=\{(g,a)\in M_H:\det(g)a(g)^2=1\}^\circ$, there is a unique $a(g)\in F^\times$ such that $\det(g)a(g)^2=1$, denote it by $\det(g)^{-\frac{1}{2}}$. Then $\eta(a(g'))=\eta(\det(g')^{-\frac{1}{2}})=\eta(\det(a')^{-\frac{1}{2}})$.
Now put everything together we obtain that 
$$D_{\chi}(s)=\int_{A'} (\sum_{1\leq d_B(w',e)}(B^G_{\varphi}(g(a'),f_{1,1,w'})-B^G_{\varphi}(g(a'),f_{2,1,w'})))\chi(\det(a'))$$$$\cdot \eta(\det(a')^{-\frac{1}{2}})^{-1}\vert\det(a')\vert^{\frac{s}{2}}\prod_{i=2}^n\vert d'_i\vert^{\tau(i,s)}\prod_{i=2}^nd^{\times}d'_i$$
$$= \sum_{1\leq d_B(w.e)}\int_{A^{w'}_{\dot{w}_G}}(\int_{A'_{w'}}(B^G_{\varphi}(g(bc'),f_{1,1,w'})-B^G_{\varphi}(g(bc'),f_{2,1,w'}))\prod_{i=2}^n\vert c'_i\vert^{\tau(i,s)}\chi(\det(c'))$$$$\cdot\eta(\det(c')^{-\frac{1}{2}})^{-1}\vert \det(c')\vert^{\frac{s}{2}}dc')\chi(\det(b))\eta(\det(b)^{-\frac{1}{2}})^{-1}\vert\det(b)\vert^{\frac{s}{2}}\prod_{i=2}^n\vert b_i\vert^{\tau(i,s)}db.$$
where $a=diag(d_1,\cdots,d_n)=bc=bc'z$ gives the corresponding entries $b_i$ of $b$ and $c'_i$ of $c'$ for $1\leq i \leq n$, and the measure $db$ and $dc'$ on $A^{w'}_{\dot{w}_G}$ and $A_{w'}$ respectively.

Notice that inside the inner integral
the function $$(B^G_{\varphi}(g(bc'),f_{1,1,w'})-B^G_{\varphi}(g(bc'),f_{2,1,w'}))\prod_{i=2}^n\vert c'_i\vert^{\tau(i,s)}$$ is uniformly smooth as a function of $c'\in A_{w'}$ for each fixed $b\in A_{\dot{w}_G}^{w'}$, since both
$B^G_{\varphi}(g(bc'),f_{1,1,w'})$ and $B^G_{\varphi}(g(bc'),f_{2,1,w'})$ are by Proposition 6.17.

Therefore if we take $\chi$ to be sufficiently ramified, we see that the inner integral
$$\int_{A'_{w'}}(B^G_{\varphi}(g(bc'),f_{1,1,w'})-B^G_{\varphi}(g(bc'),f_{2,1,w'}))\prod_{i=2}^n\vert c'_i\vert^{\tau(i,s)}\chi(\det(c'))$$$$\cdot\eta(\det(c')^{-\frac{1}{2}})^{-1}\vert\det(c')\vert^{\frac{s}{2}}dc'=0$$
So we obtain that $D_{\chi}(s)=0$, and therefore 
$$C_{\psi}(s, \sigma_{1,\eta}\otimes \chi)=C_{\psi}(s, \sigma_{2,\eta}\otimes \chi).$$
\end{proof}

\section{ACKNOWLEDGEMENTS}
I would first like to sincerely thank my advisor, Dr. Freydoon Shahidi, for recommending this problem to me, for many of his deep insights and brilliant ideas in mathematics which greatly helped me to increase my research abilities through my Ph.D career, and for all the enlightening discussions with him that give me long-term visions toward my future research.

I would like to acknowledge James Cogdell, Mahdi Asgari, Chung Pang Mok, and David Goldberg for several helpful discussions on various related topics. I would also like to thank my friend Daniel Shankman for many years of mutual learning, and for his willingness and patience to carefully check several technical parts of my proof. 

My research was supported by the NSF grant DMS-1500759.
\section{REFERENCES}

[1] M. Asgari, "Local L-function for Split-spinor Groups", Cand J. Math Vol.54(4), 2002, pp. 673-693.

[2] M. Asgari, F. Shahidi, "Generic Transfer for General Spin Groups", Duke Mathematical Journal, Vol. 132, No.1, 2006, pp. 137–190.

[3] I.N. Bernstein, A.V. Zelevinsky, "Induced representations of reductive p-adic groups(I)", Annales Scientifiques de l'École Normale Supérieure, Série 4, 10 (4): pp. 441–472, 1977. ISSN 0012-9593, MR 0579172.

[4] C.J. Bushnell, G. Henniart, "The Local Langlands Conjecture for $GL(2)$", Grundlehren
der Math. Wiss. 335, Springer-Verlag, 2006.

[5] W. Casselman, "Introduction to the theory of admissible representations of p-adic reductive groups", preprint, 1995.

[6] J.W. Cogdell, I.I. Piatetski-Shapiro, F. Shahidi, "Partial Bessel functions for Quasi-split Groups." In "automorphic Representations, L-functions and Applications: Progress and Prospects." Berlin: Walter de Gruyter, 2005, pp. 95-128.

[7] J.W. Cogdell, F. Shahidi, T.-L. Tsai, "Local Langlands Correspondence for $GL_n$, and the Exterior and Symmetric Square $\epsilon$-factors", Duke Mathematical Journal, Vol 166, No. 11, 2017, pp. 2053-2132.

[8] S. Gelbart, H. Jacquet, "A relation between automorphic representations of GL(2) and GL(3)", Ann. Sci.
Ecole Norm. Sup. (4) 11 (1978), pp. 471–542.

[9] P. Deligne, "Les constantes des equations fonctionnelles des fonctions L" in Modular Functions of One Variable, II(Antwerp,1972) Lecture Notes in Math. 349, Springer, Berlin, 1973, pp. 501-597. MR 0349635.

[10] M. Harris, R. Taylor, "The Geometry and Cohomology of Some Simple Shimura Varieties". Annals of Mathematics Studies, No. 151, Princeton University Press, 2001.

[11] G. Henniart, "Correspondance de Langlands et fonctions L des carr\'es ext\'erieur et sym\'etrique", Int. Math. Res. Not. IMRN 2010, no. 4, pp. 633-673. MR 2595008.

[12] G. Henniart, "Une preuve simple des conjectures de Langlands pour $\textrm{GL}(n)$ sur un corps p-adique", Invent. Math. 139(2000), pp. 439-455. MR 1738446.

[13] R.P. Langlands, "On the classification of irreducible representations of real algebraic groups", in 'Representation Theory and Harmonic Analysis on Semisimple Lie Groups' (Editors: P.J. Sally, Jr. and D.A. Vogan), Mathematical Surveys and Monographs, AMS, Vol.31, 1989, pp. 101-170.

[14] G. Laumon, M. Rapoport, U. Stuhler, “D-elliptic sheaves and the Langlands
correspondence”, Invent. Math. 113:2 (1993), pp. 217–338. MR 96e:11077.

[15] P. Scholze, "The local Langlands correspondence for $GL_n$
over p-adic fields". Invent. Math. 192.3 (2013), pp. 663–715.

[16] F. Shahidi, "A proof of Langlands's conjecture on Plancharel Measures, Complementary series for p-adic Groups", Annals of Mathematics 132, No.2, 1990, pp. 273–330. MR 91m:11095.

[17] F. Shahidi, "Eisenstein Series and Automorphic L-Functions", American Mathematical Society, Colloquium Publications, volume 58 (2010).

[18] F. Shahidi, "Local coefficients as Artin factors for real groups", Duke Math. J. 52(1985), pp. 973-1007. MR 0816396. DOI 10. 1215/S0012-7094-85-0525204.(2064, 2070)

[19] F. Shahidi, "Local Coefficients as Mellin Transforms of Bessel Functions Towards a General Stability", IMRN 2002, No. 39, pp. 2075-2119.

[20] F. Shahidi, "On non-vanishing of twisted symmetric and exterior square L-functions for GL(n)", Olga TausskyTodd: in memoriam. Pacific J. Math. 1997, Special Issue, pp. 311–322.

[21] A. J. Silberger, Ernst-Wilhelm Zink, "Langlands classification for L-parameters", J. Algebra 511
(2018), pp. 299–357.

[22] R. Sundaravaradhan, "Some Structural Results for the Stability of Root Numbers", Int. Math. Res. Not.
IMRN 2007, No. 2, pp. 1–22.

[23] I. Vogt, "The local Langlands correspondence for $\textrm{GL}_n$ over a p-adic field", Available at: http://web.stanford.edu/~vogti/LLC.pdf, pp. 1-18.

\end{document}